\documentclass[12pt,reqno]{amsart}

\usepackage{amssymb, amsfonts, amsthm, amsmath}
\usepackage{mathtools}
\usepackage[hidelinks]{hyperref} 
\usepackage{geometry}            
\usepackage{mathrsfs}            
\usepackage{graphicx}

\usepackage{comment}
\usepackage[]{todonotes}

\usepackage{tikz-cd}

\usepackage{booktabs}
\usepackage{tabularx}
\usepackage{dynkin-diagrams}

\newcommand{\RR}{\mathbb{R}}
\newcommand{\CC}{\mathbb{C}}

\newcommand{\HH}{\mathbb{H}}
\newcommand{\PP}{\mathbb{P}}

\newcommand{\op}{\mathcal{O}\mathfrak{p}}

\DeclareMathOperator{\Osc}{Osc}
\DeclareMathOperator{\Ep}{Ep}

\DeclareMathOperator{\rk}{rk}
\DeclareMathOperator{\Dev}{Dev}
\DeclareMathOperator{\II}{II}
\DeclareMathOperator{\Hess}{Hess}
\DeclareMathOperator{\Inj}{Inj}

\DeclarePairedDelimiter{\norm}{\lVert}{\rVert}

\newtheorem{thm}{Theorem}[section]
\newtheorem{lemma}[thm]{Lemma}
\newtheorem{prop}[thm]{Proposition}
\newtheorem{coro}[thm]{Corollary}
\newtheorem{ques}[thm]{Question}

\theoremstyle{remark}
\newtheorem{rmk}[thm]{Remark}
\AtBeginEnvironment{rmk}{%
  \pushQED{\qed}%
}
\AtEndEnvironment{rmk}{\popQED}

\theoremstyle{definition}
\newtheorem{defi}[thm]{Definition}

\newtheorem{ex}[thm]{Example}
\AtBeginEnvironment{ex}{%
  \pushQED{\qed}%
}
\AtEndEnvironment{ex}{\popQED}

\title{Epstein-Poincar\'e surfaces for $G-$opers}
\author{Joaqu\'in Lema}
\address{Department of Mathematics, Boston College Chestnut Hill, MA 02467}
\email{lemajo@bc.edu}
\date{}

\begin{document}
\maketitle
\begin{abstract}
 Given a complex, simple Lie group $G$ of adjoint type, we introduce the notion of an Epstein-Poincar\'e surface associated to a $G$-oper.
 These surfaces generalize Epstein's classical construction for $G=PGL_2 (\mathbb{C})$.
 As an application, we provide a criterion ensuring that the holonomy of the oper is $\Delta$-Anosov.
 Finally, we discuss how the developing map of the oper interacts with domains of discontinuity of the holonomy (whenever Anosov) and the transversality properties it satisfies.
 Along the way, we provide a quick review of opers that we hope serves as a self-contained introduction.
\end{abstract}

\section{Introduction}

\subsection{The Ahlfors-Weill Theorem}

A \emph{complex projective structure} on a Riemann surface $X$ is a $(\mathbb{P}^1, PGL_2 (\mathbb{C}))$ structure (in the language of \cite[Chapter 3]{thurston2022geometry}) with an atlas compatible with the complex structure of $X$.
By a developing map construction, this information is equivalent to a pair $(\mathcal{D},\rho)$, where $\mathcal{D}: \tilde{X} \rightarrow \mathbb{P}^1$ is a locally injective holomorphic map equivariant with respect to the deck group action and a representation $\rho: \pi_1 (X) \rightarrow PGL_2 (\mathbb{C})$ called the \emph{holonomy} of the complex projective structure.

Thinking of a developing map as a holomorphic curve into $\mathbb{P}^1$, one can check that the pair $(\mathcal{D},\rho)$ is determined (up to isomorphisms) by a special type of complex ODE, classically referred to as the Schwarzian differential equation.
This perspective allows one to identify the space $\mathcal{P}(X)$ of isomorphism classes of complex projective structures on $X$ with the space $H^0 (X,K^2)$ of holomorphic quadratic differentials on $X$ (see \cite{dumas2009complex}).

If $X$ is a hyperbolic Riemann surface, i.e., $X$ is biholomorphic to $\mathbb{H}^2/\rho_F (\pi_1 (X))$ for some discrete and faithful representation $\rho_F : \pi_1 (X) \rightarrow PSL_2 (\mathbb{R})$, the pair $(\mathcal{D}_F,\rho_F)$ defines what we call the Fuchsian projective structure, where $\mathcal{D}_F$ is the inclusion of $\tilde{X} \cong \mathbb{H}^2 \subset \mathbb{P}^1$.
Thinking of $PGL_2 (\mathbb{C})$ as the isometry group of $\mathbb{H}^3$, the representation $\rho_F$ defines a hyperbolic structure on $S \times \mathbb{R}$ that is \textit{convex cocompact} (see \cite[Chapter 8]{thurston2022geometry}).
A representation $\rho: \pi_1 (X) \rightarrow PGL_2 (\mathbb{C})$ giving rise to such a convex cocompact hyperbolic structure is called \emph{quasi-Fuchsian}, or just convex cocompact. 
These form an open set of the space of all discrete and faithful representations of $\pi_1 (X)$ into $PGL_2 (\mathbb{C})$ and are characterized by the property that their limit set $\Lambda \subset \mathbb{P}^1$ is a Jordan curve that splits $\mathbb{P}^1$ into two (topological) disks on which the representation acts properly discontinuously.

In particular, the holonomy of a complex projective structure close to the Fuchsian structure will induce a quasi-Fuchsian representation.
Ahlfors and Weill gave an effective bound on how open the convex cocompact property really is.

\begin{thm}[\cite{ahlfors1962uniqueness}]
\label{thm:ahlforsweill}
Let $X$ be a compact hyperbolic Riemann surface, and identify $\mathcal{P}(X)$ with $H^0 (X,K^2)$ using the Fuchsian projective structure as a basepoint.
Then, the complex projective structure associated to an $\alpha \in H^0 (X,K^2)$ satisfying $\norm{\alpha} < \frac{1}{2}$ will have quasi-Fuchsian holonomy. 
Here $\norm{\cdot}$ denotes the norm on $K^2$ defined by the hyperbolic metric of constant curvature $-1$.
\end{thm}

In particular, the developing map of the projective structure defined by $\alpha$ will be an embedding onto one of the domains of discontinuity.
It is well-known that this result is in fact sharp, as there are complex projective structures defined by an $\alpha$ with hyperbolic norm $\frac{1}{2}$ which are not quasi-Fuchsian.

Ahlfors and Weill's original proof reduces to a clever complex analysis computation.
Later, Epstein in \cite{epstein2024envelopes} was able to generalize Theorem \ref{thm:ahlforsweill} by using the geometry of hyperbolic three-space.
The main ingredient he introduced is the notion of what we now call \emph{Epstein surfaces}, which have recently proved useful in and around hyperbolic geometry (see for instance \cite{krasnov2008renormalized}, \cite{bridgeman2025universal}).

An Epstein surface is a map $\Ep: \tilde{X} \rightarrow \mathbb{H}^3$ built from a conformal metric at infinity and the developing map of the projective structure.
Furthermore, the surface is equivariant with respect to the holonomy.
Whenever the input is the hyperbolic metric of constant negative curvature $-1$ at infinity, we refer to $\Ep$ as the \emph{Epstein-Poincar\'e} (EP) surface associated to the conformal structure.

Epstein proved that under the hypothesis of Theorem \ref{thm:ahlforsweill}, the EP-surfaces are embedded, and have principal curvatures inside the $(-1,1)$ interval.
This key property implies that:

\begin{description}
   \item[Epstein's approach] $\Ep$ is a quasi-isometric embedding.
   \item[Foliation approach] The surfaces equidistant to the Epstein surface foliate $\mathbb{H}^3$.
\end{description}

Either of these two facts implies Theorem \ref{thm:ahlforsweill}.

The goal of this article is to propose a higher rank generalization of EP-surfaces to the context of $G-$opers, which one can think of as a higher rank generalization of complex projective structures, as we now describe.

\subsection{Generalized Epstein-Poincar\'e surfaces}

It is classically known that if one composes the developing map of a complex projective structure, with the Veronese embedding $V: \mathbb{P}^1 \rightarrow \mathbb{P}^n$, the resulting map from $\tilde{X}$ to $\mathbb{P}^n$ is described by a particular type of ODE of order $n+1$ (see for instance \cite{hejhal1975monodromy}, \cite{wilczynski1906projective}).
One can perturb this ODE using $k-$differentials ($2\leq k \leq n+1$), to obtain maps from $\tilde{X}$ to $\mathbb{P}^n$ equivariant with respect to a representation $\rho : \pi_1 (X) \rightarrow PGL_{n+1} (\mathbb{C})$.

Beilinson–Drinfeld greatly generalized these classical constructions under the general framework of \emph{opers} (see \cite{beilinson2005opers}).
Given $G$ a complex semisimple Lie group of adjoint type with full flag manifold $\mathcal{B}$, a $G-$oper on the Riemann surface $X$ is a pair $(\mathcal{D},\rho)$, where $\mathcal{D}: \tilde{X} \rightarrow \mathcal{B}$ is a holomorphic immersion, equivariant with respect to a holonomy representation $\rho: \pi_1 (X) \rightarrow G$ and with the property that $\mathcal{D}$ is everywhere tangent to the set of principal directions $\mathcal{O} \subset T \mathcal{B}$.

The set $\mathcal{O}$ is defined as the set of tangent vectors to principal embeddings $\mathcal{P} : \mathbb{P}^1 \rightarrow \mathcal{B}$ (see Section \ref{subsection:triplets}).
When $G = PGL_{n+1} (\mathbb{C})$, principal embeddings are defined as the flags osculating the image of $\mathbb{P}^1$ under the Veronese embedding to $\mathbb{P}^n$ (see Example \ref{ex:veronese}).

Requiring the curve to be tangent to $\mathcal{O}$ imposes strong conditions on the curve.
Beilinson-Drinfeld proved that if we fix a principal embedding of $\phi :PGL_2 (\mathbb{C}) \hookrightarrow G$, then for any pair of opers $\mathcal{D}_1, \mathcal{D}_2$ there is a canonical holomorphic map $\Osc (\mathcal{D}_1,\mathcal{D}_2) : \tilde{X} \rightarrow G$ interpolating $\mathcal{D}_1$ and $\mathcal{D}_2$ in the sense that $\Osc (\mathcal{D}_1,\mathcal{D}_2) (z) . \mathcal{D}_1 (z) = \mathcal{D}_2 (z)$, plus some properties we describe later (see Section \ref{section:ds}). 
This generalizes the notion of osculating M\"obius maps to the context of $G-$opers (see \cite{dumas2009complex}, \cite{thurston1986zippers}).

With these osculating maps in hand, one can prove that the space of isomorphism classes of opers $\op_X (G)$ is an affine space modeled on the Hitchin base $\bigoplus_{i=1}^l H^0 (X, K^{m_i+1})$, for $m_i$ the exponents of $G$. 
However, when $\rk  (G) >1$ the identification depends on some Lie theoretic choices that we describe in detail in Section \ref{sec:introopers}.
We introduce a specific parametrization $\op_X (G) \cong \bigoplus_{i=1}^l H^0 (X,K^{m_i+1})$ that we call \emph{normalized} which fixes some of these choices.

Picking a Fuchsian projective structure $(\mathcal{D}_F, \rho_F)$, and composing it with a principal embedding $\mathcal{P}: \mathbb{P}^1 \rightarrow \mathcal{B}$ produces a $G-$oper which we will also call \textit{Fuchsian}.
The holonomy of this oper preserves a totally geodesic copy of the hyperbolic plane inside the symmetric space $\mathbb{X}$ of $G$ (thought of as a real Lie group), that we can parametrize as $\Ep^0 : \tilde{X} \rightarrow \mathbb{X}$ in a way that is compatible with $\mathcal{P} \circ D_F$ (as described in Section \ref{sec:epstein}).

The osculating transformations let us now push this surface to any other $G-$oper:

\begin{defi}
\label{defi:naiveepstein}
 Fix a principal embedding $\phi: PGL_2 (\mathbb{C}) \rightarrow G$. 
 Given $(\mathcal{D},\rho)$ a $G-$oper over $X$, we define its \textit{associated Epstein-Poincar\'e surface} to be $\Ep : \tilde{X} \rightarrow \mathbb{X}$ defined as:
   $$\Ep (z) = \Osc (\mathcal{P}_\phi \circ \mathcal{D}_F, \mathcal{D}) (z) . \Ep^0 (z),$$
 where $\mathcal{P}_\phi$ is the principal embedding defined by $\phi$, and $(\mathcal{D}_F, \rho_F)$ is a Fuchsian oper as before (see Figure \ref{fig:epstein}).
\end{defi}

The surface we obtain is independent of the choice of the reference principal embedding.
When $G = PGL_2 (\mathbb{C})$, this recovers the notion of EP-surface (see Proposition \ref{prop:epsteinisepstein}).

Just as in the $PGL_2 (\mathbb{C})$ situation, EP maps can be naturally ``thickened'' to get a one-parameter family of equidistant surfaces in $\mathbb{X}$, generalizing the $\mathbb{H}^3$ picture to higher rank (see \cite[Section 5.8]{bridgeman2024epstein}).
This fact will not play a role in our applications, thus we defer it to Appendix \ref{sec:equidistant}.

\begin{rmk}
All of this makes sense in a ``universal'' setting (i.e., without asking for a surface group representation) if one simply studies holomorphic maps from the disk to $\mathcal{B}$ tangent to $\mathcal{O}$.
\end{rmk}

\subsection{Main results and Open Questions}

Just as Fuchsian projective structures belong to the family of quasi-Fuchsian representations, the holonomy of a Fuchsian oper is an example of an \textit{Anosov representation}.
Initially introduced by Labourie \cite{labourie2006} for $G = SL_n (\mathbb{R})$, and later generalized by Guichard-Wienhard \cite{guichard2012anosov} for general semisimple Lie groups, the Anosov condition is nowadays understood as a higher rank generalization of convex cocompact representations in rank one.

To be Anosov is an open condition in the character variety of representations to $G$.
However, verifying the condition for an explicit (non-Fuchsian) representation, or deducing the Anosov condition from infinitesimal data, has proven challenging.
See for instance \cite{bronstein2025anosov}, \cite{filip2021uniformization}, \cite{riestenberg2024quantified}, \cite{zhang2025non} for works in this direction.

Crucially for us, Davalo \cite{davalo2025nearly} produced a criterion to verify Anosovness using what he called $\tau\text{-}$nearly geodesic surfaces. 
Given $\tau \subset \partial \mathbb{X}$ a stratum of the visual boundary (i.e., the $G-$orbit of a point), the $\tau\text{-}$nearly geodesic property is a convexity condition imposed on the surface that forbids it from entering inside any tangent horoball based at a point in $\tau$.
This generalizes the notion of surfaces in $\mathbb{H}^3$ with principal curvatures in $(-1,1)$.
Given $\Theta \subset \Delta$ a Weyl-orbit of simple roots, Davalo describes a stratum $\tau_\Theta$, such that if we are able to find a $\tau_\Theta-$nearly geodesic surface equivariant with respect to a given representation to $G$, then the representation will be $\Theta-$Anosov.

Having our EP-surfaces in hand and Davalo's criteria, we proceed to find opers with Anosov holonomy.

\begin{defi}
 Given $(D,\rho)$ a $G-$oper associated to the hyperbolic Riemann surface $X$, we say that:
   \begin{itemize}
      \item $(D,\rho)$ is \textit{quasi-Hitchin} if its holonomy is $\Delta-$Anosov and can be connected to the Fuchsian oper by a path of opers with $\Delta-$Anosov holonomy.
      \item $(D,\rho)$ is \textit{EP-witnessed} if the EP-surface associated to it is $\tau_\Theta-$nearly geodesic, for $\Theta$ any Weyl orbit of simple roots.
   \end{itemize}
 We denote the isomorphism classes of quasi-Hitchin opers as $\mathcal{QH} (X)$, and by $\mathcal{E} (X)$ the EP-witnessed classes.
\end{defi}

Since being $\tau_\Theta-$nearly geodesic is an open condition, $\mathcal{E}(X)$ is open and its connected component containing the Fuchsian oper is contained in $\mathcal{QH}(X)$ by Davalo's work.
Explicit control on the second fundamental form of EP-surfaces yields the following neighborhood of EP-whitnessed opers:

\begin{thm}
\label{thm:maingeneral}
Let $\mathfrak{g}$ be a simple Lie algebra and consider $\mathcal{U}_X (G)$ the set of differentials $\vec{\alpha}\in \bigoplus_{i=1}^l H^0 (X,K^{m_i+1})$ satisfying the inequality:

\[
\frac{\left( \sum_{i=1}^l m_i \norm{\alpha_i}^2 + \norm {\nabla \alpha_i}^2 \right)^{\frac{1}{2}} + \frac{l c_{\mathfrak{g}}}{2} \sum_{i=1}^l \norm{\alpha_i}^2}{(1 - \frac{\norm{\alpha_1}}{2})^2 + \frac{1}{4} \sum_{i=2}^l \norm{\alpha_i}^2} \leq 2 \sqrt{\kappa (e,f)} \cos^2 (2\phi_{\Theta_{S}}), 
\]

pointwise, where:
\begin{itemize}
   \item $(e,h,f)$ is a principal $\mathfrak{sl}_2$ triplet.
   \item $\kappa$ is the Killing form.
   \item $m_i$ for $i=1,\ldots,l$ are the exponents of $\mathfrak{g}$.
   \item $\phi_{\Theta_{S}}$ is the minimal angle between $h$ and the walls defined by the Weyl orbit of the short root.
\end{itemize}

Finally, the constant $c_{\mathfrak{g}}$ depends only on the Lie algebra, and the norms are taken with respect to the hyperbolic metric on $X$ with constant curvature $-2$.\footnote{See Section \ref{sec:notation} below for the conventions we adopt in this article.}
Then under a normalized parametrization $\op_X (G) \cong \bigoplus_{i=1}^l H^0 (X,K^{m_i+1})$ , any oper associated with $\vec{\alpha} \in \mathcal{U}_X (G)$ is quasi-Hitchin.
\end{thm}

In fact, our methods provide a larger neighborhood described by a more complicated expression (see Theorem \ref{thm:megatheorem}).
We do not expect any of these neighborhoods to be sharp, since we use some rough estimates in the computation in favor of a general result for any $G-$oper.

A $G-$oper is \emph{cyclic} if it is defined by the differential $\alpha \in H^0 (X,K^{m_l+1})$ alone.
In this context, we can get sharper estimates:

\begin{thm}
\label{thm:maincyclic}
Given $\mathfrak{g} \neq \mathfrak{sl}_2 (\mathbb{C})$ a simple Lie algebra, and let $\mathcal{U}_X (G)$ the set of differentials $\alpha \in H^0 (X,K^{m_l+1})$ satisfying:
\[
\frac{ \sqrt{\norm{\alpha}^4 c_l + 2m_l \norm{\alpha}^2 + 2 \frac{\norm{\nabla \alpha}^2}{1 + \frac{1}{4} \norm {\alpha}^2}}}{\left( 2\cot (\phi_{\Theta_{S}}) - \norm{\alpha}\right)^2} \leq \sqrt{\kappa (e,f)} \sin (\phi_{\Theta_{S}}),
\]
where $c_l = \frac{m_l^2 \kappa (h,h)}{\sin^2 (\phi_\theta)}$ for $\theta$ is the highest root of $\mathfrak{g}$, and $\phi_\theta$ is the angle between $h$ and $\ker \theta$.
Then the cyclic oper associated to $\alpha$ under a normalized parametrization is quasi-Hitchin.
\end{thm}

In this situation, we cannot obtain a better bound using the methods we have.
Notice that, unlike the Ahlfors-Weill Theorem, the open sets of Anosov opers we get around the identity do depend on the conformal structure we equip the surface with.
One can use standard Cauchy estimates to produce criteria depending only on the norm of the differential and the injectivity radius of the underlying Riemann surface:

\begin{coro}
\label{coro:cauchy}
 Consider $\alpha \in H^0 (X,K^{m_l+1})$ satisfying:
      $$\norm {\alpha}^2 \leq \frac{-(m_l + C(X)) + \sqrt{(m_l + C(X))^2 + 16 c_l \kappa (e,f) \cos^4 (\phi_\theta)}}{c_l},$$
 where $C(X) = \underset{\alpha \in H^0 (X,K^{m_l+1})\setminus \{0\}}{\sup}  \frac{\norm{\nabla \alpha}_{\infty}}{\norm{\alpha}_\infty}$. 
 Then the associated cyclic oper via a normalized parametrization is quasi-Hitchin.
\end{coro}

Since $C(X) \approx \frac{1}{\Inj (X)}$ when $\Inj (X)$ is small (see Appendix \ref{subsec:computations2}), the bound from above gets worse as we escape compact sets of $\mathcal{T}(S)$.

\begin{ques}
Let $\bigoplus_{i=1}^l \mathcal{K}^{m_i} \rightarrow \mathcal{T}(S)$ be the bundle over Teichm\"uller space whose fiber over $X$ is the vector space $\bigoplus_{i=1}^l H^0 (X,K^{m_i + 1})$.
Is there a uniform neighborhood around the zero section for which the associated opers are $\Delta-$Anosov?
\end{ques}

Holonomies of opers seem to occupy a special place in the character variety of representations to complex semisimple Lie groups.
In \cite{sanders2018pre}, Sanders proved that the map mapping an oper to its holonomy is locally injective, although the map is never globally injective (see for instance \cite{goldman1987projective}). 
However, if we fix a point in Teichmuller space, then the holonomy map is injective.
We believe this fact is well known to experts. 
We include a concise proof of this (Theorem \ref{thm:injectivity}) since we could not find it in the literature except for $PGL_n (\mathbb{C})$ (see \cite{wentworth2016higgs}).
Recently, El-Emam and Sagman parametrized a neighborhood of the character variety of rank two complex semisimple Lie groups containing the Fuchsian locus and found out that opers play a ubiquitous role in their parametrization (see \cite[Theorem D]{emam2025complex}).

An advantage of our techniques is that they provide information about the domains of discontinuity of the holonomy of an oper.
Davalo showed that whenever we have a $\rho-$equivariant $\tau_{\Theta}\text{-}$nearly geodesic immersion, we get a preferred cocompact domain of discontinuity $\Omega_\rho$ in the partial flag manifold $\mathcal{P}$ isomorphic to the stratum $\tau_{\Theta}$ (see Section \ref{subsec:domains}).
This domain comes from a balanced ideal construction as in \cite{kapovich2017dynamics}.
Moreover, the quotient $\Omega_\rho /\rho (\pi_1 (S))$ is a smooth fiber bundle over $S$ (which is an instance of more general results, see \cite{alessandrini2025fiber}).
We prove:

\begin{thm}
\label{thm:domaindiscontinuity}
 Let $\pi: \mathcal{B} \rightarrow \mathcal{P}$ be the projection of the full flag manifold to the partial flag manifold $\mathcal{P}$ defined above, and let $(\mathcal{D},\rho)$ be an EP-whitnessed oper. 
 Then $\pi \circ \mathcal{D} (\tilde{X})$ is contained in $\Omega_\rho$ and descends to a smooth section of the fiber bundle $\Omega_\rho /\rho (\pi_1 (S))$.
\end{thm}

Clearly, the quotient of the curve $\pi \circ \mathcal{D}$ is holomorphic in the complex manifold $\Omega_\rho/ \rho (\pi_1 (S))$.
It is worth noticing that $\Omega_\rho/ \rho (\pi_1 (S))$ is never a holomorphic fiber bundle, as one can see from \cite[Theorem D]{dumas2020geometry}.

The following seems plausible:

\begin{ques}
\label{ques:extension}
 Is the developing map of a quasi-Hitchin oper always contained in a domain of discontinuity of the full flag manifold $\mathcal{B}$? 
 If we identify $\tilde{X} \cong \mathbb{H}^2$, does the developing map $\mathcal{D}$ extend to a continuous map $\overline{\mathcal{D}} : \overline{\mathbb{H}^2} \rightarrow \mathcal{B}$, where the restriction to $\partial \mathbb{H}^2$ coincides with the Anosov limit map?
\end{ques}

We hope that a positive answer to the previous question may lead to holomorphic rigidity phenomena for their domains of discontinuity.
It is not hard to see that Question \ref{ques:extension} has a positive answer in a neighborhood of the $G-$Fuchsian oper (see Lemma \ref{lemma:localanswer}).
Coincidentally, when $G = PGL_3 (\mathbb{C})$, the partial flag manifold $\mathcal{P}$ introduced above is simply the full flag manifold, which, together with Theorem \ref{thm:domaindiscontinuity}, allows us to say something stronger:

\begin{thm}
\label{thm:extensionsl3}
 Let $(D,\rho)$ be an EP-whitnessed $PGL_3 (\mathbb{C})-$oper, then $\mathcal{D} (\tilde{X})$ is contained in a cocompact domain of discontinuity of $\rho$, and we can extend $\mathcal{D}$ to a map $\overline{\mathcal{D}}: \overline{\mathbb{H}^2} \rightarrow \mathcal{B}$ such that $\overline{\mathcal{D}}|_{\partial \mathbb{H}^2}$ is the limit map of $\rho$.
\end{thm}

We conclude this article by discussing some transversality properties of $G-$opers.
There is a notion of real oper in $SL_n (\mathbb{R})$ which are known to produce curves in the space of full flags with good transversality properties, which we now call Frenet, see \cite[Section 1.2]{labourie2006}, \cite{labourie2018goldman}, \cite{segal1991geometry}.
It is natural to question whether some of these good properties subsist in the complex setup.

\begin{defi}
Given $(\mathcal{D},\rho)$ a $G-$oper, we say that it is \emph{transverse} if given $z\neq w \in \tilde{X}$, we get that $\mathcal{D}(z)$ is transverse to $\mathcal{D}(w)$.
\end{defi}

Recall that a subset $S \subset \mathcal{B}$ is locally maximally transverse if there exists a neighborhood $S \subset \mathcal{U}$ such that any $\mathfrak{b} \in \mathcal{U}\setminus S$ is not transverse with respect to some element of $S$.
One can verify that $G-$Fuchsian opers are locally maximally transverse (this is a consequence of Proposition \ref{lemma:fuchsiantransverse}).
More generally, we prove:

\begin{thm}
\label{thm:transversality}
 Given $(\mathcal{D},\rho)$ be a $G-$oper and $z_0 \in \tilde{X}$, the set $$\text{NT}_{z_0} = \{ z \in \tilde{X}: D(z)\text{ is not transverse to }D(z_0)\}$$ is isolated.
\end{thm}

Theorem \ref{thm:extensionsl3} implies that:

\begin{coro}
 When $G = PGL_3 (\mathbb{C})$ and the oper $(\mathcal{D},\rho)$ is EP-whitnessed, then the set $\text{NT}_{z_0}$ is finite for any $z_0 \in \tilde{X}$. 
\end{coro}

This motivates the following question:

\begin{ques}
Are quasi-Hitchin opers transverse? Are they locally maximally transverse?
\end{ques}

Notice that without the quasi-Hitchin assumption, this question makes no sense even for $PGL_2 (\mathbb{C})$.

In the real setup, there is a rich study of transverse sets in partial flag manifolds (see for instance \cite{dey2025borel,dey2024restrictions,evans2025transverse}).
In many cases, one can show that any transverse circle has to be locally maximal transverse (see \cite[Theorem 0.19]{evans2025transverse}), which applies for instance to limit sets of Anosov representations.
Opers seem to point towards a rich picture in the complex setup as well.

\subsection{Structure of the article}
Sections \ref{sec:lietheory} and \ref{sec:introopers} are preliminaries and contain no new results, with the exception of Theorem \ref{thm:injectivity}.
A lot of care has been put into making it a self-contained introduction to opers and into bridging the gap between readers familiar with Higgs bundles and those with a Thurstonian perspective.

Section \ref{sec:epstein} introduces the notion of EP-surfaces, Section \ref{sec:regularity} develops a criterion to check whether these surfaces are immersed and regular, and Section \ref{sec:criterion} contains the bulk of the computations leading to Theorems \ref{thm:maingeneral}, \ref{thm:maincyclic}.
Finally, Section \ref{subsec:domains} studies the domain of discontinuity for EP-whitnessed opers, and Section \ref{sec:transversality} finishes with some remarks about transversality properties of opers.

We include a few appendices: the first one describes the foliated EP-surface picture, the second one contains some technical computations to streamline the flow of the main proofs, and the third one includes an explanation of the gauge theoretic interpretation of maps into homogeneous spaces.
The last one collects Lie theoretic constants relevant to our computations.

\subsection{Notation and Conventions}
\label{sec:notation}

We follow standard differential geometric notation as found in \cite{Kobayashi1963Foundations}. 
Throughout the paper, we assume $M$ is a smooth manifold, and we reserve $X$ for Riemann surfaces.

\begin{description}
   \item[Principal Bundles] 
 Given a principal $G$-bundle $P \to M$, and $V$ a vector space in which $G$ acts linearly, we denote the associated bundle by $P_G[V] = P \times_G V$. 
 A connection $A$ on $P$ induces a covariant derivative on $P_G[V]$, denoted simply by $\nabla$ or $d^\nabla$. 
 The curvature of an affine connection $\nabla$ is denoted as $R^\nabla$.
    
    \item[Bundles and Forms] 
 For a smooth vector bundle $E \to M$, we denote the space of smooth $k$-forms with values in $E$ by $\mathcal{A}^k(M, E)$. 
 Whenever $M$ is a complex manifold, we denote the space of complex-valued forms of type $(p,q)$ as $\mathcal{A}^{(p,q)}(M, E)$. 
 If $E$ is holomorphic, $H^0(M, E)$ denotes the space of holomorphic sections.

    \item[Riemann surfaces] 
 A Riemann surface $X$ admitting a compatible metric $H =H(z) dz \overline{dz}$ of constant negative curvature is called hyperbolic. 
 Given $K$ the holomorphic cotangent bundle, we equip each $K^l$, $l \in \mathbb{Z}$ with the hermitian metric $\langle \alpha,\beta\rangle = \frac{\alpha \overline{\beta}}{H^l}$.
 We abuse notation and denote this metric $\langle \cdot,\cdot \rangle$ for every $l$. 
 We denote by $*_l :K^l \rightarrow K^{-l}$ the antilinear isomorphism defined by the property $\alpha *_l \beta = \langle \alpha,\beta\rangle$ for every $\alpha,\beta \in \mathcal{A}^0 (X,K^l)$.
 The metrics $\langle \cdot,\cdot \rangle$ in $K^l$ have an associated Chern connection $\nabla^l$. 
 We normalize the hyperbolic metric so that $R^{\nabla^l} = l H(z)dz \wedge \overline{dz}$.
 This corresponds to $H$ having Gaussian curvature equal to $-2$.
\end{description}

We assume some familiarity with the structure theory of complex semisimple Lie algebras. 
A concrete introduction can be found in \cite{serre2000complex}.

\begin{description}
   \item[Adjoint group] We will denote by $G$ the adjoint group $G= (Aut (\mathfrak{g}))^0$.
 A Lie group $G$ is of adjoint type if it arises from the construction above.

   \item[Root decomposition] A Cartan subalgebra $\mathfrak{h} \subset \mathfrak{g}$ (maximal abelian) induce the root space decomposition $\mathfrak{g} = \mathfrak{h} \oplus \bigoplus_{\alpha \in \Sigma} \mathfrak{g}_\alpha$, where $\Sigma = \Sigma (\mathfrak{h}) \subset \mathfrak{h}^*$ is the set of roots.
 Each root $\alpha$ defines a coroot $h_\alpha \in \mathfrak{h}$ defined via the Killing form $\kappa$ as $\alpha (h) = 2\frac{\kappa (h,h_\alpha)}{\kappa (h_\alpha,h_\alpha)}$ for every $h \in \mathfrak{h}$.
 The real span of $h_\alpha$ for $\alpha \in \Sigma$ defines a real form $\mathfrak{a}$ of $\mathfrak{h}$. The hyperplanes $\ker \alpha$ split $\mathfrak{a}$ into cones called Weyl chambers.
 A choice of Weyl chamber splits $\Sigma$ into positive and negative roots, which induces an order on $\Sigma$.
 The set of minimal positive roots is called the set of simple roots, and we denote it by $\Pi$. 
\end{description}

\subsection*{Acknowledgements} This paper greatly benefited from discussions with Andy Sanders and Max Riestenberg, to both of whom I am grateful.
I also extend my gratitude to Le\'on Carvajales, Colin Davalo, Joaqu\'in Lejtreger, and Andr\'es Sambarino for enriching conversations and comments on earlier versions of this article.
Most importantly, I thank my advisor Martin Bridgeman for his support and guidance.

\section{Lie Theory Preliminaries}
\label{sec:lietheory}

In this text, we will be concerned with complex semisimple Lie algebras $\mathfrak{g}$.
We start this section by reviewing some geometric facts about the full flag manifold $\mathcal{B}$, and the symmetric space $\mathbb{X}$ of $\mathfrak{g}$.
Our approach to symmetric spaces is classical in the Higgs bundles literature (as in \cite{labourie2017cyclic}, \cite{hitchin1992lie}), but we include some details to try to clarify the geometry of the symmetric space.
We wrap up with a review of $\mathfrak{sl}_2$ triplets, which play a crucial role in this article.

\subsection{Full flag manifold}
\label{sec:fullflag}

Recall that when $G = PGL_n (\mathbb{C})$, the maximal (projective) compact homogeneous quotient of $PGL_n (\mathbb{C})$ is the set of full flags $\mathcal{F}_n$ in $\mathbb{C}^n$ whose elements are chains of subspaces $0 \subset F^1 \subset \ldots \subset F^n = \mathbb{C}^n$ satisfying that $\dim F_i = i$.
The tangent space at a point $F \in \mathcal{F}_n$ can be described as: 
$$T_F \mathcal{F}_n = \{ (\phi_i)_{i=1,\ldots,n-1} \in \bigoplus_{i=1}^{n-1} Hom (F^i, \mathbb{C}^n/F^i): \phi_i|_{F^{i-1}} = \pi_i \circ \phi_{i-1}, \forall i\},$$
where $Hom (V, W)$ denotes the set of linear maps between two vector spaces, $F = F^1 \subset \ldots \subset F^{n-1}$ is the flag, and $\pi_i: \mathbb{C}^n \rightarrow \mathbb{C}^n/F^i$ is the natural projection.
There is a canonical $PGL_n (\mathbb{C})-$invariant filtration of the tangent space of $\mathcal{F}_n$ by a family of distributions which in some way characterizes the geometry of the full flag manifold.
Define $T^i_F \mathcal{F}_n$ as:

$$T^i_F \mathcal{F}_n = \{ (\phi_i) \in T_F \mathcal{F}_n : \phi_j (F^j) \subset F^{i+j}/F^j, \forall j \},$$

with the convention that $F^l = \mathbb{C}^n$ for $l \geq n$ (see Figure \ref{fig:distribution}).
Clearly we get that $T^1 \mathcal{F}_n \subset \ldots \subset T^n_F \mathcal{F}_n = T \mathcal{F}_n$.
These distributions define something called a Cartan structure in $\mathcal{F}_n$ (see \cite{vcap2024parabolic} for more information).

\begin{center}
   \begin{figure}[h!]
   \label{fig:distribution}
   \includegraphics{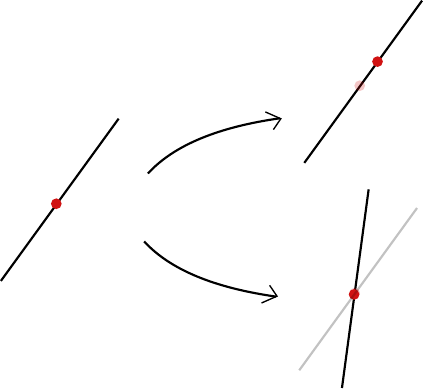}
   \caption{Directions tangent to $T^1 \mathcal{F}_n$ slide $F_i$ into $F_{i+1}$}
   \end{figure}
\end{center}

More generally, if $\mathfrak{g}$ is a complex semisimple Lie algebra, we define its \emph{full flag manifold} to be the set of Borel subalgebras $\mathfrak{b} \subset \mathfrak{g}$.
Recall that a Borel subalgebra is a maximally solvable subalgebra $\mathfrak{b} \subset \mathfrak{g}$.
One can show that $\mathcal{B}$ is a projective variety (thus compact), and is in fact the largest compact quotient of the adjoint group $G$ (see \cite[Chapter VIII]{humphreys2012linear}).

The tangent space $T_\mathfrak{b} \mathcal{B}$ is canonically identified with $\mathfrak{g}/\mathfrak{b}$.
The Borel subgalgebra $\mathfrak{b}$ acts in an upper triangular way on $T_\mathfrak{b} \mathcal{B}$, therefore inducing a filtration of $T \mathcal{B}$:

$$0 \subset T^1 \mathcal{B} \subset \ldots \subset T^l \mathcal{B} = T \mathcal{B},$$

where $l$ is the rank of $\mathfrak{g}$ (see Section \ref{subsection:triplets} below).
The distribution $T^i_\mathfrak{b}$ is $B-$invariant, for $B \subset G$ the group corresponding to $\mathfrak{b}$.
This filtration generalizes the picture we had for $\mathfrak{sl}_n (\CC)$.

\subsection{Symmetric spaces}
\label{subsection:complexasreal}

Let $\mathfrak{g}$ be a complex semisimple Lie algebra, equipped with the Killing form $\kappa (x,y) = tr (ad (x)\circ ad(y))$.
A \emph{Cartan involution} on $\mathfrak{g}$ is a complex antilinear automorphism $\theta : \mathfrak{g} \rightarrow \mathfrak{g}$ satisfying $\theta^2 = Id$, and the property that $- \kappa (x,\theta x) \geq 0$ for every $x \in \mathfrak{g}$.
The last condition is equivalent to the fixed point set $Fix (\theta) =: \mathfrak{k}_\theta$ being a (real) maximal compact subalgebra of $\mathfrak{g}$.
Every maximal compact subalgebra arises this way.

Define $\mathbb{X}$, the \emph{symmetric space of $\mathfrak{g}$}, as the space of Cartan involutions.
One can check that $\mathbb{X}$ is a submanifold of $Aut (\mathfrak{g})$.
Let $\underline{\mathfrak{g}} = \mathbb{X} \times \mathfrak{g}$ denote the trivial bundle over $\mathbb{X}$. 
Since $TAut (\mathfrak{g})$ trivializes via left-translation, we identify the pullback of the bundle $T Aut (\mathfrak{g})$ via the inclusion with $\underline{\mathfrak{g}}$.
The bundle $\underline{\mathfrak{g}}$ comes equipped with a tautological section $\Theta \in \mathcal{A}^0 (\mathbb{X},End (\underline{\mathfrak{g}}))$ for which $\Theta_\theta$ acts as $\theta$ on the fiber.

Define the \emph{Maurer-Cartan form} $\omega \in \mathcal{A}^1 (\mathbb{X},\underline{\mathfrak{g}})$ by the expression:

$$D \Theta = \Theta \circ ad (\omega),$$

where $D$ is the trivial flat connection on $\underline{\mathfrak{g}}$.
Since $\Theta$ is an involution ($\Theta^2 = Id$), differentiating yields $\Theta \circ D\Theta + D\Theta \circ \Theta = 0$, which implies that $\omega$ takes values in the $-1$ eigenspace of $\Theta$.
Denote the $\pm 1$ eigenspaces of the involution at $\theta \in \mathbb{X}$ by $\mathfrak{k}_\theta$, and $\mathfrak{p}_\theta$ respectively.
Thus, $\omega$ defines an isomorphism $T_\theta \mathbb{X} \cong \mathfrak{p}_\theta$.

Let $\underline{\mathfrak{k}}$, $\underline{\mathfrak{p}}$ the bundles that are fixed and antifixed by $\Theta$.
By the previous paragraph, $\underline{\mathfrak{p}}$ is identified with $T \mathbb{X}$.
The connection 
$$\nabla = D + \frac{1}{2} ad (\omega),$$ 
preserves the splitting on $\overline{\mathfrak{g}}$ because $d^\nabla \Theta = 0$.
This connection corresponds to the Levi-Civita connection on $\mathbb{X}$.
If we decompose the curvature of $\nabla$ into its $\underline{\mathfrak{p}}$ and $\underline{\mathfrak{k}}$ components, we get:

\begin{lemma}[Maurer-Cartan]
 The following two equations are satisfied:
   \begin{align*}
   &d^\nabla \omega = 0, \\
   &R^\nabla = -\frac{1}{8} ad([\omega,\omega])
   \end{align*}   
 where $\frac{1}{2} [\omega,\omega] (v,w) = [\omega (v),\omega(w)]$ is the bracket of $1-$forms with values in $\underline{\mathfrak{g}}$.
\end{lemma}

From the expression above, we get that geodesics passing through $\theta \in \mathbb{X}$ are of the form $Ad(e^{tx}) \circ \theta \circ Ad(e^{-tx})$ for $x \in \mathfrak{p}_\theta$, and that $\mathbb{X}$ is non-positively curved.
Here is another useful corollary:

\begin{coro}
\label{prop:symmetric}
Given $\mathfrak{s} \subset \mathfrak{g}$ a reductive subalgebra of $\mathfrak{g}$, then:
$$P_{\mathfrak{s}} = \{ \theta \in \mathbb{X}: \theta \mathfrak{s} = \mathfrak{s}\},$$
is a totally geodesic submanifold. 
The tangent space at $\theta \in \mathbb{X}$ is given by $\mathfrak{s} \cap \mathfrak{p}_\theta$.
\end{coro}

Since $\mathbb{X}$ is connected \cite[Section VI.2]{knapp1996lie}, it is diffeomorphic to $T_\theta \mathbb{X}$ via the exponential map.
Let $\theta' = Ad (e^x) \theta Ad (e^{-x})$ for $x \in \mathfrak{p}_{\theta}$.
One can show that any $x \in \mathfrak{p}_\theta$ acts on $\mathfrak{g}$ with real eigenvalues (see \cite[Theorem 6.16]{knapp1996lie}). 
Let $\mathfrak{l}_x$ the centralizer of $x$ (i.e., the zero eigenspace), and $\mathfrak{p}_x$ the sum of nonnegative eigenspaces (clearly $\mathfrak{l}_x \subset \mathfrak{p}_x$).
We call $\mathfrak{l}_x$ a \emph{Levi subalgebra}, and $\mathfrak{p}_x$ a \emph{parabolic subalgebra}.

The Levi subalgebra is always reductive, therefore Corollary \ref{prop:symmetric} tells us that it is tangent to a totally geodesic submanifold of $\mathbb{X}$ that we call a \emph{parallel set}.
Geometrically, it consists of all geodesics parallel to the geodesic passing through $\theta$ and $\theta'$ (as in \cite[Section 2.11]{eberlein1996geometry}).
Levi subalgebras associated to $x \in \mathfrak{p}_\theta$ always contain at least one Cartan subalgebra.
We say that the direction $x \in \mathfrak{p}_\theta$ is \emph{regular} if $\mathfrak{l}_x$ is precisely a Cartan subalgebra.

By Corollary \ref{prop:symmetric}, Cartan subalgebras define maximally flat, totally geodesic submanifolds, in the sense that they are the largest submanifolds with zero curvature.
Let $\mathfrak{h} \subset \mathfrak{l}_x$ be a Cartan, then the tangent space to the flat at $\theta$ corresponds to $\mathfrak{a} = \mathfrak{h} \cap \mathfrak{p}_\theta$, which is always a real form of the Cartan $\mathfrak{h}$ (in the sense that $\mathfrak{h} = \mathfrak{a} \oplus i \mathfrak{a}$).
Since there are always copies of Cartan subalgebras in $\mathfrak{l}_x$, any two points in $\mathbb{X}$ are connected by a flat.
This flat is unique if the direction joining the pair of points is regular.

We can identify the flat $P_\mathfrak{h}$ determined by a $\theta-$invariant Cartan with $\mathfrak{a} = \mathfrak{p}_\theta \cap \mathfrak{h}$ via the exponential map.
The roots $\Sigma$ defined by $\mathfrak{h}$ define coroots $\{h_\alpha: \alpha \in \Sigma\}$ whose real span is equal to $\mathfrak{a}$ by the Cartan involution condition.
The zeros of the hyperplanes $\ker \alpha$ split $\mathfrak{a}$ into sectors $W$ called \emph{Weyl chambers}.
Correspondingly, we get a splitting of $P_\mathfrak{h}$ into convex cones tipped at $\theta$, which we call Weyl chambers tipped at $\theta$.
We identify those with their corresponding $W \subset \mathfrak{a}$.

Given a Weyl chamber tipped at $\theta$ corresponding to $W \subset \mathfrak{a}$, any geodesic starting at $\theta$ tangent to $W$ will give us different points in the visual boundary $\partial \mathbb{X}$.
This embeds:
\begin{align*}
& \mathbb{P}^+(W) \xrightarrow{\iota_W} \partial \mathbb{X}\\
& \iota_W (\langle x \rangle) = \exp_\theta (\infty x),  
\end{align*}

where $\mathbb{P}^+(W)$ is the (positive) projectivization of $W$, and $\exp_\theta (\infty x)$ denotes the corresponding point at infinity of the geodesic starting at $\theta$ with direction $x$.
One can prove that the image of $\iota_W$ into $\partial \mathbb{X}$ is a fundamental domain for the action of $G$ on $\partial \mathbb{X}$.
We call such a fundamental domain of $\partial \mathbb{X}$ an \textit{apartment} of $\partial \mathbb{X}$.
The terminology comes from the theory of spherical buildings (see \cite[Chapter 2]{eberlein1996geometry} for more information).

Given a direction defined by $x \in W$, one can show that the stabilizer of the corresponding point at infinity is $P_x$, the group of $G$ integrating the parabolic $\mathfrak{p}_x$.
This shows that for each direction $\tau \in \mathbb{P}(W)$, the $G-$orbit of $G \iota_W (\tau) \subset \partial \mathbb{X}$ is canonically identified with the \emph{partial flag manifold} consisting of parabolic subalgebras in the same conjugacy class.
The isomorphism assigns a direction $x \in \mathfrak{g}$ in the $G-$orbit of an element in $\tau$ to its corresponding parabolic subalgebra $\mathfrak{p}_x$.

When the direction in $\mathbb{P}(W)$ corresponds to a regular vector, the parabolic subalgebra $\mathfrak{p}_x$ is a Borel subalgebra, and therefore the orbit of its corresponding point at infinity identifies with the full flag manifold introduced in the previous section.
The Borel subgroup not only fixes a regular vector, but also the whole apartment containing it. 
This identifies the $G-$orbit of an apartment with $\mathcal{B}$.
Figure \ref{fig:stratification} summarizes this discussion.

Sometimes, we will fix a direction in the Weyl chamber and look at the $G-$orbit of the point defined at infinity. 
We will call such an orbit a \textit{stratum} of $\partial \mathbb{X}$.

\begin{center}
\label{fig:stratification}
\begin{figure}[h!]
   \includegraphics[scale=1.2]{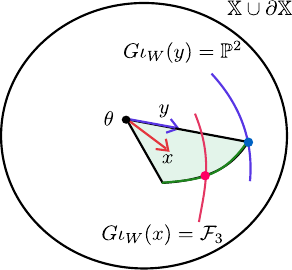}
   \caption{When $G = SL_3 (\mathbb{C})$, the boundary stratifies into $G$ orbits identified with the full flag manifold $\mathcal{F}_3$ or partial flag manifolds like $\mathbb{P}^2$}
\end{figure}   
\end{center}

\subsection{Principal triplets}
\label{subsection:triplets}

Let $(e_0,h_0,f_0)$ be the standard generators of $\mathfrak{sl}_2$ given by:

$$
e_0 = \begin{pmatrix}
0 & 1 \\
0 & 0
\end{pmatrix},\:
h_0 = 
\begin{pmatrix}
\frac{1}{2} & 0\\
0 & -\frac{1}{2}
\end{pmatrix}, \:
f_0 = \begin{pmatrix}
0 & 0\\
1 & 0
\end{pmatrix}.$$

These elements satisfy the relations:

\begin{equation}
\label{eq:triplets}
 [h_0,e_0] = e_0,\: [h_0,f_0] = -f_0, \: [e_0,f_0] = h_0,
\end{equation}

which make the representation theory of $\mathfrak{sl}_2$ easy to understand.
Any $\mathfrak{sl}_2$ representation splits as a direct sum of irreducible representations, and there is a unique irreducible $(n+1)-$dimensional representation $W^n$ for every $n$.

Each $W^n$ splits as a direct sum of $n+1$ lines, which are the eigenspaces of $h_0$ (the weight spaces), and the eigenspaces are ordered by their corresponding eigenvalue (weight).
The transformation $e_0$ sends the weight space at level $l$ to the weight space at level $l+1$, whereas $f$ lowers the levels.
This allows us to draw cartoons of these representations as in Figure \ref{fig:sp6}.

\begin{center}
   \begin{figure}[h!]
      \label{fig:sp6}
      \includegraphics[scale=1.2]{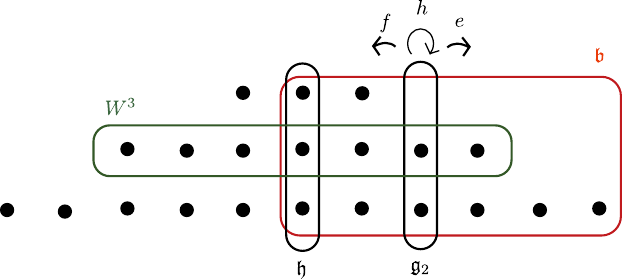}
      \caption{Decomposition induced on $\mathfrak{sp}_6$ by a principal triple (exponents $1,3$ and $5$). Each dot corresponds to a weight space for $h$.}
   \end{figure}
\end{center}

Given $\mathfrak{g}$ a complex Lie algebra, we say that a triple $(e,h,f) \in \mathfrak{g}^3$ is a $\mathfrak{sl}_2-$triplet if it satisfies the relation given by Equation (\ref{eq:triplets}) (we follow the convention by Kostant \cite{kostant2009principal}).
The span generated by such a triple is a subalgebra isomorphic to $\mathfrak{sl}_2 (\mathbb{C})$.
This subalgebra acts on $\mathfrak{g}$, and decomposes it as a direct sum of irreducible representations.

Recall that an element $x \in \mathfrak{g}$ is \emph{regular} if $\dim Z (x) = \rk (\mathfrak{g}) =l$ (which is the minimal possible dimension).
We say that a $\mathfrak{sl}_2$ triplet is \emph{principal} if any element in the triplet is regular.
Let $\mathfrak{g} \cong \bigoplus_{i=1}^l W^{k_i}$ be the decomposition of $\mathfrak{g}$ in irreducible $\mathfrak{sl}_2-$representations.
Each $W^{k_i}$ must contain a zero-weight space that spans a Cartan, so $k_i = 2m_i$, and the eigenvectors of $h$ are integers.
The numbers $m_i$, $i=1,\ldots,l$ appearing in that decomposition are called the \emph{exponents} of $\mathfrak{g}$ (see \cite{kostant2009principal}). 
We will assume the numbers $m_i$ are ordered in increasing order.

Collecting elements associated with the same weight space, we also get a grading $\mathfrak{g}$ $\mathfrak{g} = \oplus_{i=-m_l}^{m_l} \mathfrak{g}_{i},$ where $\mathfrak{g}_{i} = \{x \in \mathfrak{g}: [h,x] = ix\}$.
The grading is compatible with the bracket: $[\mathfrak{g}_{i}, \mathfrak{g}_{j}] \subset \mathfrak{g}_{(i+j)}$.
One can verify that the solvable subalgebra $\bigoplus_{i \geq 0} \mathfrak{g}_i$ is Borel.
Furthermore, $\mathfrak{g}^j = \bigoplus_{i \geq -j} \mathfrak{g}_i$ induces the filtration of Section \ref{sec:fullflag} in the sense that $\mathcal{D}_i = \mathfrak{g}^i/\mathfrak{b} \subset T_\mathfrak{b} \mathcal{B}$.

A principal $\mathfrak{sl}_2-$triple, can be constructed from a Chevalley basis $\langle h_\alpha, e_\alpha,f_\alpha : \alpha \in \Pi \rangle$ of $\mathfrak{g}$ (as in \cite[Chapter VI]{serre2000complex}) via:

$$h = \frac{1}{2}\sum_{\alpha \in \Sigma(\mathfrak{h})^+} h_\alpha =: \sum_{\alpha \in \Pi} r_\alpha h_\alpha, \\
e = \sum_{\alpha \in \Pi} \sqrt{r_\alpha} e_\alpha,\: f= \sum_{\alpha \in \Pi} \sqrt{r_\alpha} f_\alpha,$$

where $r_\alpha$ are integer numbers depending on the root datum.
For this triple, one can check that $\mathfrak{g}_i$ is the sum of the root spaces $\mathfrak{g}_\alpha$ for which $\alpha$ has height $i$.

\begin{rmk}
\label{rmk:highestroot}
If $\mathfrak{g}$ is simple, the highest weight vector in the representation $W^{m_l}$ corresponding to the largest exponent is always the root space $\mathfrak{g}_\theta$ associated to the highest root $\theta$ of $\mathfrak{g}$.
\end{rmk}

Let $\phi : \mathfrak{sl}_2 \rightarrow \mathfrak{g}$ be the morphism induced by an $\mathfrak{sl}_2$ triplet $(e,h,f)$, bijective onto $\mathfrak{s} \subset \mathfrak{g}$.
This further induces a \emph{principal embedding}:

$$\mathcal{P}_\phi : \mathcal{B}_{\mathfrak{sl}_2 (\CC)} \rightarrow \mathcal{B},$$

defined by sending a Borel $\mathfrak{b}$ of $\mathfrak{sl}_2$ to the unique Borel subalgebra of $\mathfrak{g}$ containing $\phi (\mathfrak{b})$.
It is not hard to check that this defines an embedding of $\mathbb{P}^1$ into $\mathcal{B}$.
Notice that $\mathcal{P}_\phi$ is $PGL_2 (\CC)-$equivariant. 
Fixing $\mathfrak{b}_0 = \langle h_0,e_0\rangle$, we get that $d_{\mathfrak{b}_0} \mathcal{P}_\phi (f_0 + \mathfrak{b}_0) = f + \mathcal{P}_\phi (\mathfrak{b}_0)$, therefore tangent to the first level of the canonical distribution in $T \mathcal{B}$.

\begin{prop}
\label{prop:principaldirections}
Let $v_f = f \: mod (\mathfrak{b}) \in T_\mathfrak{b} \mathcal{B}$, then $G v_f$ is open and dense in the sub-bundle $T^1 \mathcal{B}$.
\begin{proof}
It suffices to verify this at the base point $\mathfrak{b}$. 
The fiber $T^1_{\mathfrak{b}} \mathcal{B}$ identifies with $\mathfrak{g}_{-1} = \bigoplus_{\alpha \in \Pi} \mathfrak{g}_{-\alpha}$. 
Given $\mathfrak{n} = [\mathfrak{b},\mathfrak{b}] = \bigoplus_{i \geq 1} \mathfrak{g}_i$, note that $[\mathfrak{n}, \mathfrak{g}_{-1}] \subset \mathfrak{b}$, so the action of $B$ on the fiber factors through the group corresponding to the Cartan $H \cong B/N$. 
The adjoint action of $H$ scales each simple root space $\mathfrak{g}_{-\alpha}$ independently. 
Since $f$ has non-vanishing components in all simple root spaces, its $H$-orbit is isomorphic to $(\mathbb{C}^*)^l$, which is open and dense in $T^1_{\mathfrak{b}} \mathcal{B} \cong \mathbb{C}^l$.
\end{proof}
\end{prop}

The set $G v_f$ is the space of all vectors obtained as the differential of some principal embedding.
We will call $\mathcal{O} \subset T^1 \mathcal{B}$ the set of \emph{principal directions}.

\begin{ex}
\label{ex:veronese}
 We can think of $\mathbb{C}^{n+1}$ as the $n-th$ symmetric power $Sym^n \mathbb{C}^2$. 
 With this point of view, we have the Veronese map $V: \mathbb{P} (\mathbb{C}^2) \rightarrow \mathbb{P} (Sym^n \mathbb{C}^2)$ defined by sending a line $L$ to its $n-$th power $L^n$.
   
 This map is equivariant with respect to the standard embedding of $PGL_2 (\mathbb{C}) \rightarrow PGL_{n+1} (\mathbb{C})$ coming from the unique irreducible representation.
 The image of $V( \mathbb{P}^1)$ is called a rational normal curve.
 One can prove that it is characterized by the property that any $n+3$ tuple of different points is in general position in $\mathbb{P}^n$ (see \cite{chern1978abel}).
   
 This implies that for every $1 \leq k \leq n$, we get osculating copies of $\mathbb{P}^k$ around every point of $V(\mathbb{P}^1)$.
 To see this, fix a basepoint $p$ in the curve, and take $k$ different points nearby.
 The property of the rational normal curve implies that there is a unique copy of $\mathbb{P}^k$ passing through $p$ and these $k$ points.
 As we move our points closer to $p$, the $k-$planes converge to a copy of $\mathbb{P}^k$ tangent to the curve.

 This allows us to extend the Veronese map to a map from $\mathbb{P}^1$ into the full flag manifold $\mathcal{F}_{n+1}$.
 By equivariance, we get that it coincides with the map induced by the principal embedding.
 The osculating construction implies that the curve is tangent to the distribution $T^1 \mathcal{F}_n$ defined before.
\end{ex}

A morphism $\phi : \mathfrak{sl}_2 \rightarrow \mathfrak{g}$ similarly defines a map between the symmetric spaces:

$$\mathcal{H}_\phi : \mathbb{H}^3 \cong \mathbb{X}_{\mathfrak{sl}_2 (\mathbb{C})} \rightarrow \mathbb{X},$$

where $\mathcal{H}_\phi (\theta)$ is defined as the unique Cartan involution extending $\phi \circ \theta \circ \phi^{-1} : \mathfrak{s} \rightarrow \mathfrak{s}$.
This is well defined because of the following classical fact (see for instance \cite[Section 2.4]{labourie2017cyclic}).

\begin{lemma}
   Given $(e,h,f)$ a principal triple in $\mathfrak{g}$, then there is a unique Cartan involution $\theta$ verifying that $\theta (e) = -f,\: \theta (h) = -h, \: \theta (f) = -e.$
\end{lemma}

This motivates the following definition:

\begin{defi}
\label{defi:adapted}
   We will say that a Cartan involution $\theta \in \mathbb{X}$ is \emph{adapted} to the triplet $(e,h,f)$ if it satisfies:
   $$\theta (e) = -f,\: \theta (h) = -h, \: \theta (f) = -e.$$
\end{defi}

Notice that the image of the map $\mathcal{H}_\phi$ is composed of those Cartan involutions leaving the subalgebra $\mathfrak{s}$ invariant, therefore, it is a totally geodesic submanifold by Corollary \ref{prop:symmetric}.
In fact, one can easily check that $\mathcal{H}_\phi$ is an isometry.

The maps $\mathcal{P}_\phi$ and $\mathcal{H}_\phi$ are compatible with each other in the sense that $\mathcal{B}_{\mathfrak{sl}_2 (\mathbb{C})}$ is naturally identified with $\partial \mathbb{X}_{\mathfrak{sl}_2 (\mathbb{C})}$ (because $\rk \mathfrak{sl}_2 (\mathbb{C}) =1$), and the boundary map $\partial \mathcal{H}_\phi : \partial \mathbb{X}_{\mathfrak{sl}_2 (\mathbb{C})} \rightarrow \partial \mathbb{X}$ lands in a stratum isomorphic to $\mathcal{B}$.
Passing through these identifications, $\partial \mathcal{H}_\phi =\mathcal{P}_\phi$.
See Appendix \ref{sec:equidistant} for a related discussion.

\section{Opers: definition and parametrization}
\label{sec:introopers}

In this section, we review basic results about opers.
We compile sketches of some crucial results with the hopes of making the material more accessible than the current literature.
New material includes a proof of Theorem \ref{thm:injectivity} using the osculating M\"obius approach, and Definition \ref{defi:normalizedparam} of normalized parametrization, which is adapted to our purposes.
Our exposition follows and expands on \cite[Sections 3,4]{sanders2018pre}.

\subsection{Developing data}
\label{subsec:developing}
In order to make computations in homogeneous spaces, it is often convenient to adopt a principal bundle perspective.
We introduce the following notation:

\begin{defi}
 Let $G$ be a Lie group, $H \subset G$ a closed subgroup, and $M$ a manifold.
 We will call a triple $(P_G,P_H,A)$ \emph{developing data for $G/H$}, where:
   \begin{itemize}
      \item $P_G \rightarrow M$ is a $G-$principal bundle.
      \item $P_H \hookrightarrow P_G$ is a reduction of structure to $H$.
      \item $A$ is a flat connection on $P_G$.
   \end{itemize}
\end{defi}

The nomenclature arises from the fact that developing data $(P_G, P_H, A)$ on $G$ defines maps $\Dev: \tilde{M} \rightarrow G/H$ equivariant with respect to a representation induced by $A$.
These maps are unique up to translation by $G$, and we will refer to them as \emph{Developing maps} for the triple $(P_G, P_H, A)$.
This fact is classical in the Higgs bundles literature (see for instance \cite[Section 4.2]{labourie2017cyclic}), but we remind the reader how this works in Appendix \ref{subsub:homogeneousspaces}.

Given $P$ a principal bundle, we identify the vertical bundle of $P$ with $\underline{\mathfrak{g}} = P \times \mathfrak{g}$ using the action of $G$.
Whenever we have developing data $(P_G,P_H,A)$, one can take the inclusion $\iota: P_H \rightarrow P_G$ and consider:

$$TP_H \xrightarrow{d\iota} TP_G \xrightarrow{A} \underline{\mathfrak{g}} \rightarrow \underline{\mathfrak{g}}/\underline{\mathfrak{h}},$$

where the last arrow corresponds to the natural projection coming from the quotient of $\underline{\mathfrak{g}}|_{P_H}$ and $\underline{\mathfrak{h}}$.
The result is a $1-$form $\overline{\omega} (P_G,P_H,A) \in \mathcal{A}^1 (P_H, \underline{\mathfrak{g}}/\underline{\mathfrak{h}})$ which can be checked, is $H-$invariant, thus descends to $\omega (P_G,P_H,A) \in \mathcal{A}^1 (M, P_H [\mathfrak{g}/\mathfrak{h}])$.
We will call $\omega (P_G, P_H, A)$ the \emph{Maurer-Cartan form} associated with the developing data.
This Maurer-Cartan form gets identified with the derivative of the developing map $d \Dev: T \tilde{M} \rightarrow T G/B$.
More precisely, $dDev = p^* \omega$, where $p : \tilde{M} \rightarrow M$ is a universal covering map.

We will be working with semisimple Lie groups $G$ of adjoint type.
Under those conditions, $G$-bundles over $M$ are in bijective correspondence with $\mathfrak{g}-$bundles on $M$.
By $\mathfrak{g}-$bundle, we mean a vector bundle equipped with a bracket $[\cdot, \cdot]$ making each fiber isomorphic to $\mathfrak{g}$.
Flat connections on the $G-$bundle are in correspondence with flat affine connections on the corresponding $\mathfrak{g}-$bundle.

We will be concerned with two special instances of this construction corresponding to maps to the flag manifold and the symmetric space.
In these cases, the developing data and the Maurer-Cartan form take a more concrete description:

\begin{description}
   \item[Maps to the full flag manifold] Let $B \subset G$ a Borel subgroups, developing data is determined by $(P_G [\mathfrak{g}], P_B [\mathfrak{b}], D)$.
 The Maurer-Cartan form $\omega$ can be computed by considering any connection $D_B$ that preserves the reduction of structure $P_B [\mathfrak{b}]$ and considering $\omega = D - D_B \: mod \: P_B[\mathfrak{b}]$.
   \item[Maps to the symmetric space] Developing data is determined by $(P_G [\mathfrak{g}],\Theta,D)$, where $\Theta \in Aut (P_G [\mathfrak{g}])$ is fiberwise a Cartan involution.
 The Maurer-Cartan form $\omega$ can be computed via the expression $D \Theta = \Theta \circ ad (\omega)$.
\end{description}

\subsection{Definitions and basic properties}
Let $X$ be a Riemann surface, and $G$ a complex semisimple Lie group of adjoint type.
Beilinson and Drinfeld introduced the following definition in \cite{beilinson2005opers}.
We rephrase it slightly to fit our purposes.

\begin{defi}
\label{defi:oper}
Let $X$ be a Riemann surface with a choice $p: \tilde{X} \rightarrow X$ of universal cover with deck group $\pi_1 (X)$.
A $G-$oper is a pair $(\mathcal{D}, \rho)$, where:
\begin{itemize}
    \item $\rho : \pi_1 (X) \rightarrow G$ is a representation.
    \item $\mathcal{D} : \tilde{X} \rightarrow \mathcal{B}$ is a $\rho-$equivariant holomorphic map such that $d_z \mathcal{D} (T_z X) \subset \mathcal{O}_{\mathcal{D}(z)}$ for all $z \in \tilde{X}$.
 Here $\mathcal{O} \rightarrow \mathcal{B}$ is the set of principal directions introduced in Section \ref{subsection:triplets}.
\end{itemize}
We will refer to $\mathcal{D}$ as the developing map, and $\rho$ as the holonomy of the oper.
\end{defi}

\vspace{\parskip}

\textbf{Examples:} 

\begin{description}
    \item[Complex projective structures] The condition $d \mathcal{D} (T\tilde{X}) \subset \mathcal{O}$, for $G = PGL_2 (\mathbb{C})$ is equivalent to $\mathcal{D}$ being locally injective.
    This is the same as giving a $(\mathbb{P}^1,PGL_2 (\mathbb{C}))$ structure.
    \item[$G-$opers coming from $PGL_2 (\mathbb{C})$ opers] Given $\mathcal{P}: \PP^1 \rightarrow \mathcal{B}$ a principal embedding onto the flag manifold of $G$, the composition of the developing map of the complex projective structure with $\mathcal{P}$ defines a $G-$Oper.
   \item[Opers and ODE's] As a subexample of the previous item, when $G = PGL_n (\mathbb{C})$, the principal embedding comes from the osculating flags construction on a rational normal curve in $\mathbb{P}^{n-1}$ as described in Section \ref{subsection:triplets}.
 The curve to projective space is determined by a special type of complex ODE, which has been classically recognized as part of a family parametrized by $\bigoplus_{i=2}^n H^0 (X, K^i)$ (see \cite[Chapter 4]{wentworth2016higgs}, \cite{hejhal1975monodromy}).
 A similar description exists for other classical Lie groups (see \cite{beilinson2005opers}).
\end{description}

There is a natural notion of isomorphic opers:

\begin{defi}
We will say that a pair of $G-$opers $(\mathcal{D}_1,\rho_1), ( \mathcal{D}_2,\rho_2)$ are \emph{isomorphic} if there exists $g \in G$ and $h : \tilde{X}_1 \rightarrow \tilde{X}_2$ biholomorphic and equivariant for the deck group action, such that the following diagram commutes:
\[
\begin{tikzcd}
    \tilde{X} \arrow{r}{\mathcal{D}_1} \arrow{d}{h} & \mathcal{B} \arrow{d}{Ad(g)} \\
    \tilde{X} \arrow{r}{\mathcal{D}_2} & \mathcal{B}
\end{tikzcd}
\]
The set of equivalence classes of opers will be called the \emph{moduli space of $G$ opers} and denoted $\op_X (G)$. 
\end{defi}

It will be convenient to describe opers using developing data as defined in the previous section.

\begin{prop}
Let $B \subset G$ be a Borel subgroup, and $(P_G [\mathfrak{g}],P_B [\mathfrak{b}],D)$ be developing data to $\mathcal{B}$ satisfying:
\begin{itemize}
    \item $D^{(0,1)} (P_B [\mathfrak{b}]) \subset P_B [\mathfrak{b}]$, where $D^{(0,1)}$ is the decomposition of $D$ to its $(0,1)$ component.
    \item The Maurer-Cartan form $\omega_{P_B} (A)$ takes values in $P_B[\mathcal{O}_\mathfrak{b}] \subset P_B [\mathfrak{g}/\mathfrak{b}]$.
    \end{itemize}
Then any developing map $\mathcal{D}: \tilde{X} \rightarrow \mathcal{B}$ will define an oper.
Different developing maps correspond to isomorphic opers.
\begin{proof}
The first condition corresponds to $\mathcal{D}$ being holomorphic. 
The second condition implies that the derivative is tangent to $\mathcal{O}$.
Finally, the equivalence of this opers is a consequence of the discussion in Appendix \ref{subsub:homogeneousspaces}.
\end{proof}
\end{prop}

It is worth noticing that this is usually the starting definition of an oper in the literature.

\subsection{The Drinfeld-Sokolov normal form}
\label{section:ds}

Requiring the developing map to be tangent to the set of principal directions imposes strong restrictions on the developing data that gives rise to opers.
This is an observation attributed to Drinfeld and Sokolov \cite{drinfeld1981equations}, which allows us to parametrize the space of opers $\op_X (G)$.

Choose $(e,h,f)$ a principal $\mathfrak{sl}_2$ triplet, and $\mathfrak{b} = \bigoplus_{i \geq 0} \mathfrak{g}_i$, where $\mathfrak{g}_i$ are the graded pieces determined by the triplet.
Denote by $V_e = \ker (ad (e))$ the space of highest weight vectors for the $\mathfrak{sl}_2$ representation.
Notice that $V_e$ decomposes as a direct sum of weight spaces corresponding to $W^{m_i} \cap V_e$, where $W^{m_i}$ is an irreducible subrepresentation.

From now on, we will assume that the group $G$ is \emph{simple}.
We only use this condition to ensure that there is a unique appearance of $1$ as an exponent of $\mathfrak{g}$.
One can generalize all the results to semisimple groups by keeping track of each simple factor.

\begin{lemma}
\label{lemma:drinfeldsokolov}
Let $(P_G [\mathfrak{g}],P_B [\mathfrak{b}], D)$ be developing data defining a $G-$oper.
In any coordinate chart $z$ defined in $U \subset X$, we can trivialize $P_B [\mathfrak{b}]|_U \cong U \times \mathfrak{b}$ in such a way that $D$ takes the canonical form:
\begin{equation}
D = D^0 + ad(f + \sum_{i=1}^l \alpha_i (z) e_{m_i}) \otimes dz,      
\end{equation}
where $D^0$ is the trivial connection and $e_{m_i} \in V_e \cap W^{m_i}\setminus \{0\}$.
There are no non-trivial gauge transformations preserving this canonical form.
Furthermore, changing coordinate charts $z = \phi (w)$ and transforming back to canonical form, gives us a decomposition with:
$$\alpha_1 (w) = \alpha_1 (z)(\phi' (w))^2 - \frac{1}{2} S\phi (w),\: \alpha_i (w) = \alpha_i (z)(\phi' (w))^{m_i+1}, \forall i \geq 2$$
where $S \phi = (\frac{\phi''}{\phi'})' - \frac{1}{2} (\frac{\phi''}{\phi'})^2$ is the Schwarzian derivative.
\begin{proof}[Sketch]
Take any trivialization of $P_B [\mathfrak{b}]$ in the coordinate chart.
We can write:
$$D = D^0 + ad (A_{-1} (z) + v(z)) dz,$$
where $A_{-1}$ is a section of $P_B [\mathfrak{g}_{-1}]$ and $v (z)$ a section of $P_B [\mathfrak{b}]$.
Using the proof of Proposition \ref{prop:principaldirections}, we can apply a gauge transformation in $H$ to ensure that $A_{-1} (z) = f$.
Once we make this choice, the only remaining freedom we have is to apply $N = [B, B]$ gauge transformations.
A computation shows that we can push $v(z)$ upward in the grading using $N-$valued gauge transformations until only terms in $P_B [V_e]$ remain.
This relies on the fact that $[\mathfrak{n},\mathfrak{g}_i] \subset \bigoplus_{j\geq 1} \mathfrak{g_{i+j}}$, which allows us to use the gauge action to progressively eliminate terms layer-by-layer (with respect to the grading) until we get to the last one.

Now that we express $D$ locally as in the statement, notice that there are no gauge transformations preserving this form.
Such a $B-$valued gauge transformation should centralize $f$, but the centralizer of $f$ is the space of lowest weight vectors, which is not in $\mathfrak{b}$.
See \cite[Chapter 4]{frenkel2007langlands} for a detailed proof, and the change of coordinate computation.
\end{proof}
\end{lemma}

As a corollary of Lemma \ref{lemma:drinfeldsokolov}, we get our desired parametrization of the space of opers:

\begin{thm}[{\cite{beilinson2005opers}}]
\label{thm:bd}
Fix $(e,h,f)$ a principal triplet and $(P_G^0[\mathfrak{g}],P_B^0 [\mathfrak{b}],D^0)$ the developing data of a fixed $G-$oper. 
Then, given $(P_G [\mathfrak{g}], P_B [\mathfrak{b}], D)$ developing data describing a $G-$oper, there exists a unique isomorphism $\phi: P^0_B [\mathfrak{b}] \rightarrow P_B [\mathfrak{b}]$ such that:
$$\phi^* D - D^0 \in H^0 (X, K \otimes P_B [V_e]).$$
Furthermore, choosing $e_{m_i} \in V_e \cap W^{m_i}$, $P_B^0 [V_e] \cong \bigoplus_{i=1}^l K^{m_i}$.
\end{thm}

This shows that $\op_X (G)$ is an affine space modeled on the Hitchin base $\bigoplus_{i=1}^l H^0 (K^{m_i+1})$.
The parametrization depends on a choice of principal triplet and the $V_e$ basis $e_{m_i}$.
We will say that a basis $\{e_{m_i}\}$ of $V_e$ is \emph{homogeneous} if $e_{m_i} \in V_e \cap W^{m_i}$ for all $i$. 

A more down-to-earth treatment is the following: pick two $G-$opers with developing maps $\mathcal{D}^i : \tilde{X} \rightarrow \mathcal{B}$ and identify $\mathcal{B}$ with $G/B$.
The Drinfeld-Sokolov normal form provides canonical lifts $\overline{\mathcal{D}^i} : \tilde{X} \rightarrow G$ of $\mathcal{D}^i$ verifying that $(\overline{\mathcal{D}^i})^{-1} (\overline{\mathcal{D}^i})' = f + \sum_{j=1}^l \alpha_j^l e_{m_j}$.
The gauge transformation $\phi$ from Theorem \ref{thm:bd} is represented by the holomorphic map $\Osc: \tilde{X} \rightarrow G$ defined as $\Osc (z) = \overline{\mathcal{D}^1} (\overline{\mathcal{D}^2})^{-1}$.

Differentiating, we get the expression:

$$(\Osc(z))^{-1} (\Osc(z))' = Ad (\overline{\mathcal{D}^2} (z)) (\sum_{i=1}^l (\alpha_i^1 - \alpha_i^2) e_{m_i}),$$

where $\alpha_i^1 - \alpha_i^2$ are the $K^{m_i+1}$ differential defined by the previous theorem.
A computation shows that this implies that the $2-$jet of $\Osc(z) \mathcal{D}_2 (\cdot)$ coincides with the $2-$jet of $\mathcal{D}_1 (\cdot)$ at $z$ (after taking derivatives and projecting to $T_{gB} G/B = Ad(g) \mathfrak{b}$).
When $G= PGL_2 (\mathbb{C})$, this recovers a well-known result about osculating M\"obius transformations (see \cite[Section 3]{dumas2009complex},\cite{thurston1986zippers}).
However, higher derivatives have more complicated expressions.
Nevertheless, we will still refer to $\Osc (z)$ as the \textit{osculating maps}.

\begin{ex}
Let $G = PGL_{n+1} (\mathbb{C})$ and suppose $\mathcal{D}$ is the developing map of an oper satisfying 
$$\overline{\mathcal{D}^{-1}} \overline{\mathcal{D}}' (z)= 
\begin{pmatrix}
0 & 0 & \cdots & 0 & \alpha_n (z) \\
1 & 0 & \cdots & 0 & 0 \\
0 & 1 & \ddots & \vdots & \vdots \\
\vdots & \ddots & \ddots & 0 & 0 \\
0 & \cdots & 0 & 1 & 0
\end{pmatrix},$$ 
this is of the form $f + \alpha_n e_{m_n}$, for $e_{m_n}$ a highest weight vector of higher degree, for some $\mathfrak{sl}_2$ principal triplet in $\mathfrak{sl}_{n+1} (\mathbb{C})$.
The differential equation implies that $\overline{\mathcal{D}}$ is the Wronskian $W(x_1,\ldots,x_{n+1})$ of $(n+1)-$linearly independent solutions of the differential equation:
$$x^{(n+1)} - \alpha_n  x =0.$$
For each linearly independent solutions, the curve in $PGL_{n+1} (\mathbb{C})$ projects down to a unique curve in $\mathcal{F}_{n+1}$.
When $\alpha_n =0$, one can take the solutions to be $x_i (z) = z^i$, which recovers the curve osculating the Veronese embedding $(1:z) \to (1:z:\ldots:z^n)$.
\end{ex}

Notice that we have a natural \emph{holonomy map} $Hol: \op_X (G) \rightarrow Hom (\pi_1 (X), G)/G$ into the character variety $Hom (\pi_1 (X), G)/G$ defined as the quotient of $Hom (\pi_1 (X), G)$ under the conjugation action (one can check that the image lands on smooth points of the character variety).
We get the following corollary of Theorem \ref{thm:bd}, that we could only find documented for $G = PGL_n (\mathbb{C})$ in \cite{wentworth2016higgs}.

\begin{thm}
\label{thm:injectivity}
If $X$ is a compact Riemann surface, then the holonomy map $Hol: \op_X (G) \rightarrow Hom (\pi_1 (X), G)/G$ is injective.
\begin{proof}
   Suppose we have two opers with developing maps $\mathcal{D}^i$ and holonomies $\rho^i$ for $i=1,2$.
   The osculating map introduced above satisfies $\Osc (\gamma (z)) \rho_2 (\gamma) = \rho_1 (\gamma) \Osc (z)$. 
   
   If both opers share the same holonomy $\rho$, we get that $\rho(\gamma)^{-1} \Osc (\gamma(z)) \rho (\gamma) = \Osc (z)$.
   Particularly $tr (Ad(\Osc (\gamma z)^k)) = tr (Ad(\Osc (z))^k))$ for every $k \in \mathbb{Z}_{\geq 1}$, $z \in \tilde{X}$ and $\gamma$ (we think of $\Osc$ as acting in $\mathfrak{g}$). 
   This implies that $z \to tr (Ad(\Osc (z)^k))$ descends to an holomorphic function $F_k : X \rightarrow \mathbb{C}$.
   Since $X$ is compact, the function is constant for every $k$ by the maximum principle. 
   This implies that $\Osc(z) =g \in G$ is a constant function, and the opers are isomorphic via $g$.
\end{proof}
\end{thm}

More generally, Sanders \cite{sanders2018pre} proves that the holonomy map is locally injective when we allow ourselves to change the Riemann surface structure on the surface.

\begin{rmk}
As mentioned before, a $PGL_n (\mathbb{C})$ oper is determined by a curve $c: \tilde{X} \rightarrow \mathbb{P}^{n-1}$.
If we identify $\tilde{X}$ with the unit disk $\Delta$, then for $z_0 \in \Delta$ and $\varepsilon$ sufficiently small, the oper condition ensures that the points $c (z_0), c(z_0 + \varepsilon),\ldots, c(z_0 + (n+2)\varepsilon)$ lie in general position.
By the defining property of rational normal curves (recall Example \ref{ex:veronese}), there exists a unique rational normal curve $C_\varepsilon$ passing through these $n+3$ points.
One may expect that when $\varepsilon$ approaches zero, the rational normal curves $C_\varepsilon$ converge to a rational normal curve ``osculating'' $c$ at $z_0$.
Some arduous computations for $G = PGL_3 (\mathbb{C})$ show that this intuition is correct.
When $G = PGL_2 (\mathbb{C})$, this is a well-known computation (see \cite[Section 1.3]{ovsienko2004projective}).
Since the group $PGL_n (\mathbb{C})$ acts transtively on parametrized normal curves, thus the construction recovers the osculating map as the only transformation that maps the standard normal curve to the osculating one.
\end{rmk}

\subsection{Normalized parametrization of \texorpdfstring{$\op_X (G)$}{Op_X (G)}}
\label{subsec:normalizedparam}

We have seen that a base-point oper, a principal triplet, and a homogeneous basis for $V_e$ identify the moduli of opers $\op_X (G)$ with the Hitchin base.
Our goal for this section is to pick a special parametrization of $\op_X (G)$ which we call \emph{normalized}.

Our preferred basepoint for $\op_X (G)$ will be the \emph{Fuchsian oper} that we now describe.
Recall that the Uniformization Theorem claims that $\tilde{X}$ is biholomorphic to either a round disk, $\CC$, or $\PP^1$.
We call a map $u : \tilde{X} \rightarrow \PP^1$ realizing this biholomorphism a \emph{uniformization}, and we denote its holonomy by $\rho_f : \pi_1 (X) \rightarrow PGL_2 (\CC)$.

\begin{defi}
The \emph{Fuchsian oper} is the class $[(p,\mathcal{P}_\phi \circ u, \phi \circ \rho_f)] \in \op_X (\mathfrak{g})$, for $\phi : PGL_2 (\CC) \rightarrow G$ a morphism onto a principal $PGL_2 (\CC)$, and $\mathcal{P}_\phi$ the induced map between the flag manifolds.
\end{defi}

Notice that the definition is independent of $\mathcal{P}_\phi$, and the choice of uniformization.

The standard assumption from now on will be that $X$ is hyperbolic with metric $H$ (as in Section \ref{sec:notation}).

In \cite[Example 1.5]{hitchin1987self}, Hitchin gave a gauge theoretic proof of the uniformization theorem that allows us to explicitly write down a principal bundle, and a flat connection that gives rise to the holonomy $\rho_f$ which we now describe.

Let $(e,h,f)$ be an $\mathfrak{sl}_2-$triple of $\mathfrak{sl}_2 (\CC)$ decomposing $\mathfrak{sl}_2 (\CC) = \bigoplus_{i=-1}^1 \mathfrak{g}_i$ according to the weight spaces.
We can use them to define an $\mathfrak{sl}_2-$bundle over $X$ as: 
$$P_{PGL_2 (\CC)} [\mathfrak{sl}_2 (\CC)] = (K^{-1} \otimes \underline{\mathfrak{g}_{-1}}) \oplus (\mathcal{O} \oplus \underline{\mathfrak{g}_0}) \oplus (K \otimes \underline{\mathfrak{g}_1}),$$
where given $V$ a vector space, $\underline{V} = X \times V$ denotes the trivial bundle.
This bundle admits a natural reduction of structure to the upper triangular group $B$ by defining $P_B [\mathfrak{b}] = (\mathcal{O} \otimes \underline{\mathfrak{g}_0}) \oplus (K \otimes \underline{\mathfrak{g}_1}).$
Moreover, $P_{PGL_2 (\mathbb{C})} [\mathfrak{sl}_2 (\mathbb{C})]$ inherits a connection $\nabla$ from the Chern connection on $K^i$ by declaring $\nabla (\beta_i \otimes x_i) = (\nabla^i \beta_i) \otimes x_i,\: \forall \beta_i \otimes x_i \in \mathcal{A}^0 (X, K^i \otimes \underline{\mathfrak{g}_i}),\: i=-1,0,1$ (this connection is well defined because the bundles $\underline{\mathfrak{g}_i}$ are trivial).
It is not hard to see that $\nabla$ cannot be flat.
However, we can define:

$$\Phi = \tau \otimes ad(f) + H \otimes ad(e) \in \mathcal{A}^0 (X, End (P_{PGL_2 (\CC)} [\mathfrak{sl}_2 (\CC)])),$$

where $\tau \in \mathcal{A}^{(1,0)} (X, K^{-1})$ is given in coordinates by $\tau = dz \otimes \partial_z$ (independent of the choice of coordinates), and $H$ is the hyperbolic metric thought of as $H dz \otimes \overline{dz} \in \mathcal{A}^{(0,1)}(X,K)$.

\begin{prop}[Hitchin]
\label{prop:hitchin}
The connection $D = \nabla + \Phi$ is flat and the triplet \\$(P_{PGL_2 (\CC)}[\mathfrak{sl}_2],P_B [\mathfrak{b}], D)$ defines developing data for the Fuchsian oper.
\begin{proof}
We have the well-known formula $R^D = R^\nabla + d^\nabla \Phi + \frac{1}{2}[\Phi,\Phi]$.
Notice that $d^\nabla \Phi = 0$ by compatibility of the Chern connection with $H$.
On the other hand, given $\beta_i \otimes x_i \in \mathcal{A}^0 (X, K^i \otimes \underline{\mathfrak{g}_i})$, $R^\nabla (\beta_i \otimes x_i) = i H \beta_i \otimes x_i = H[1\otimes h, \beta_i \otimes x_i]$. 
The final term left to compute is $\frac{1}{2} [\Phi,\Phi] = -H ad (1 \otimes h)$ (because $[f,e] = -h$).
Adding those terms in the expression for $R^D$, we get zero.

Notice that both $\nabla$ and $H \otimes ad (e)$ preserve the subbundle $P_B [\mathfrak{b}]$.
This shows that $P_B [\mathfrak{b}]$ is a holomorphic subbundle (for the holomorphic structure $D^{(0,1)}$) and that the Maurer-Cartan form of the triple is given by $\tau \otimes f \in \mathcal{A}^{(1,0)} (X, K^{-1}\otimes \underline{\mathfrak{g}_{-1}})$ which is clearly tangent to the set of principal directions.
All of that shows that the triple defines an oper.

To see that it defines the Fuchsian oper, one verifies that the holonomy preserves a totally geodesic copy of $\mathbb{H}^2$ inside $\mathbb{X}_{PGL_2 (\mathbb{C})}$.
Picking $\theta$ a Cartan involution adapted to the triple, and consider $*_i : K^i \rightarrow K^{-i}$ the antilinear involution defined by the metric (as in Section \ref{sec:notation}), we can define:
$$
\Theta (\beta_i \otimes x_i) = *_i \beta_i \otimes \theta x_i,\: \forall \beta_i \otimes x_i \in \mathcal{A}^0 (X,K^i \otimes \underline{\mathfrak{g}_i}),
$$
which is fiberwise a Cartan involution.
The triple $(P_G [\mathfrak{g}], \Theta, D)$ defines developing data for a map into the symmetric space. 
One can verify that its image is a totally geodesic copy of $\mathbb{H}^2$.
This will follow from the discussion in the upcoming sections, so we omit the proof. 
\end{proof}
\end{prop}

Notice that one needs the curvature of $H = H(z) dz \overline{dz}$ to be $-2$ in order for the previous computation to work.

The same objects can be defined over a principal $\mathfrak{sl}_2$ subalgebra of a simple Lie algebra $\mathfrak{g}$.
We obtain:

\begin{coro}
\label{coro:fuchsian}
Let $X$ be a hyperbolic Riemann surface and $(e,h,f)$ a principal $\mathfrak{sl}_2$ triplet on $\mathfrak{g}$ inducing a grading $\mathfrak{g} = \bigoplus_{i=-m_l}^{m_l} \mathfrak{g}_i$. 
Define:
\begin{align*}
& P_G [\mathfrak{g}] = \bigoplus_{i=-m_l}^{m_l} K^i \otimes \underline{\mathfrak{g}_i}, \: P_B [\mathfrak{b}] = \bigoplus_{i=0}^{m_l} K^i \otimes \underline{\mathfrak{g}_i}, \: D = \nabla + \Phi,
\end{align*}
where $\nabla$ is the connection induced by the Chern connections $\nabla^i$ on $K^i$, and:
$$\Phi = \tau \otimes ad (f) + H \otimes ad (e).$$
Then, the connection $D$ is flat, and the mapping data $(P_G [\mathfrak{g}],P_B[\mathfrak{b}], D)$ defines the Fuchsian oper on $G$.
\begin{proof}
   Given $S$ the subgroup defined by the triplet $(e,h,f)$, notice that $P_S [\mathfrak{sl}_2 (\mathbb{C})] = (K^{-1} \otimes \underline{\langle f \rangle}) \oplus (\mathcal{O} \otimes \underline{\langle h \rangle}) \oplus (K \otimes \underline{\langle e \rangle})$ is a $D-$invariant subbundle isomorphic to $P_{PGL_2 (\mathbb{C})} [\mathfrak{sl}_2 (\mathbb{C})]$.
   This defines an injective bundle map $P_{PGL_2 (\mathbb{C})} [\mathfrak{sl}_2 (\mathbb{C})] \hookrightarrow P_G [\mathfrak{g}]$ for which the pullback of $D$, corresponds to the flat connection of the $PGL_2 (\mathbb{C})$ Fuchsian oper defined earlier.
   Moreover, the reduction of structure to the upper triangular group is pointwise contained in a unique Borel subalgebra, which is the one determined by the reduction to $P_B [\mathfrak{b}]$.
   This implies that when we develop the triple $(P_G [\mathfrak{g}],P_B [\mathfrak{b}], D)$, the image is equal to the composition of the developing map of the $PGL_2 (\mathbb{C})$ oper with a principal embedding as in Section \ref{subsection:triplets}.
\end{proof}
\end{coro}

\begin{defi}
   Given $X$ a hyperbolic Riemann surface, and $(e,h,f)$ a principal $\mathfrak{sl}_2 (\mathbb{C})-$\\triplet on $\mathfrak{g}$, we will refer to $(P_G [\mathfrak{g}], P_B [\mathfrak{b}], D)$ as the \emph{$G-$Fuchsian triplet} defined by $(e,h,f)$.
\end{defi}

Motivated by the Drinfeld-Sokolov normal form, we can alter the flat connection on the Fuchsian connection to obtain a developing data representing any oper.
Given $(e,h,f)$ a fixed principal triplet and $\{e_{m_i}\}$ a homogeneous basis for $V_e$, we can associate to $\vec{\alpha} \in \bigoplus_{i=1}^l H^0 (X,K^{m_i+1})$ the tensor:

$$\eta^{\vec{\alpha}} = \sum_{i=1}^l (\tau \alpha_i) \otimes ad(e_{m_i}) \in \mathcal{A}^{(1,0)} (X, End(P_G[\mathfrak{g}])).$$

\begin{prop}[{\cite[Proposition 4.3]{sanders2018pre}}]
\label{prop:canonicalmodel}
Given $(e,h,f)$ a principal $\mathfrak{sl}_2-$triplet, $\{ e_{m_i}\}$ a homogeneous basis for $V_e$, and $\vec{\alpha} \in \bigoplus_{i=1}^l H^0 (X,K^{m_i+1})$, we get that:
$$D^{\vec{\alpha}} = D + \eta^{\vec{\alpha}},$$
is a flat connection in the Fuchsian bundle $P_G [\mathfrak{g}]$ defined by $(e,h,f)$.
Moreover, the mapping data $(P_G[\mathfrak{g}],P_B[\mathfrak{b}],D^{\vec{\alpha}})$ defines a $G-$oper.
\end{prop}

Thanks to Theorem \ref{thm:bd}, we get:

\begin{coro}
\label{coro:parametrization}
The map: 
\begin{align*}
\bigoplus_{i=1}^l H^0 (&X,K^{m_i+1}) \rightarrow \op_X (G) \\
&\vec{\alpha} \to [(P_G [\mathfrak{g}],P_B [\mathfrak{b}], D^{\vec{\alpha}})].
\end{align*}
is a bijection.
\end{coro}

The following remark will be key for future computations:

\begin{rmk}
\label{rmk:bdgauge}
   The unique isomorphism $\phi : P_G [\mathfrak{g}] \rightarrow P_G [\mathfrak{g}]$ for which $\phi^* D^{\vec{\alpha_1}} - D^{\vec{\alpha_2}} \in H^0 (X, K \otimes P_B [V_e])$ given by Theorem \ref{thm:bd} is the identity map.
\end{rmk}

In this construction, the choice of $\mathfrak{sl}_2$ triplet is not crucial, since two different choices are always naturally identified.
However, the geometry of the parametrization of Corollary \ref{coro:parametrization} will be sensitive to the choice of highest weight vectors $\{e_{m_i}\}_{i=1,\ldots,l}$.
To get rid of this ambiguity, we introduced the following definition:

\begin{defi}
\label{defi:normalizedparam}
Given a principal $\mathfrak{sl}_2$ triplet $(e,h,f)$ in a simple Lie algebra $\mathfrak{g}$, we say that a collection $\{e_{m_i}\}_{i=1,\ldots,l}$ of highest weight vectors $e_{m_i} \in V_e \cap W^{m_i}$ is \emph{normalized} if:
\begin{itemize}
    \item $e_{m_1} = e$.
    \item $-\kappa (e_{m_i}, \theta e_{m_i}) = \kappa (e,f)$ for all $i$, where $\theta$ is the Cartan involution adapted to the triple $(e,h,f)$ (see Definition \ref{defi:adapted}). 
\end{itemize}
We say that the bijection between $\op_X (G) \cong \bigoplus_{i=1}^{l} H^0 (X,K^{m_i+1})$ is \emph{normalized} if we construct it using a Fuchsian oper as our basepoint, and a normalized collection $\{e_{m_i}\}$.
\end{defi}

Normalized collections $\{e_{m_i}\}_{i=1,\ldots,l}$ always exist because $-\kappa (e_{m_i},\theta e_{m_i}) >0$ by the Cartan involution condition.
They are not unique, two such collections differ by $e_{m_j}' = e^{i\theta_j} e_{m_j}$ for some $\theta_j \in \RR$, $j > 1$.
Notice that for $\mathfrak{g} = \mathfrak{sl}_2 (\CC)$, there is no choice to be made.

\begin{rmk}
   As mentioned before, the simplicity of $\mathfrak{g}$ is not crucial if one is careful in picking normalized collections from each simple factor to come up with the parametrization of $\op_X (G)$, and remembers that the highest weight vectors in $\mathfrak{g}_1$ play a special role.
   Most of what we do can be generalized to this setting verbatim.
\end{rmk}

Notice that the highest weight $e_{m_l}$ is special in the sense that it belongs to the root space of $\mathfrak{g}_\theta$, for $\theta$ the highest root (see Remark \ref{rmk:highestroot}).

\begin{defi}
\label{defi:cyclicoper}
   We say that an oper $(D,\rho)$ is \emph{cyclic} if it corresponds to a vector $\vec{\alpha} \in \bigoplus_{i=1}^l H^0 (X,K^{m_i+1})$ via a parametrization of $\op_X (G)$ based at the Fuchsian oper, whose only non-zero entry is in the top differential $\alpha \in H^0 (X,K^{m_l+1})$.
\end{defi}

\section{Epstein-Poincar\'e surfaces}
\label{sec:epstein}
Given $G$ a simple Lie group of adjoint type, the goal of this section is to introduce the notion of an Epstein-Poincar\'e (EP) surface associated to a $G-$oper.
After performing some initial computations with these surfaces, we verify that this recovers Epstein's classical definition  (see \cite{epstein2024envelopes}).

\subsection{The definition}

As we saw in the previous section, a principal $\mathfrak{sl}_2 (\mathbb{C})$ (and a hyperbolic metric) defines a $G-$Fuchsian triplet describing the Fuchsian oper, for $X$ a hyperbolic Riemann surface. 
The corresponding Fuchsian bundle can be equipped with extra structure, e.g., the Cartan involution $\Theta$ defined in Proposition \ref{prop:hitchin} which defines a totally geodesic embedding of $\mathbb{H}^2$ into $\mathbb{X}$. 
The key idea in our construction is that Theorem \ref{thm:bd} provides us a canonical isomorphism between any bundle coming from a $G-$oper and the Fuchsian bundle, allowing us to pullback any extra structure we may have there.

Given $H = H(z) dz \overline{dz}$ the hyperbolic metric on $X$ (normalized as in Section \ref{sec:notation}), we can define a Cartan involution $\Theta^0$ on the $G-$Fuchsian bundle $P_G^0 [\mathfrak{g}]$ defined by $(e,h,f)$ via:
\begin{equation}
    \label{eq:cartan}  
      \Theta^0 (\beta_i \otimes x_i) = * \beta_i \otimes \theta x_i,\: \forall \beta_i \otimes x_i \in \mathcal{A}^0 (X,K^i \otimes \overline{\mathfrak{g}_i}),
\end{equation}
for every $i$, where $*$ is involution in $\bigoplus_{i=-\infty}^\infty K^i$ defined in coordinates by sending $\alpha dz^i \to \frac{\overline{\alpha}}{H^k} dz^{-i}$, and $\theta$ is a Cartan involution adapted to the triplet.
We encountered this Cartan involution used in Proposition \ref{prop:hitchin}.

\begin{defi}[Gauge theoretic definition]
\label{defi:epstein}
 Let $X$ be a hyperbolic Riemann surface and $(P_G [\mathfrak{g}],P_B[\mathfrak{b}], D)$ mapping data for a $G-$oper.
 We define the \emph{Epstein-Poincar\'e} (EP) involution $\Theta^{EP}$ on $P_G [\mathfrak{g}]$ as the fiberwise Cartan involution defined via $\Theta^{EP} = \phi^* \Theta^0$, where $\phi : P_G [\mathfrak{g}] \rightarrow P_G^0 [\mathfrak{g}]$ is the unique isomorphism given by Theorem \ref{thm:bd} and $P_G^0 [\mathfrak{g}]$ is the Fuchsian bundle defined by a principal triplet $(e,h,f)$. 
\end{defi}

This definition is independent of the chosen triplet by the uniqueness of the isomorphism $\phi$.

We can reinterpret this definition in terms of developing maps.
Notice that there are two maps associated with the $G-$Fuchsian oper (defined by a triplet) equipped with the Cartan involution $\Theta^0$.
On the one hand, we have the developing map of the Fuchsian oper $\mathcal{D}^0 : \tilde{X} \rightarrow \mathcal{B}$, and on the other we get the developing $\Ep^0 : \tilde{X} \rightarrow \mathbb{X}$ of the triplet $(P_G^0 [\mathfrak{g}], \Theta^0, D^0)$ with image in a copy of $\mathbb{H}^2 \subset \mathbb{X}$.

\begin{rmk}
 From now on, when we go from developing data to a map to a homogeneous space, we implicity assume that we choose the same trivialization for $P_G [\mathfrak{g}]$ to obtain $\mathcal{D}$ (to $\mathcal{B}$) and $\Ep$ (to $\mathbb{X}$).
\end{rmk}

The isomorphism $\phi$ between the principal bundles is represented by the unique osculating map  $\Osc: \tilde{X} \rightarrow G$ satisfying $\Osc (z). \mathcal{D}^0 (z) = \mathcal{D}(z)$.
This tells us that:

\begin{lemma}[Developed picture]
 Given $(P_G [\mathfrak{g}],P_B [\mathfrak{b}], D)$ mapping data for a $G-$oper, and $\Theta$ the EP involution, then:
   $$\Ep (z) = Osc (z) . \Ep^0 (z),$$
 is a developing map for the triplet $(P_G [\mathfrak{g}],\Theta, D)$.
\end{lemma}

We will call a developing map for $(P_G [\mathfrak{g}], \Theta^{EP}, D)$ an \emph{Epstein-Poincar\'e} (EP) surface.

\begin{center}
   \begin{figure}[h!]
      \label{fig:epstein}
      \includegraphics{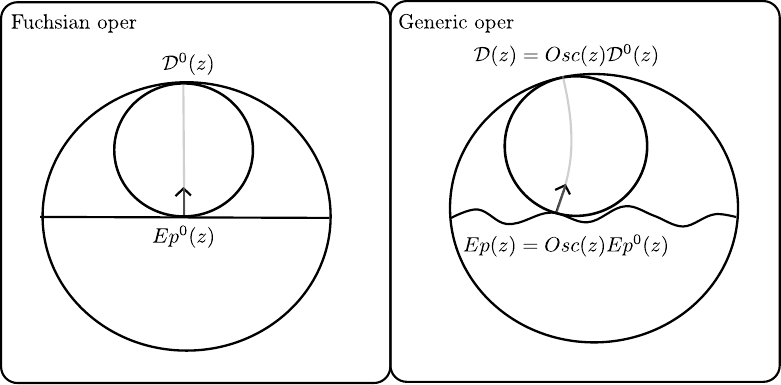}
      \caption{The osculating map allows us to propagate the data at infinity down to the symmetric space.}
   \end{figure}
\end{center}

As we will see in the next section (see Lemma \ref{lemma:normalweyl}), we have a natural distribution of Weyl chambers pointed at the Epstein surface $\Ep$ and orthogonal to it. 
The apartments at infinity corresponding to these Weyl chambers coincide pointwise with the developing map for the oper $\mathcal{D}$. 
This is illustrated in Figure \ref{fig:epstein}.

\subsection{Basic computations}

Thanks to the explicit mapping triples parametrizing the space of opers introduced in Corollary \ref{coro:parametrization}, we can compute any geometric quantity we are interested about $\Ep$ in terms of the defining differentials.
We will assume that $G$ is simple, $(e,h,f)$ a principal triplet and we identify $\op_X (G) \cong \bigoplus_{i=1}^l H^0 (X, K^{m_i +1})$ via a normalized parametrization (see Definition \ref{defi:normalizedparam}). 

For each $\vec{\alpha} \in \bigoplus_{i=1}^l H^0 (X, K^{m_i +1})$ we get developing data $(P_G [\mathfrak{g}],P_B [\mathfrak{b}],D^{\vec{\alpha}})$ as in Proposition \ref{prop:canonicalmodel}.
Remark \ref{rmk:bdgauge} implies that the unique isomorphism $\phi$ between any two such opers given by Theorem \ref{thm:bd} is simply the identity. 
This implies that the EP-involution $\Theta$ on the Fuchsian bundle $P_G [\mathfrak{g}]$ remains the same independent of $\vec{\alpha}$, and is given by Equation (\ref{eq:cartan}).
Thus $(P_G [\mathfrak{g}],\Theta, D^{\vec{\alpha}})$ is developing data for the EP-surface corresponding to the oper defined by $\vec{\alpha}$.

The Cartan involution $\Theta$ splits $P_G [\mathfrak{g}] = P_K [\mathfrak{k}] \oplus P_K [\mathfrak{p}]$, where $\mathfrak{g} = \mathfrak{k} \oplus \mathfrak{p}$ is the splitting defined by the Cartan involution $\theta$ adapted to the triplet, and $K$ is a group integrating $\mathfrak{k}$.
Recall that the bundle $P_K [\mathfrak{p}]$ gets identified with $\Ep^* T\mathbb{X}$, and that the Maurer-Cartan form $\omega$ defined via $D^{\vec{\alpha}} \Theta = \Theta \circ ad (\omega)$ is identified with the differential of the EP-surface corresponding to $\vec{\alpha}$. 
Therefore, we can think of the following as the computation of $d \Ep$:

\begin{lemma}
\label{lemma:maurercartan}
Given $\vec{\alpha} \in \oplus_{i=1}^l H^0 (X, K^{m_i+1})$, the Maurer-Cartan form of $(P_G [\mathfrak{g}], \Theta, D^{\vec{\alpha}})$ is given by the expression:
\begin{align*}
\omega = -2 \left( \tau \otimes f + H \otimes e + \frac{1}{2}  \left(\sum_{i=1}^l (\tau \alpha_i)\otimes e_{m_i} - (H *\alpha_i) \otimes \theta (e_{m_i})\right) \right).
\end{align*}
\begin{proof}
We get that $D^{\vec{\alpha}}\Theta = \nabla\Theta + ad (\Phi + \eta^{\vec{\alpha}}) \circ \Theta - \Theta \circ ad (\Phi + \eta^{\vec{\alpha}})$, with $\eta^{\vec{\alpha}}$ as in Section \ref{subsec:normalizedparam}.
By compatibility of $*$ with the Hermitian metric defined by $H$ on $K^i$, we get that $\nabla \Theta = 0$.
On the other hand: 

\begin{align*}
\Theta \circ (\Phi + \sum_{i=1}^l \eta^{\alpha_i}) \circ \Theta &= ad (\Theta (\tau \otimes f + H \otimes e + \sum_{i=1}^l (\tau \alpha_i) \otimes e_{m_i})),\\    
&= ad(- H \otimes f - \tau \otimes e + \sum_{i=1}^l (H *\alpha_i) \otimes \theta (e_{m_i}) ),
\end{align*}

where we used that $\Theta (\tau \circ f) = - H \circ e$, and $\Theta (H \circ e) = - \tau \circ f$.
subtracting $\Phi + \sum_{i=1}^l \eta^{\alpha}_i$, we get the desired result.
\end{proof}
\end{lemma}

The Killing form $\langle\cdot,\cdot \rangle$ on $P_G [\mathfrak{g}]$ is a $\CC-$bilinear form that restricts to a (real) inner product on $P_K [\mathfrak{p}]$, which identifies with the metric on the symmetric space $\mathbb{X}$.
The grading simplifies computations with the Killing form as follows:

\begin{rmk}
\label{rmk:orthogonality}
Given $\beta_i \otimes x_i \in \mathcal{A}^0 (X, K^i \otimes \underline{\mathfrak{g}_i})$ for $i =1,2$, then:

$$\langle \beta_i \otimes x_i, \beta_j \otimes x_j\rangle = 
\begin{cases}
0,\: \text{if }i\neq -j.\\
(\beta_i,\beta_j) \kappa (x_i,x_j),\: \text{if }i=-j,
\end{cases}$$

where $(\cdot,\cdot)$ denotes the pairing between $K^i$ and $K^{-i}$.
The formula follows from $\kappa (\mathfrak{g}_i,\mathfrak{g}_j)=0$ if $i \neq -j$.
\end{rmk}

Notice that the $\mathfrak{h}-$subbundle $\mathcal{O} \otimes \mathfrak{h} \subset P_G [\mathfrak{g}]$ intersects $P_K [\mathfrak{p}]$ along $\mathcal{O} \otimes \mathfrak{a}$, where $\mathfrak{a}$ is the real form of $\mathfrak{h}$ spanned by the simple coroots.
The subbundle $\mathcal{O} \otimes \mathfrak{a}$ defines a distribution in $T \mathbb{X}$ along $\Ep$ that is orthogonal to the EP-surface thanks to Lemma \ref{lemma:maurercartan} and Remark \ref{rmk:orthogonality}. 
This distribution is tangent to a family of flats.

\begin{lemma}
\label{lemma:normalweyl}
 Let $W \subset \mathfrak{a}$ be the Weyl chamber containing $h$, then after developing, $\mathcal{O} \otimes W$ defines a distribution of $T\mathbb{X}$ along $\Ep$ tangent to Weyl chambers $W_z \subset \mathbb{X}$ tipped at $\Ep(z)$ whose apartment at infinity is determined by the Borel subalgebra $\mathcal{D}(z)$ corresponding to the developing map of the oper.
\begin{proof}
 As explained in Section \ref{subsection:complexasreal}, any vector tangent to the distribution $\mathcal{O}\otimes \mathfrak{a}$ in the same Weyl chamber as $1 \otimes h$ acts on $P_G [\mathfrak{g}]$ with real eigenvalues and defines a Borel subalgebra by taking the direct sum of non-negative eigenspaces. 
 This Borel subalgebra is the same as the one defined by $1 \otimes h$, which is simply the reduction of structure $P_B [\mathfrak{b}]$.
   
 This implies that all the vectors in that Weyl chamber normal to $\Ep (z)$ are pointing towards points in $\partial \mathbb{X}$ defined by $\mathcal{D} (z)$, where $\mathcal{D}(z)$ is the developing map of the oper. 
\end{proof}
\end{lemma}
This generalizes the picture described in the previous section in $\mathbb{H}^3$ to higher rank.

Any section of $P_K [\mathfrak{p}]$ is in the $K-$orbit of a unique element of $1 \otimes W$.
This unique element defines a \textit{Cartan projection}: 
$$\mu: P_K [\mathfrak{p}] \rightarrow 1 \otimes W,$$ 
sending any section $s \in \mathcal{A}^0 (X,P_K [\mathfrak{p}])$ to its unique $K-$representative in $1 \otimes W \subset 1 \otimes \mathfrak{h}$.

Recall that the Levi-Civita connection $\nabla^{LC} = D^{\vec{\alpha}} + \frac{1}{2} ad (\omega)$ preserves $P_K [\mathfrak{p}]$ and corresponds to the pullback of the Levi-Civita connection in $\mathbb{X}$.
Lemma \ref{lemma:maurercartan} gives us an expression for $\nabla^{LC}$:

\begin{coro}
\label{coro:levicivita}
Fixing $(P_G[\mathfrak{g}],\Theta, D^{\vec{\alpha}})$, the Levi-Civita connection on $P_K [\mathfrak{p}]$ is given by the expression:
\begin{equation}
\label{eq:levicivita}
\nabla^{LC} = \nabla + \frac{1}{2} ad \left( \sum_{i=1}^l (\tau \alpha_i) \otimes e_{m_i} + (H *\alpha_i) \otimes \theta (e_{m_i}) \right).
\end{equation}    
\end{coro}

\subsection{Relation with the classical definition}

We now verify that when $G = PGL_2 (\mathbb{C})$, this recovers Epstein's construction.
More precisely, Epstein proved:

\begin{thm}[{\cite{epstein2024envelopes}}]
\label{thm:epstein}
 Let $(\lambda,\Dev)$ be a pair consisting of a $C^1$ conformal metric $\lambda$ in $\mathbb{H}^2$, and $\Dev: \mathbb{H}^2 \rightarrow \mathbb{P}^1$ a locally injective holomorphic map. 
 Then, there exists a unique continuous map $\Ep (\lambda, \Dev) : \mathbb{H}^2 \rightarrow \mathbb{H}^3$ such that:
   $$\Dev^* \nu_{\Ep (\lambda,Dev) (z)} = \lambda, \: \forall z \in \mathbb{H}^2,$$
 where $\nu_{\Ep (\lambda,Dev) (z)}$ is the visual metric on $\partial \mathbb{H}^3 \cong \mathbb{P}^1$ induced from the point $\Ep (\lambda,Dev) (z) \in \mathbb{H}^3$.
 Furthermore, the point $\Ep(\lambda,Dev)(z)$ depends only on the $1-$jet of $\lambda$ at $z$ and $\Ep$ has one less derivative than $\lambda$.
\end{thm}
 
If we normalize the curvature of $\mathbb{H}^3$ to be $-1$, the classical Epstein-Poincar\'e maps are recovered when we consider the pair associated with a hyperbolic metric of constant curvature $-1$, and a given projective structure.

The equivalence with our definition follows from the following:

\begin{prop}
\label{prop:epsteinisepstein}
 Fix $\lambda$ a conformal metric in $\mathbb{H}^2$, and $\Dev^i : \mathbb{H}^2 \rightarrow \mathbb{P}^1$, $i=1,2$ be locally injective holomorphic maps. 
 Let $\Osc : \mathbb{H}^2 \rightarrow PGL_2 (\mathbb{C})$ be the unique osculating Mobius transformation satisfying that the $2-$jets at $z$ of $ \Osc (z)\circ \Dev^2 (\cdot) = \Dev^1 (\cdot)$ coincide, then:
   $$\Osc (z) . Ep_{(\lambda,Dev^2)} =  Ep_{(\lambda,\Dev^1)}.$$
   \begin{proof}
 Notice that $(\Osc(z))^* \nu_{\Osc(z).\Ep (\lambda,\Dev^2)(z)} = \nu_{\Ep (\lambda,\Dev^2)(z)}$. 
 On the other hand:
      $$\lambda = (\Dev^2)^* \nu_{\Ep (\lambda,\Dev^2)(z)} = (\Dev^1)^* \nu_{\Osc (z). \Ep (\lambda,\Dev^2)(z)},$$
 where we used that the $1-$jet of $\Osc (z) \circ \Dev^2 (\cdot) = \Dev^1 (\cdot)$ coincide.
 Uniqueness of the Epstein surface in Theorem \ref{thm:epstein} finishes the proof.
   \end{proof} 
\end{prop}

More intuitively, one can think of a totally geodesic $\mathbb{H}^2$ as generated by an envelope of horocycles tangent to it. 
Each horocycle has a preferred point defined by the point of tangency.
The osculating Mobius map defined at the tip of the horocycle maps each horocycle with its preferred point to a unique horocycle and preferred point. 
The envelopes of these new horocycles define the new Epstein surface. 
These horocycles are illustrated in Figure \ref{fig:epstein}.

\section{Regularity of the EP-surfaces}
\label{sec:regularity}

In this section, we will study when the EP-surface of an oper is \emph{regular}, in the following sense:

\begin{defi}
Let $M$ be a smooth manifold, and $\mathbb{X}$ a symmetric space.
We will say that an immersion $F: M \rightarrow \mathbb{X}$ is \emph{regular} if $d_p F (v) \in \mathfrak{p}^{F(p)}$ is a regular vector for all $(p,v) \in TM$ (as defined in Section \ref{subsection:complexasreal}).
\end{defi}

This infinitesimal condition is natural in the study of well-behaved surfaces from a large-scale geometric perspective.
The main result of this section is a sufficient condition for $\Ep$ to be a regular immersion:

\begin{thm}
\label{thm:regularity}    
Let $G$ be a complex simple Lie group, and $\Ep$ be an EP-surface for the oper $(P_G [\mathfrak{g}],P_B[\mathfrak{b}],D^{\vec{\alpha}})$, as in a normalized parametrization of $\op_X (G)$.
Then:
\begin{itemize}
    \item $\Ep$ is immersed at $z \in \tilde{X}$ as long as $\norm{\alpha_1}(z) \neq 2$ or $\alpha_i (z) \neq 0$.
    \item Let $\phi_{min}$ be the minimal angle between $h$ and the walls of a Weyl chamber.
Then as long as $\Ep$ is immersed, and:
$$\frac{(1-\frac{\norm{\alpha_1}}{2})}{\sqrt{(1-\frac{\norm{\alpha_1}}{2})^2 + \frac{1}{4} \sum_{i=2}^l \norm{\alpha_i}^2}} \geq \cos (\phi_{min}),$$
the surface $\Ep$ will be a regular immersion. 
\end{itemize}
\end{thm}

When $\mathfrak{g} = \mathfrak{sl}_2 (\mathbb{C})$, the second condition is vacuous.
Notice that the constant $\phi_{min}$ can easily be computed in terms of the root data (see Appendix \ref{appendix:root}).

We will be using the Cartan projection $\mu : P_K [\mathfrak{p}] \rightarrow 1 \otimes W$ introduced by the end of the previous section.

\subsection{A totally geodesic distribution along $\Ep$}

The principal triple $(e,h,f)$ used to define $P_G [\mathfrak{g}]$ defines a subbundle:

$$P_S [\mathfrak{s}] = (K^{-1} \otimes \langle f \rangle) \oplus (\mathcal{O} \otimes \langle h \rangle) \oplus (K \otimes \langle e \rangle),$$

The intersection $P_S [\mathfrak{s}] \cap P_K [\mathfrak{p}]$ defines a complex, three-dimensional subbundle, which corresponds to a real three-dimensional distribution along $\Ep$ (after developing).

\begin{lemma}
\label{rmk:distribution}
The bundle $P_S [\mathfrak{s}] \cap P_K [\mathfrak{p}]$ induces a plane field distribution along $\Ep$ that is pointwise tangent to a totally geodesic copy of the symmetric space of $\mathfrak{sl}_2 (\CC)$ (i.e., $\HH^3$).
\begin{proof}
   Pulling back the triplet $(P_G [\mathfrak{g}],\Theta, D^{\vec{\alpha}})$ via $p : \tilde{X} \rightarrow X$ a universal cover and trivializing in such a way that $(p^* P_S [\mathfrak{s}])_{z_0} \cong \mathfrak{s}$ for a basepoint $z_0 \in \tilde{X}$, we get that $p^*\Theta_{z_0}$ preserves $\mathfrak{s}$, therefore belongs to the totally geodesic submanifold defined as:
   $$P_\mathfrak{s} = \{\theta \in \mathfrak{g}: \theta \mathfrak{s} = \mathfrak{s}\},$$
   which integrates $\mathfrak{s}$ at $z_0$, thus proving the claim.
\end{proof}
\end{lemma}

This is an instance of the geometry of the Slodowy slice \cite{bradlow2024general} (see \cite[Section 4]{bronstein2025anosov} for a similar remark in a slightly different setup).

By Lemma \ref{lemma:maurercartan}, we can split the Maurer-Cartan form of the triple $(P_G [\mathfrak{g}], \Theta, D^{\vec{\alpha}})$ as:

\begin{equation}
\label{eq:decompositionomega}
\omega = \omega^{\alpha_1} + \sum_{i=2}^l (H *\alpha_i) \otimes \theta(e_{m_i}) - (\tau \alpha_i) \otimes e_{m_i},
\end{equation}

where $\omega^{\alpha_1} = -2 ((\tau + \frac{1}{2} H *\alpha_1) \otimes f + (H + \frac{1}{2} \tau \alpha_1)  \otimes e)$ is tangent to $P_S [\mathfrak{s}] \cap P_K [\mathfrak{p}]$.

In order to check regularity, we need some control on the Cartan projection $\mu$ of $\omega$. 
The advantage of the previous decomposition and Lemma \ref{rmk:distribution} is that any section $s\in \mathcal{A}^0 (X,P_S [\mathfrak{s}] \cap P_K [\mathfrak{p}])$ will lie in the same $K-$orbit because $\mathbb{H}^3$ is a rank one symmetric space. 
Therefore, we get:

\begin{lemma}
\label{lemma:cartanofomega}
Let $z$ be a local coordinate on $X$, then given $v_\phi = \cos \phi \partial_x + \sin \phi \partial_y \in T_z X$, we get:
$$\mu (\omega^{\alpha_1} (v_\phi)) = \sqrt{2 H(z)} \vline 1 + \frac{e^{-2i\phi} \alpha_1 (z)}{2H(z)} \vline 1 \otimes h,$$
where $H = H(z) dz \overline{dz}$, $\alpha_1 = \alpha_1 (z) dz^2$ are the hyperbolic metric and the differential in coordinates.
\begin{proof}
As previously remarked, Lemma \ref{rmk:distribution} implies that $\kappa (\omega^{\alpha_1} (v_\phi)) = \lambda(\phi) 1 \otimes h$, for $\lambda (\phi) \in \RR$. To figure out the constants, it is enough to compute the norm $\norm{\omega^{\alpha_1} (v_\phi)}$.
Notice that $v_\phi = \frac{1}{2} (e^{-i\phi} \partial_z + e^{i\phi} \partial_{\bar{z}})$, therefore:

$$\omega^{\alpha_1} (v_\phi) = - \left( (e^{-i\phi} + \frac{e^{i\phi} \overline{\alpha_1}}{2H}) \partial_z \otimes f + (e^{i\phi} H + \frac{e^{-i\phi}\alpha_1}{2}) dz \otimes e \right).$$

Taking inner products, we get:

$$\norm{\omega^{\alpha_1} (v_\phi)}^2 = 2 H \left|e^{i\phi} + \frac{e^{-i\phi} \alpha_1}{2H}\right|^2 \kappa (e,f).$$

Dividing by $\norm{1 \otimes h}^2 = \kappa (h,h) = \kappa (e,f)$ and taking square roots, we get the lemma.
\end{proof}
\end{lemma}

Notice that Remark \ref{rmk:orthogonality} implies that terms in the decomposition given by Equation (\ref{eq:decompositionomega}) are pairwise orthogonal.
This fact and Lemma \ref{lemma:cartanofomega} allow us to say when the EP-surfaces $\Ep$ are immersed.

\begin{prop}
The EP-surface associated to an oper given by $(P_G [\mathfrak{g}], P_B[\mathfrak{b}],D^{\vec{\alpha}})$ fails to be immersed precisely when $\norm{\alpha_1}(z) = 2$, and $\alpha_i (z) = 0$ for every $i =2,\ldots,l$.
Here $\norm{\alpha_1}$ denotes the hyperbolic norm.
\begin{proof}
Given $v_\phi \in T_z X$ as in Lemma \ref{lemma:cartanofomega}, we need to show that $\omega (v_\phi) \neq 0$ for every $\phi$.
That will happen as long as at least one term on the decomposition given by Equation (\ref{eq:decompositionomega}) is non-zero.
The term $\left( (H * \alpha_i) \otimes \theta (e_{m_i}) - (\tau \alpha_i) \otimes e_{m_i} \right) (v_\phi)$ will be non-zero as soon as $\alpha_i (z) \neq 0$.
On the other hand, Lemma \ref{lemma:cartanofomega} tells us that $\omega^{\alpha_1} (v) \neq 0$ as soon as:
$$min_{\phi \in [0,2\pi)} \left|1 + \frac{e^{-2i\phi}\alpha_1}{2 H}\right|= 1- \frac{\norm{\alpha_1}}{2} \neq 0.$$
This finishes the proof.
\end{proof}
\end{prop}

\subsection{Proof of Theorem \ref{thm:regularity}}

The sketch of Theorem \ref{thm:regularity} goes as follows: we know that the direction $\omega^{\alpha_1} (v_\phi)$ is regular (if non-zero).
On the other hand, we have the following classical result:

\begin{prop}
\label{prop:cartan}
Let $\mathbb{X}$ be the symmetric space of $G$, and $K_\theta$ the stabilizer of $\theta \in \mathbb{X}$. 
Given $x,y \in T_\theta \mathbb{X}$ unit vectors, define:
\begin{gather*}
F : K_\theta \rightarrow \RR\\
F(k) = \kappa (x,Ad(k)y),    
\end{gather*}
then:
\begin{itemize}
    \item If $k$ is a local maximum for $F$, then $Ad (k) y$ lies in a common Weyl-chamber with $x$.
    \item If $x$ is regular, then $F$ is Morse and has a unique local maximum.
\end{itemize}
\end{prop}

The reader can find a proof in \cite[Proposition 3.5]{riestenberg2024quantified}.
Since $1 \otimes h$ is regular, Proposition \ref{prop:cartan} tells us that the angle between the Cartan projections $\mu (\omega^{\vec{\alpha}_1} (v_\phi))$ and $\mu (\omega (v_\phi))$ is bounded above by the angle between $\omega^{\vec{\alpha}_1} (v_\phi)$ and $\omega (v_\phi)$.
If the last angle is sufficiently small, then the Cartan projection $\mu (\omega (v_\phi))$ will not reach the boundary of the Weyl chamber.

\begin{lemma}
\label{lemma:angles}
Taking coordinates $z$ in $X$, and let $v_\phi = \cos \phi \partial_x + \sin \phi \partial_y$ as before, then: 
$$\cos \angle (\omega (v_\phi), \omega^{\alpha_1} (v_\phi)) = \frac{\left|1+ \frac{e^{-2i\phi} \alpha_1}{2H}\right|}{\sqrt{\left|1+ \frac{e^{-2i\phi} \alpha_1}{2H}\right|^2 + \frac{1}{4}\sum_{i=2}^l \norm{\alpha_i}^2}}.$$
\begin{proof}
Recall that the terms in the decomposition $\omega = \omega^{\alpha_1} + \sum_{i=2}^l (H *\alpha_i) \otimes \theta(e_{m_i}) - (\tau \alpha_i) \otimes e_{m_i}$ are pairwise orthogonal. 
Therefore:
$$\cos \angle (\omega (v_\phi), \omega^{\alpha_1} (v_\phi)) = \frac{\langle \omega (v_\phi), \omega^{\alpha_1} (v_\phi)\rangle}{\norm{\omega(v_\phi)}\norm{\omega^{\alpha_1}(v_\phi)}} = \frac{\norm{\omega^{\alpha_1} (v_\phi)}}{\norm{\omega (v_\phi)}}.$$
We know from Lemma \ref{lemma:cartanofomega} that $\norm{\omega^{\alpha_1}(v_\phi)}= \sqrt{2H \kappa (e,f)} \left|1+\frac{e^{-2i\phi}\alpha_1}{2H}\right|.$
On the other hand, we have:
$$\langle (H  *\alpha_i) \otimes \theta(e_{m_i}) - (\tau \alpha_i) \otimes e_{m_i},(H *\alpha_i) \otimes \theta(e_{m_i}) - (\tau \alpha_i) \otimes e_{m_i}\rangle = 2 H\norm{\alpha_i}^2 \kappa (e,f) dz \overline{dz},$$
for $\norm{\alpha_i}^2 = H^i (\alpha_i,\alpha_i)$.
This shows that:
$$\langle \omega (v_\phi),\omega(v_\phi) \rangle = 2H \kappa (e,f) \left( \left|1 + \frac{e^{-2i\phi}\alpha_1}{2H}\right|^2 + \frac{1}{4} \sum_{i=2}^l \norm{\alpha_i}^2 \right).$$
Which finishes with the proof.
\end{proof}
\end{lemma}

Notice that given $c >0$, the function $x \to \frac{x}{\sqrt{x^2 + c}}$ is monotonic in $x$.
This fact together with Propostion \ref{prop:cartan} imply:

\begin{coro}
In the same setting as in the previous lemma, we get that:
$$\cos \angle (\mu (\omega (v_\phi)), \mu (\omega^{\alpha_1} (v_\phi))) \geq \frac{(1-\frac{\norm{\alpha_1}}{2})}{\sqrt{(1-\frac{\norm{\alpha_1}}{2})^2 + \frac{1}{4} \sum_{i=2}^l \norm{\alpha_i}^2}}.$$
\end{coro}

This completes the proof of Theorem \ref{thm:regularity}:

\begin{proof}[Proof of Theorem \ref{thm:regularity}]
Let $\phi_{min}$ be the minimal angle between $h$ and a wall of the Weyl chamber.
Notice that as long as $\cos \angle (\mu (\omega (v_\phi)), \mu (\omega^{\alpha_1} (v_\phi))) \geq \cos \phi_{min}$, then the EP-surface $\Ep$ will be regular.
The previous corollary shows that this inequality is satisfied when:
$$\frac{(1-\frac{\norm{\alpha_1}}{2})}{\sqrt{(1-\frac{\norm{\alpha_1}}{2})^2 + \frac{1}{4} \sum_{i=2}^l \norm{\alpha_i}^2}} \geq \cos (\phi_{min}).$$
\end{proof}

\section{A quantifiable Anosov neighborhood}
\label{sec:criterion}

The goal of this section is to provide an explicit neighborhood for the zero differential such that the associated opers are $\Delta-$Anosov (see Theorem \ref{thm:megatheorem}).
The extra ingredient we will need is Davalo's work on nearly geodesic immersions \cite{davalo2025nearly}, which we review below.
This will reduce the criterion for Anosovness to having a good control over the second fundamental form of the EP-surface.

As in previous sections, $\mathfrak{g}$ is a simple Lie algebra, and we identify $\op_X (G)$ with $\bigoplus_{i=1}^l H^0 (X,K^{m_i+1})$ via a normalized parametrization.

\subsection{Nearly geodesic immersions}

In this section, we briefly recall the previous work of Davalo \cite{davalo2025nearly}, which will be used to verify the Anosov condition for the holonomy of an oper. 
We will restrict the discussion to our setting of $G$, a simple complex Lie group of adjoint type.

We will use the following notations: given $\xi \in \partial \mathbb{X}$, $b_\xi (\cdot,x_0): \mathbb{X} \rightarrow \mathbb{R}$ denotes the Busemann function at $\xi \in \partial \mathbb{X}$ based at $x_0 \in \mathbb{X}$. 
Given $\tau \subset \partial \mathbb{X}$ a stratum of the visual boundary defined as the $G-$orbit of the direction at infinity of a $v \in T \mathbb{X}$, we define the opposite stratum $\iota \tau$ as the $G-$orbit of the point at infinity defined by $-v \in T \mathbb{X}$ (i.e., the set antipodal to $\tau$).

As mentioned in the introduction, immersions into $\mathbb{H}^3$ with principal curvatures in between $(-1,1)$ are well behaved from a large-scale geometry viewpoint (see \cite[Section 3]{epstein2024envelopes}, \cite[Section 5]{bridgeman2024epstein}).
Motivated by trying to generalize this picture for higher rank, Davalo proposed the following definition:

\begin{defi}
 Given $\tau \subset \partial \mathbb{X}$ a stratum of the visual boundary, an immersion $u : M \rightarrow \mathbb{X}$ is called $\tau\text{-}$nearly geodesic if whenever the differential $d(b_\xi (u(\cdot),x_0))(v)=0$ for some $v \in TM \setminus \{0\}$ and $\xi \in \tau \cup \iota \tau$, then the Hessian $\Hess(b_\xi (u(\cdot),x_0))$ is positive definite. The Hessian is taken with respect to the induced metric $u^* g_\mathbb{X}$.
\end{defi}

If we refer back to Figure \ref{fig:epstein}, this condition means that the surface will stay outside the horosphere tangent to it.

Complete $\tau\text{-}$nearly geodesic immersions $u$ satisfy many desirable properties; for instance, they are always quasi-isometrically embedded into $\mathbb{X}$.
They are also quasi-isometrically embedded in a stronger sense, which we will now describe.

Fix $\mathfrak{h} \subset \mathfrak{g}$ a Cartan subalgebra, and $W \subset \mathfrak{a}$ a Weyl chamber of the span of the coroots, and $\mu : T\mathbb{X} \rightarrow W$ the Cartan projection sending each vector $v \in \mathfrak{p}_\theta$ to the unique element of its $Ad(G)-$orbit in $W$.
Given $\theta, \theta' \in \mathbb{X}$ two Cartan involutions, we get $\theta' = \exp_{\theta} (v)$, for a unique $v \in \mathfrak{p}_\theta$.
We define the \textit{vectorial distance} between $\theta$ and $\theta'$ to be $\vec{d_\mathfrak{a}} (\theta,\theta') = \mu (v)$.

The Anosov condition is a strengthening of a group being quasi-isometrically embedded into $\mathbb{X}$.
We will follow the definitions given in \cite{kapovich2017anosov}, \cite{bochi2019anosov}:

\begin{defi}
Let $\Theta \subset \Delta$ be a subset of the set of simple roots, and $\Gamma$ a group equipped with a word metric $|\cdot |$. 
We will say that a representation $\rho : \Gamma \rightarrow G$ is $\Theta-$Anosov if given $\theta_0 \in \mathbb{X}$ there are $A,B > 0$ such that:
$$\alpha (\vec{d_\mathfrak{a}} (\theta_0,\rho(\gamma) \theta_0)) \geq A |\gamma| - B,\: \forall \alpha \in \Theta, \gamma \in \Gamma.$$ 
\end{defi}

We say that $\Theta \subset \Delta$ is a Weyl orbit of simple roots if it is the intersection of a Weyl orbit of a root with the set of simple roots.
Every Weyl orbit of roots $W \alpha$ has the property that there is a unique $\beta \in W\alpha$ such that its coroot is contained in our fixed Weyl chamber $W$.
We will denote by $h_\Theta$ this coroot, by $\norm{\Theta}$ the norm of any root in the orbit, and by $\tau_\Theta \subset \partial \mathbb{X}$, the stratum at infinity defined as the $G-$orbit of the direction defined by $h_\Theta$.

Using the notion of $\tau\text{-}$nearly geodesic immersions, Davalo was able to give a criterion for $\Theta-$Anosovness:

\begin{thm}[{\cite[Theorem 5.20]{davalo2025nearly}}]
\label{thm:davaloholonomy}
 Let $\Theta \subset \Delta$ be a Weyl orbit of simple roots. 
 Suppose that a representation $\rho : \pi_1 (M) \rightarrow G$ admits an equivariant, $\tau_\Theta-$nearly geodesic immersion $u$, then $\rho$ is $\Theta-$Anosov.
\end{thm}

Moreover, Davalo gave necessary conditions for a surface to be $\tau_\Theta-$nearly geodesic, which generalizes the principal curvature condition in $\mathbb{H}^3$:

\begin{thm}[{\cite[Theorem 5.23]{davalo2025nearly}}]
\label{thm:davalosufficient}
 Given a Weyl orbit of simple roots $\Theta \subset \Delta$ , define:
    $$c_\Theta = \min_{
        \substack{\beta \in \Sigma, \\ \beta (h_\Theta) \neq 0}
        } \frac{|\beta (h_\Theta)|}{2\norm{\Theta}},$$
 for $h_\Theta \in W$ the unique coroot of $\Theta$ in $W$, and $\alpha \in \Theta$.
 Suppose that an immersion $u : M \rightarrow \mathbb{X}$ satisfies:
    $$\norm{\II (v,v)} < c_\Theta |\alpha (\mu (du(v)))|^2,\: \forall v \in TM, \alpha \in \Theta,$$
 where $\II$ is the second fundamental form of $u$, then $u$ is a $\tau_\Theta$ nearly geodesic immersion.
\end{thm}

The combination of the previous couple of results is the crucial ingredient in our Anosov criterion.

When $\mathfrak{g}$ is a simple Lie algebra, we can have at most two Weyl orbits of roots characterized by the set of roots of a given fixed length. 
If there are two, we will call one of them long or short depending on whether the norm $\norm{\alpha}$ for $\alpha \in \Theta$ is the largest/minimal possible in the root system, and denote them as $\Theta_{L}, \Theta_{S}$.
We check in Appendix \ref{appendix:root} that $c_{\Theta_{L}}$ is always smaller that $c_{\Theta_{S}}$.

\subsection{The second fundamental form}

As one can observe in Theorem \ref{thm:davalosufficient}, to check the Anosovness of the oper holonomy, it will be enough to control the second fundamental form of the EP-surface.

Recall that the second fundamental form $\II$ is a symmetric tensor with values in the orthogonal complement to $\omega (TX)\subset P_K [\mathfrak{p}]$ satisfying:

$$\nabla^{LC}_X \omega (Y) = \omega (\nabla^g_X Y) + \II (X,Y),\: \forall X,Y \in \mathcal{A}^0 (X,TX),$$

where $\nabla^g$ is the Levi-Civita connection of the first fundamental form $g = \Ep^* g_\mathbb{X}$.
Pick a local coordinate $z$ for the Riemann surface $X$, then using Equation (\ref{eq:decompositionomega}), $g$ takes the form:

\begin{equation}
\label{eq:metric}
g = 4 \kappa (e,f) \left( \alpha_1 dz^2 + 2H(z) \left( 1 + \sum_{i=1}^l \frac{\norm{\alpha_i}^2}{4}\right)dz \overline{dz} + \overline{\alpha_1 dz^2}\right).
\end{equation}

We abuse notation slightly: $\alpha_1$ denotes the coefficient of the quadratic differential $\alpha_1 = \alpha_1 (z)dz^2$, and $H = H(z)$ is the coefficient of the hyperbolic metric $H(z) dz \overline{dz}$.
The norms $\norm{\alpha_i}^2$ are taken with respect to the metric (i.e., $\norm{\alpha_i}^2 = \frac{\alpha_i (z) \bar{\alpha_i (z)}}{H^{m_i+1} (z)}$).

\begin{rmk}
When $\alpha_1 = 0$ and the conditions of Theorem \ref{thm:regularity} are verified, the surface $\Ep: \tilde{X} \rightarrow \mathbb{X}$ is a conformal immersion.
\end{rmk}

We work in the complexified tangent bundle $TX \otimes \mathbb{C}$, extending all tensors and connections $\mathbb{C}-$linearly.

Since $\nabla^{LC}$, and $\omega$ are explicit, one can straightforwardly compute the second fundamental form $\II$ for the EP-surface. 
However, this expression is complicated in general, so we have to work around it to access the norm of the second fundamental form.

\begin{lemma}
\label{lemma:secondform-general}
   Let $\vec{\alpha} \in \bigoplus_{i=1}^l H^0 (X,K^{m_i +1})$, and $z$ a local chart for $X$. Defining $v_\phi = \cos \phi \partial_x + \sin \phi \partial_y \in \mathcal{A}^0 (X, TX)$, then the second fundamental form of the EP-surface associated to the oper defined by $\vec{\alpha}$ satisfies the following bound:

   \begin{equation*}
   \resizebox{\displaywidth}{!}{$\norm{\II(v_\phi,v_\phi)}^2 \leq \frac{1}{4} \!\left( \!\left( R^{g}_{z\overline{zz}z} \!- \!4 H^2 \kappa (e,f) \!\left( 1 \!-\! \frac{1}{2} \sum_{i=1}^l m_i \norm{\alpha_i}^2 \right) \right)^{\frac{1}{2}} \!+ \!\left( R^{LC}_{z\overline{zz}z} \!- \!4 H^2 \kappa (e,f) \!\left( 1 \!- \!\frac{1}{2} \sum_{i=1}^l m_i \norm{\alpha_i}^2 \right) \right)^{\frac{1}{2}}  \right)^2$,}   
   \end{equation*}
   where $R_{ijkl} = \langle R (\partial_i,\partial_j) \partial_k, \partial_l\rangle$ for $i,j,k,l \in \{z,\overline{z}\}$.
\begin{proof}
   The first part of the estimate is just a series of algebraic manipulations.
   Notice that $v_\phi = \frac{1}{2} (e^{-i\phi}\partial_z + e^{i\phi} \partial_{\overline{z}})$, therefore $\II(v_\phi,v_\phi) = \frac{1}{4} (e^{-2i\phi} \II_{zz} + 2 \II_{z\overline{z}} + e^{2i\phi} \II_{\overline{zz}})$, where $\II_{ij} = \II (\partial_i,\partial_j)$ for $i,j \in \{z,\overline{z}\}$. Since $\nabla^{LC} \Theta = 0$, we get that $\II_{zz} = - \Theta \II_{\overline{zz}}$, and $\II_{z\overline{z}} = - \Theta \II_{z\overline{z}}$.

   The Killing form $\kappa$ verifies that $\kappa (\theta \cdot, \theta \cdot) = \overline{\kappa (\cdot,\cdot)}$.
   This fact and the previous paragraph imply that:

   \begin{align*}
   \norm{\II (v_\phi,v_\phi)}^2 \!&= \!\frac{1}{16} \!\left( 2 Re (e^{-4i\phi} \!\langle \II_{zz}, \II_{zz}\rangle) \!+ \!2 \langle \II_{zz},\II_{\overline{zz}} \rangle \!+ \!4 \langle \II_{z\overline{z}},\II_{z\overline{z}}\rangle \!+ \!8 Re (e^{-2i\phi} \langle \II_{zz}, \II_{z\overline{z}}\rangle) \right)\\
   &\leq \frac{1}{4} \left( \langle \II_{zz},\II_{\overline{zz}}\rangle + 2 Re (e^{-2i\phi} \langle \II_{zz},\II_{z\overline{z}} \rangle) + \langle \II_{z\overline{z}} , \II_{z\overline{z}} \rangle \right),
   \end{align*}

   where we used that $Re (\kappa (x,x)) \leq -\kappa (x,\theta x)$ for any $x \in \mathfrak{g}$ applied to $\II_{zz}$.
   Applying Cauchy-Schwarz inequality with $\langle \II_{zz},\II_{z\overline{z}} \rangle$, we get that the norm of this term is upper bounded by $\langle \II_{zz},\II_{\overline{zz}}\rangle^{\frac{1}{2}} \langle \II_{z\overline{z}},\II_{z\overline{z}} \rangle^{\frac{1}{2}}$, therefore:

   $$\norm{\II (v_\phi,v_\phi)}^2 \leq \frac{1}{4} (\langle \II_{zz},\II_{\overline{zz}}\rangle^{\frac{1}{2}} + \langle \II_{z\overline{z}},\II_{z\overline{z}} \rangle^{\frac{1}{2}})^2.$$

   The term $\II_{z\overline{z}}$ appearing on the right-hand side has a simple expression.
   Holomorphicity of $\alpha_i$ implies that:

   \begin{align*}
   \nabla_{\partial_{\bar{z}}}^{LC} &\omega (\partial_z) = -2 \left( \nabla_{\partial_{\bar{z}}} +ad \left( \frac{1}{2}  \sum_{i=1}^l H (\cdot,\partial_z) *\alpha_i \otimes \theta e_{m_i}\right) \right) \left( \partial_z \otimes f + \frac{1}{2} \sum_{i=1}^l \alpha_i (\partial_z,\cdot) \otimes e_{m_i} \right) \\
   &= -\frac{1}{2} \sum_{i,j=1}^l H(z) *\alpha_i \alpha_j \otimes [\theta (e_{m_i}), e_{m_j}].
   \end{align*}

   Since $\kappa (e_{m_i}, [e_{m_k},\theta e_{m_j}]) = 0$ for every $i,j,k =1,\ldots,l$, we get that $\II_{z\bar{z}} = \nabla_{\partial_z} \omega (\partial_{\bar{z}})$.
   Thus, one gets $\II_{z\overline{z}} = \frac{1}{2} [\omega (\partial_{\bar{z}}),\omega (\partial_z)] + [H(\cdot,\partial_z)\otimes e,\omega (\partial_z)]$, where $[H(\cdot,\partial_z)\otimes e,\omega (\partial_z)] = -2 H (z) \otimes h$.
   Using the equation $R^{LC} (\partial_z,\partial_{\bar{z}}) = - \frac{1}{4} ad ([\omega (\partial_z),\omega (\partial_{\overline{z}})])$ and doing standard manipulations, we get:

   \begin{equation}
   \label{eq:II_zbarz}
   \langle \II_{z\overline{z}},\II_{z\overline{z}} \rangle = R^{LC}_{z \overline{zz}z}-4H^2 \kappa (e,f) \left( 1 - \frac{1}{2} \sum_{i=1}^l m_i \norm{\alpha_i}^2 \right).
   \end{equation}

   We can avoid an explicit computation of $\II_{zz}$ by using the Gauss equation (see \cite[Vol. 2, Proposition 2.1]{Kobayashi1963Foundations}):
   
   \begin{equation}
   \label{eq:gauss}
   R^{LC}_{z\overline{z}z\overline{z}} = R^g_{z\overline{z}z\overline{z}}
   + \langle \II_{zz},\II_{\overline{zz}} \rangle - \langle \II_{z\overline{z}}, \II_{z\overline{z}}\rangle.   
   \end{equation}
   
   Plugging the expression for $\langle \II_{z\overline{z}},\II_{z\overline{z}} \rangle$ into the Gauss equation, and substituting in our original inequality, we get our estimate.
\end{proof}
\end{lemma}

One can observe that Lemma \ref{lemma:simplifications} below implies that this inequality is never sharp.
We set up the computation to avoid using the explicit expression for $\II_{zz}$.
However, in specific instances of opers, we can produce better upper bounds for $\norm{\II (v_\phi,v_\phi)}$.

Equations (\ref{eq:levicivita}) and (\ref{eq:decompositionomega}) imply that:

\begin{align*}
\nabla_{\partial_z}^{LC} &\omega (\partial_z) = -2 \left( \nabla_{\partial_z} + ad \left( \frac{1}{2}\sum_{i=1}^l \alpha_i (\partial_z,\cdot) \otimes e_{m_i}\right) \right) \left( \partial_z \otimes f + \frac{1}{2} \sum_{i=1}^l \alpha_i (\partial_z,\cdot) \otimes e_{m_i} \right)\\
&= -2 \left( \nabla_{\partial_z}^{-1} \partial_z \otimes f + \frac{1}{2} \sum_{i=1}^l \nabla^{m_i}_{\partial_z} (\alpha_i (\partial_z,\cdot)) \otimes e_{m_i} + \frac{1}{2} \sum_{i=1}^l \alpha_i (\partial_z,\partial_z,\cdot) \otimes [e_{m_i},f]\right).
\end{align*}

To compute $\II_{zz}$, we must subtract the Levi-Civita connection of the metric $g = \Ep^* g_{\mathbb{X}}$ that can be computed using Lemma \ref{lemma:koszul} found in the appendix.
Straightforward computations show:

\begin{lemma}
\label{lemma:simplifications}
With the same setup as before, we get that:
\begin{itemize}   
\item The term $\langle \II_{zz},\II_{zz}\rangle$ is zero iff $\alpha_1 = 0$.
\item The term $\langle \II_{zz},\II_{z \overline{z}} \rangle$ is equal to:
$$ \langle \nabla_{\partial_z} \omega (\partial_z) , \nabla_{\partial_z} \omega (\partial_{\bar{z}}) \rangle = \frac{1}{2} \sum_{i,j,k} H \alpha_i * \alpha_j \alpha_k (\partial_z,\partial_z) \kappa ([e_{m_i},f],[\theta e_{m_j},e_{m_k}]),$$
where we sum indices satisfying $m_i -1 = m_j- m_k$.
\end{itemize}   
\end{lemma}

This Lemma helps us detect situations for which both terms $\langle \II_{zz},\II_{zz}\rangle$ and $\langle \II_{zz},\II_{z \overline{z}} \rangle$ vanish. 
This occurs for instance when $\vec{\alpha}$ has only one non-zero component in some $H^0 (X,K^{m_i+1})$ factor with $i >1$.

This remark allows us to improve Lemma \ref{lemma:secondform-general} when we restrict to the study of cyclic opers (Definition \ref{defi:cyclicoper}):

\begin{lemma}
\label{lemma:secondform-cyclic}
   Let $\mathfrak{g} \neq \mathfrak{sl}_2 (\mathbb{C})$, and $\alpha_l \in H^0 (X,K^{m_l +1})$ defining a cyclic $G-$oper.
   Using the same notations as Lemma \ref{lemma:secondform-general}, we get that:
   $$\norm{\II (v_\phi,v_\phi)}^2 \!= \frac{1}{8} \left( R^{LC}_{z\overline{z}\overline{z}z} + 2 R^{g}_{z\overline{zz}z} - 12 H^2 \!\kappa (e\!,\!f) \!\left(\!1 \!- \!\frac{1}{2} m_l \norm{\alpha_l}^2\!\right) \right),$$
   independent of $\phi$.
\begin{proof}
   Repeating the proof of Lemma \ref{lemma:secondform-general}, we get that:
   $$\norm{\II (v_\phi,v_\phi)}^2 = \frac{1}{16} (2 \langle \II_{zz},\II_{\overline{zz}} \rangle + 4 \langle \II_{z\overline{z}}, \II_{z\overline{z}}\rangle),$$
   because of Lemma \ref{lemma:simplifications}.
   Using Equations (\ref{eq:II_zbarz}) and (\ref{eq:gauss}), we get the desired equality.
\end{proof}
\end{lemma}

\begin{rmk}
\label{rmk:sl2}
   When $\mathfrak{g} = \mathfrak{sl}_2 (\mathbb{C})$ the situation is special because $\omega (TX)$ coincides with the span of $\{ \partial_z \otimes f, dz \otimes e\}$, therefore the expression for $\II_{zz}$ simplifies considerably and one gets that:
   $$\II = - \left( \alpha_1 dz^2 -  H\norm{\alpha_1}^2 dz \overline{dz} + \overline{\alpha_1 dz^2}\right) \otimes h,$$
   depends solely on the quadratic differential and not on its derivatives. 
   The bounds we will develop aim for the higher rank case and are not suitable for this situation, since Lemma \ref{lemma:secondform-general} is forgetting terms that involve the quadratic differential $\alpha_1$.
\end{rmk}

The last couple of lemmas reduce the upper bound for the second fundamental form to an upper bound of curvature tensors. 

Notice that the term $R^{LC}_{z\overline{zz}z}$ is equal to $\frac{1}{4} \norm{[\omega (\partial_z),\omega (\partial_{\bar{z}})]}^2$. A straightforward computation shows:

\begin{equation}
\label{eq:RLC}
   R^{LC}_{z\overline{zz}z} = \!4 H^2 \kappa (e,f) \!\left( 1 \! - \!\frac{1}{2} \sum_{i=1}^l \! m_i\norm{\alpha_i}^2  \!+ \!\frac{1}{16} \sum_{i,j,k,l} \!\alpha_i \alpha_k \!*\!\alpha_j \!* \!\alpha_l c_{ijkl} \!\right),
\end{equation}

where we define $c_{ijkl} = \frac{\kappa ([e_{m_i},\theta e_{m_j}],[e_{m_k},\theta e_{m_l}])}{\kappa (e,f)}$. 
Notice that $c_{ijkl}$ can only be non-zero in indices satisfying that $m_i + m_k = m_j + m_l$.

The $R^g$ term is more complicated in the sense that it depends on derivatives of the differentials defining the opers.
An upper bound for $R^g$ can be deduced by using Proposition \ref{coro:curvature} whose proof we delay to Appendix \ref{subsec:computations1}:

\begin{lemma}
   The term $R^{g}_{z\overline{zz}z}$ is bounded above by:
   $$R^{g}_{z\overline{zz}z} \leq 4 \kappa (e,f)H^2 \left(1 - \frac{1}{4} \sum_{i=1}^l m_i \norm{\alpha_i}^2 + \frac{1}{4} \frac{\sum_{i=1}^l \norm{\nabla \alpha_i}^2 + \frac{1}{4} \sum_{i<j} \norm{\nabla \alpha_i \alpha_j - \alpha_i \nabla \alpha_j}^2}{1 + \frac{1}{4} \sum_{i=1}^l \norm{\alpha_i}^2} \right),$$
   where we think of $\nabla \alpha_i$ as an $(m_i+2)-$differential.
   We get equality whenever $\alpha_1 =0$.
\begin{proof}
   Notice that $R^{g}_{z\overline{zz}z} = K^g \det (g_{\mathbb{C}})$, where $K^g$ is the Gaussian curvature and $\det (g_{\mathbb{C}})$ is the determinant of the complexified metric.
   Using Equation (\ref{eq:metric}), we get $\det (g_{\mathbb{C}}) = (4\kappa (e,f)H)^2 (\norm{\alpha_1}^2 - (1 + \frac{1}{4} \sum_{i=1}^l \norm{\alpha_i}^2 )^2)<0$. 

   To get an upper bound, it is enough to give a lower bound for $K^g$ (since both $K^g$ and $\det (g_{\mathbb{C}})$ are negative). 
   Proposition \ref{coro:curvature} and Equation (\ref{eq:Liouville}) imply that:

   \begin{equation*}
   K^g \geq \frac{2 \kappa (e,f)H^2}{\det (g_{\mathbb{C}})}\left( 1 + \frac{1}{4} \sum_{i=1}^l \norm{\alpha_i}^2\right)\left( 2 + \frac{1}{2}\Delta (\log (1 + \frac{1}{4} \sum_{i=1}^l \norm{\alpha_i}^2))\right), 
   \end{equation*}
   
   where $\Delta = \frac{4}{H} \partial_z \partial_{\bar{z}}$ is the hyperbolic laplacian.
   This is an exact equality if $\alpha_1 =0$.

   Recall that given $f$ a positive function, $\Delta \log f = \frac{\Delta f}{f} - \frac{\norm{\nabla f}^2}{f^2}$, where $\nabla f$ is the hyperbolic gradient of the function.
   This equality, together with the Bochner identity (Equation (\ref{eq:Bochner})) implies that:

   \begin{equation*}
   \resizebox{\displaywidth}{!}{$\frac{1}{2} \Delta (\log (1 + \frac{1}{4} \sum_{i=1}^l \norm{\alpha_i}^2)) = \frac{- \frac{1}{2} \sum_{i=1}^l (m_i +1) \norm{\alpha_i}^2 + \frac{1}{2} \frac{\sum_{i=1}^l \norm{\nabla \alpha_i}^2 + \frac{1}{4} \left( \sum_{i,j} \norm{\nabla \alpha_i}^2 \norm{\alpha_j}^2 - \frac{1}{4} \langle \nabla \norm{\alpha_i}^2,\nabla \norm{\alpha_j}^2 \rangle \right)}{1 + \frac{1}{4}\sum_{i=1}^l \norm{\alpha_i}^2}}{1 + \frac{1}{4}\sum_{i=1}^l \norm{\alpha_i}^2} .$}
   \end{equation*}

   One can check that $\langle \nabla \norm{\alpha_i}^2, \nabla \norm{\alpha_j}^2 \rangle = 2 \left( \alpha_i \nabla \alpha_j *\alpha_j * \nabla \alpha_i + \alpha_j \nabla \alpha_i * \alpha_i * \nabla \alpha_j \right)$, which implies:
   $$\norm{\nabla \alpha_i}^2 \norm{\alpha_j}^2 + \norm{\nabla \alpha_j}^2 \norm{\alpha_i}^2 - \frac{1}{2}\langle \nabla \norm{\alpha_i}^2,\nabla \norm{\alpha_j}^2\rangle = \norm{\alpha_i \nabla \alpha_j - \alpha_j \nabla \alpha_i}^2.$$
   
   Plugging this into the lower bound for $K^g$ yields the desired expression.
\end{proof}
\end{lemma}

Notice that in the previous proof, we only disregard terms that depend on $\alpha_1 \in H^0 (X,K^2)$.
When $G = PGL_2 (\mathbb{C})$, these are precisely the terms that cancel out the dependence on $\nabla \alpha_1$.
In higher rank, these terms are not enough to cancel out the dependence on derivatives of the differentials.
An exact expression for the curvature can be computed in the general case, but we avoid it since it will not improve our results much.

The previous lemma together with Equation (\ref{eq:RLC}) and Lemma \ref{lemma:secondform-general} imply:

\begin{coro}
\label{coro:estimatesecondfundamental-general}
   Given $\vec{\alpha} \in \bigoplus_{i=1}^l H^0 (X,K^{m_i+1})$ defining a $G-$oper, we get that the second fundamental form of the associated EP-surface is bounded above by:
   
   \begin{equation*}
   \resizebox{\displaywidth}{!}{$\norm{\II (v_\phi,v_\phi)}^2 \leq \frac{H^2 \kappa (e,f)}{4} \left(  \left(\sum_{i=1}^l m_i \norm{\alpha_i}^2 + \frac{\sum_{i=1}^l \norm{\nabla \alpha_i}^2 + \frac{1}{4}\sum_{i<j} \norm{\nabla \alpha_i \alpha_j - \alpha_i \nabla \alpha_j}^2}{1 + \frac{1}{4} \sum_{i=1}^l \norm{\alpha_i}^2} \right)^{\frac{1}{2}} + \frac{1}{2}
    \left( \sum_{i,j,k,l} \alpha_i \alpha_k *\alpha_j *\alpha_l c_{ijkl} \right)^{\frac{1}{2}} \right)^2,$}
   \end{equation*}

   where $c_{ijkl} = \frac{\kappa ([e_{m_i},\theta e_{m_j}],[e_{m_k},\theta e_{m_l}])}{\kappa (e,f)}$.
\end{coro}

For cyclic opers, we can say something stronger by inserting Equation (\ref{eq:RLC}) and the previous lemma into Lemma \ref{lemma:secondform-cyclic}:

\begin{coro}
\label{coro:estimatedsecondfundamental-cyclic}
   Let $\alpha \in H^0 (X,K^{m_l+1})$ define a cyclic $G-$oper, then the norm of the second fundamental form of the associated EP-surface is given by:
   $$\norm{\II (v_\phi,v_\phi)}^2 = \frac{H^2 \kappa (e,f)}{16} \left(2m_l \norm{\alpha}^2 + 2\frac{\norm{\nabla \alpha}^2}{1 + \frac{1}{4} \norm{\alpha}^2} + \norm{\alpha}^4 c_{llll} \right).$$
\end{coro}

\begin{rmk}
   As mentioned in Remark \ref{rmk:highestroot}, $e_{m_l}$ lies in $\mathfrak{g}_\theta$ the root space associated to the highest root $\theta$ of $\mathfrak{g}$.
   One can check that $c_{llll} = \frac{m_l^2 \kappa (h,h)}{\sin^2 (\phi_\theta)}$, where $\phi_\theta$ is the angle between $h$ and the $\ker \theta$.
\end{rmk}

\subsection{The criteria}

In this section, we will combine all the ingredients to obtain our desired effective Anosov neighborhood.
Our goal is to detect conditions ensuring that our EP-surface is a $\tau_\Theta-$nearly geodesic immersion via Theorem \ref{thm:davalosufficient}, where $\Theta$ is a Weyl-orbit of simple roots.
Once we have this, Theorem \ref{thm:davaloholonomy} implies that under these conditions the holonomy of the oper is $\Theta-$Anosov.

The only missing ingredient is how to control the Cartan projection of the derivative of the EP-surface.
For that, we use the estimates found in Section \ref{sec:regularity}.

\begin{lemma}
 Let $\beta \in \Delta$ be a simple root of $\mathfrak{g}$ such that $\ker \beta$ forms an angle $\phi_\beta$ with the ray generated by $h$, $\vec{\alpha} \in \bigoplus_{i=1}^l H^0 (X,K^{m_i})$ be differentials such that 
 $$\frac{1- \frac{\norm{\alpha_1}}{2}}{\sqrt{(1 - \frac{\norm{\alpha_1}}{2})^2 + \frac{1}{4} \sum_{i=2}^l \norm{\alpha_i}^2}} \geq \cos (\phi_{min}),$$
  as in Theorem \ref{thm:regularity}. Choose local coordinates $z$, and let $v_\phi = \cos \phi \partial_x + \sin \phi \partial_y\in \mathcal{A}^0(X,TX)$. 
 Then we have: 
    $$\beta (\mu (\omega (v_\phi)))^2 \geq 2 H \kappa (e,f) \norm{\beta} \left( \cos \phi_\beta \left| 1 + \frac{e^{-2i\phi} \alpha_1}{2H} \right| - \sin \phi_\beta \left( \frac{1}{4} \sum_{i=2}^l \norm{\alpha_i}^2\right)^{\frac{1}{2}} \right)^2.$$ 
\begin{proof}
Notice that $|\beta (\mu (\omega (v_\phi)))| = \norm{\beta} \norm{\omega (v_\phi)} \sin (\angle (\mu(\omega (v_\phi)),\ker \beta))$, where \\ $\angle (\mu(\omega (v_\phi)),\ker \beta)$ denotes the angle between $\ker \beta$ and the Cartan projection of $\omega (v_\phi)$.

The triangle inequality implies that $\angle (\mu(\omega (v_\phi)),\ker \beta) \geq \phi_\beta - \angle (\mu (\omega (v_\phi)),h)$, therefore $\sin \angle (\mu(\omega (v_\phi)),\ker \beta) \geq \cos \angle (\mu (\omega (v_\phi)),h) \sin \phi_{\beta} - \sin \angle (\mu (\omega (v_\phi)),h) \cos \phi_{\beta}$.
On the other hand, Lemma \ref{lemma:angles} implies that $\norm{\omega (v_\phi)} \cos \angle (\mu (\omega (v_\phi)),h) \geq \sqrt{2 H \kappa (e,f)} \vert 1 + \frac{e^{-2i\phi}\alpha_1}{2H} \vert$, and $\sin \angle (\mu (\omega (v_\phi)),h)  \leq \sqrt{2 H \kappa (e,f)} \left( \frac{1}{4} \sum_{i=2}^l \norm{\alpha_i}^2 \right)^{\frac{1}{2}}$.
Putting everything together, we get the desired result.
\end{proof}
\end{lemma}

\begin{rmk}
 Notice that $\frac{1- \frac{\norm{\alpha_1}}{2}}{\sqrt{(1 - \frac{\norm{\alpha_1}}{2})^2 + \frac{1}{4} \sum_{i=2}^l \norm{\alpha_i}^2}} \geq \cos (\phi_{min})$ ensures that the right-hand side is positive.
\end{rmk}

Now we are ready to apply Theorem \ref{thm:davalosufficient}.
As mentioned before we always get that $c_{\Theta_{L}} \leq c_{\Theta_{S}}$ (in case there are Weyl orbits of roots), and one can verify that $c_{\Theta_{L}} \norm{\beta} \geq \frac{1}{2}$ for any simple root $\beta$. 
This shows that to check Anosovness with respect to the set $\Delta$ of simple roots, it is enough to verify that:

\begin{equation}
   \norm{\II (v_\phi,v_\phi)} \leq H \kappa (e,f) \left( \cos \phi_{\Theta_{S}} \vert 1 - \frac{\norm{\alpha_1}}{2} \vert - \sin \phi_{\Theta_{S}} \left(\frac{1}{4} \sum_{i=2}^l \norm{\alpha_i}^2 \right)^{\frac{1}{2}} \right)^2.
\end{equation}

With this in hand, we can produce our desired criteria for Anosovness.
For general opers, Corollary \ref{coro:estimatesecondfundamental-general} tells us that whenever the conditions of the previous lemma are satisfied, and:

\begin{equation}
\label{eq:criteriageneral}
\begin{split}
    &\Bigg( \sum_{i=1}^l m_i \norm{\alpha_i}^2 + \frac{\sum_{i=1}^l \norm{\nabla \alpha_i}^2 + \frac{1}{4} \sum_{i<j} \norm{\nabla \alpha_i \alpha_j - \alpha_i \nabla \alpha_j}^2}{1 + \frac{1}{4} \sum_{i=1}^l \norm{\alpha_i}^2} \Bigg)^{\frac{1}{2}} \\
    &\quad + \frac{1}{2} \left( \sum_{i,j,k,l} c_{ijkl} \, \alpha_i \alpha_k *\alpha_j *\alpha_l \right)^{\frac{1}{2}} \\
    &\quad \leq  2 \sqrt{\kappa (e,f)} \left( \cos \phi_{\Theta_{S}} \bigg\vert 1 - \frac{\norm{\alpha_1}}{2} \bigg\vert - \sin \phi_{\Theta_{S}} \left(\frac{1}{4} \sum_{i=2}^l \norm{\alpha_i}^2 \right)^{\frac{1}{2}} \right)^2
\end{split}
\end{equation}

then the holonomy of the oper will be $\Delta-$Anosov.

Similarly, when $\mathfrak{g} \neq \mathfrak{sl}_2 (\mathbb{C})$ and the oper is cyclic and defined by $\alpha \in H^0 (X,K^{m_l})$, Corollary \ref{coro:estimatedsecondfundamental-cyclic} gives us that whenever the angle condition $\frac{1}{1 + \frac{\norm{\alpha}^2}{4}} \geq \cos \phi_{min}$ is satisfied and:

\begin{equation}
\label{eq:criteriacyclic}
2 m_l \norm{\alpha}^2 + 2 \frac{\norm {\nabla \alpha}^2}{1 + \frac{1}{4} \norm{\alpha}^2} + \norm{\alpha}^4 c_{llll} \leq 16 \kappa (e,f) \left( \cos (\phi_{\Theta_{S}}) - \sin \phi_{\Theta_{S}} \frac{\norm {\alpha}}{2}\right)^4,
\end{equation}

then the holonomy will be $\Delta-$Anosov.
This is a sharper criterion, and our techniques do not allow us to do any better.

Summarizing, we get:

\begin{thm}
\label{thm:megatheorem}
   Let $\mathfrak{g}$ be a simple Lie algebra, and identify $\op_X (G) \cong \bigoplus_{i=1}^l H^0 (X,K^{m_i+1})$ via a normalized parametrization.
   If we pick a differential in:
   \begin{description}
      \item[General case] The connected component of differentials containing $\vec{0}$ satisfying Equation \ref{eq:criteriageneral}. 
      \item[Cyclic case] The connected component of $\alpha \in H^0 (X,K^{m_l+1})$ containing $0$ and  verifying Equation \ref{eq:criteriacyclic}. 
   \end{description}
   Then the holonomy of the corresponding oper will be $\Delta-$Anosov.
\begin{proof}
   The result follows from noticing that if $\vec{\alpha}$ is in the connected component of $\vec{0}$ of differentials satisfying either of those equations, we automatically get $\frac{1- \frac{\norm{\alpha_1}}{2}}{\sqrt{(1 - \frac{\norm{\alpha_1}}{2})^2 + \frac{1}{4} \sum_{i=2}^l \norm{\alpha_i}^2}} \geq \cos (\phi_{min})$. 
   If the angle condition fails, the right-hand side of the inequality becomes zero, forcing $\norm{\alpha_i}=0$ for all $i$. Since $0$ satisfies the strict inequality and we are in the same connected component, continuity ensures the angle condition holds throughout.
\end{proof}
\end{thm}

\begin{rmk}
   Standard inequalities show that:
   \begin{itemize}
      \item $\frac{\sum_{i=1}^l \norm{\nabla \alpha_i}^2 + \frac{1}{4}\sum_{i<j} \norm{\nabla \alpha_i \alpha_j - \alpha_i \nabla \alpha_j}^2}{1 + \frac{1}{4} \sum_{i=1}^l \norm{\alpha_i}^2} \leq \sum_{i=1}^l \norm{\nabla \alpha_i}^2$.
      \item $(\sum_{i,j,k,l} \alpha_i \alpha_k * \alpha_j * \alpha_l c_{ijkl})^{\frac{1}{2}}\leq l \max_{ijkl} c_{ijkl}^{\frac{1}{2}} (\sum_{i=1}^l \norm{\alpha_i}^2).$
      \item Under the $\frac{1- \frac{\norm{\alpha_1}}{2}}{\sqrt{(1 - \frac{\norm{\alpha_1}}{2})^2 + \frac{1}{4} \sum_{i=2}^l \norm{\alpha_i}^2}} \geq \cos (\phi_{min})$ condition, we get that: 
\begin{equation*}
\resizebox{\displaywidth}{!}{$\cos \phi_{\Theta_{S}} \vert 1 - \frac{\norm{\alpha_1}}{2} \vert - \sin \phi_{\Theta_{S}} \left(\frac{1}{4} \sum_{i=2}^l \norm{\alpha_i}^2 \right)^{\frac{1}{2}} \geq \cos (2 \phi_{\Theta_{S}})\left( (1-\frac{\norm{\alpha_1}}{2})^2 + \frac{1}{4}\sum_{i=2}^l \norm{\alpha_i}^2 \right)^{\frac{1}{2}}.$}
\end{equation*} 
   The right-hand side is always positive (see the table in Appendix \ref{appendix:root}).
   \end{itemize}
\end{rmk}

Using these remarks, one can simplify Equation \ref{eq:criteriageneral} to obtain user-friendly criteria, as in Theorem \ref{thm:maingeneral}.
Moreover, Cauchy estimates imply Corollary \ref{coro:cauchy}:

\begin{proof}[Proof of Corollary \ref{coro:cauchy}]
   Using the previous remark, Theorem \ref{thm:maingeneral}, and the estimate $\norm{\nabla \alpha} \leq C(X) \norm{\alpha}$, the condition:
   $$2m_l \norm{\alpha}^2 + 2 C(X) \norm{\alpha}^2 + c_l \norm{\alpha}^4 \leq 16 \kappa (e,f) \cos^4 (2\phi_{\theta}),$$
   induces $\Delta-$Anosov opers.
   One can solve this biquadratic inequality by hand to obtain the desired result.
\end{proof}

\section{Domains of discontinuity and extension to the boundary}
\label{subsec:domains}

Let $(\mathcal{D},\rho)$ be a $G-$oper, and $Ep$ be a the associated EP-surface.
A consequence of Davalo's work is that whenever $Ep$ is a $\tau_\Theta-$nearly geodesic immersion, there is a constant $c>0$ such that $\exp (cb_\xi (Ep(\cdot),\theta_0))$ is strictly convex for $b_\xi (\cdot,\theta_0)$ any Busemann function centered at $\xi \in \tau_\Theta$.

Consider:

$$\Omega_{(\mathcal{D},\rho)} =\{ \xi \in \tau_\Theta : b_\xi (Ep (\cdot),\theta_0) \text{ is proper and bounded below}\}.$$

one can prove that $\Omega_{(\mathcal{D},\rho)}$ is an open domain of discontinuity for $\rho$, and the quotient $\Omega_{(\mathcal{D},\rho)}/\rho (\pi_1 (X))$ is compact, see \cite[Theorem 7.8]{davalo2025nearly} (this result was later generalized in \cite[Theorem 1.11]{davalo2024finite}).
Moreover, one can define a (smooth) map $\pi: \Omega_{(\mathcal{D},\rho)} \rightarrow \tilde{X}$ by mapping $\xi$ to the unique $z \in \tilde{X}$ for which $b_\xi (Ep(\cdot),\theta_0)$ has its unique minimum at $z$. 
This map is well-defined by the strict convexity property mentioned above.

Recall from Section \ref{subsection:complexasreal} that the stratum $\tau_\Theta$ is canonically isomorphic to a partial flag manifold $\mathcal{P}_\Theta$ of Lie algebras of parabolics stabilizing points in $\tau_\Theta$.
The parabolic defined by $h_\Theta$ belongs to $\mathcal{P}_\Theta$ by definition.
Notice that $h_\Theta$ is the coroot associated with the largest root of the Weyl orbit.
When $G$ is simple, one can easily compute these vectors for the short and long Weyl orbits of roots, and thus compute the class of the parabolic subgroup.

\begin{rmk}
\label{rmk:flagsl_3}
When $G = PGL_n (\mathbb{C})$, the partial flag manifold $\mathcal{P}_\Theta$ is identified with the space:
$$\mathcal{F}_{1,n-1} = \{ (p,H) \in \mathbb{CP}^{n-1} \times (\mathbb{CP}^{n-1})^*: p \in H\}.$$
Notice that in $PGL_3 (\mathbb{C})$, $\mathcal{P}$ is just the full flag manifold.
\end{rmk}

There is a natural projection $p_\Theta : \mathcal{B} \rightarrow \mathcal{P}_\Theta$ mapping $\mathfrak{b}$ to the unique parabolic in $\mathcal{P}_\Theta$ containing $\mathfrak{b}$ (see \cite[Section 30]{humphreys2012linear}).
More geometrically, if we identify $\mathcal{B}$ with the $G$-orbit of an apartment at infinity, the projection maps each apartment to the unique facet whose stabilizer is in $\mathcal{P}_\Theta$.

Composing $p_\Theta \circ \mathcal{D}$, we get a holomorphic map to $\mathcal{P}_\Theta$ which we now verify has image in the domain of discontinuity $\Omega_{(\mathcal{D},\rho)}$:

\begin{proof}[Proof of Theorem \ref{thm:domaindiscontinuity}]
 Recall from Lemma \ref{lemma:normalweyl} that we have a distribution of Weyl chambers normal to our EP-surface, whose apartments at infinity are defined by the developing map $\mathcal{D}$.
 In particular, the Weyl chambers tipped at $Ep(z)$ contain a unique vector $h(z)$ pointing towards $\xi (z)$ in the stratum $\tau_\Theta$.

 Recall that the gradient $\nabla b_\xi (\cdot,\theta_0)$ at $\theta$ is given by the unique unit-length vector in $\mathfrak{p}_\theta$ pointing towards $\xi$ (see for instance \cite[Proposition 4.3]{davalo2025nearly}).
 Particularly, $\nabla b_{\xi (z)}$ at $Ep (z)$ is orthogonal to the EP-surface, therefore $b_{\xi (z)} (Ep (\cdot),\theta_0)$ has a critical point at $z$.
 Since $\exp (c b_{\xi(z)} (Ep(\cdot),\theta_0))$ is strictly convex, the Busemann function must have a unique minimum realized at $z$.
   
 This implies that $\xi (z) \in \Omega_{(\mathcal{D},\rho)}$ for every $z \in \tilde{X}$.
 By construction, $\xi (z)$ is associated to the parabolic $\pi \circ \mathcal{D}(z)$.
 Therefore $\pi \circ (p_\Theta \circ \mathcal{D}) : \tilde{X} \rightarrow \tilde{X}$ is the identity map.
\end{proof}

As a consequence, we get that the developing map $\mathcal{D}: \tilde{X} \rightarrow \mathcal{B}$ is a holomorphic embedding.
Notice that $\rho$ acts properly discontinuously along $\mathcal{D}(\tilde{X})\subset \mathcal{B}$.
However, it is not clear that it is contained in an \emph{open} domain of discontinuity.

We can observe:

\begin{lemma}
 To impose $\mathcal{D} (\tilde{X})$ to be contained in a cocompact domain of discontinuity of $\mathcal{B}$ is an open condition in the space of quasi-Hitchin opers.
   \begin{proof}
 Stecker proved that any cocompact domain of discontinuity in the full flag. manifold comes from a balanced ideal construction (see \cite[Theorem 1.2]{stecker2024balanced}), therefore, we can assume that $\mathcal{D} (\tilde{X})$ is contained in a connected component of $\mathcal{B} \setminus Th (\Lambda_\rho)$, where $Th (\Lambda_\rho)$ is a thickening of the limit set of the holonomy (see \cite[Theorem 6.13]{kapovich2017dynamics}).
      
 Since the thickening is closed (hence compact), choosing a fundamental domain $K \subset \tilde{X}$ for the deck action, we get that $\mathcal{D} (K)$ is a uniform distance away from the thickening (with respect to some Riemannian metric in $\mathcal{B}$).
 The limit set of an Anosov representation varies continuously (in the Hausdorff topology) with respect to the holonomy. 
The same happens to its thickening, therefore any oper $(\mathcal{D}',\rho')$ nearby will satisfy that $\mathcal{D}' (K) \cap Th (\Lambda_{\rho'}) = \emptyset$ as well.
 This finishes the proof.
   \end{proof}
\end{lemma}

Therefore, to answer the first part of Question \ref{ques:extension}, it is enough to show the closedness of this property.

For a $G-$Fuchsian oper (or more generally, any oper coming from a quasi-Fuchsian projective structure), one can observe that the oper map $\mathcal{D}$ can be extended continuously to $\partial \mathbb{H}^2$ as the limit map of the Anosov representation $\xi: \partial \mathbb{H}^2 \rightarrow \mathcal{B}$ (where we identify $\tilde{X}\cong \mathbb{H}^2$).
We say that a quasi-Hitchin oper verifying that property is \emph{extendable}.

Recall that two Borel subalgebras $\mathfrak{b}_1, \mathfrak{b}_2$ are said to be transverse (denoted $\mathfrak{b}_1 \pitchfork \mathfrak{b}_2$), if $\mathfrak{b}_1 \cap \mathfrak{b}_2$ is a Cartan subalgebra.
The \textit{open Schubert cell} of $\mathfrak{b}$ is defined as the set $\mathcal{C}_{\mathfrak{b}} = \{ \mathfrak{b}' \in \mathcal{B}: \mathfrak{b} \pitchfork \mathfrak{b}' \}$. 
More geometrically, two apartments in $\partial \mathbb{X}$ correspond to transverse flags $\mathfrak{b}_1 \pitchfork \mathfrak{b}_2$ if and only if they can be connected by a flat asymptotic to both of them via opposite Weyl chambers.

The following lemma provides a sufficient condition for extendability:

\begin{lemma}
\label{lemma:localanswer}
 If a quasi-Fuchsian $G-$oper $(\mathcal{D},\rho)$ verifies that any point of $\mathcal{D}(\tilde{X})$ is transverse to any point in the limit set $\Lambda_\rho$, then it is extendable.
 The property of $\mathcal{D}(\tilde{X})$ being transverse to $\Lambda_\rho$ is open in the space of quasi-Hitchin opers.
   \begin{proof}
 Calling $\mathcal{Z}_{\mathfrak{b}} = \mathcal{C}_\mathfrak{b}^c$ the set of flags non-transversal to $\mathfrak{b}$, we can define a ``naive'' thickening of $\Lambda$ as $\bigcup_{x\in \partial \mathbb{H}^2} \mathcal{Z}_{\xi (x)}$. 
 Notice that this set is compact. 
 Picking $K \subset \mathbb{H}^2$ a fundamental domain for the deck action, we get that under these hypotheses $\mathcal{D}(K)$ is at a bounded distance away from the thickening as in the previous lemma.
 Repeating the continuity argument, we get that the condition is open.

 An Anosov subgroup is a convergence subgroup in the language of \cite{kapovich2017dynamics}.
 This means that whenever $\gamma_n \to \infty$ in the group $\pi_1 (X)$, there is a subsequence and a pair of flags $\xi_+,\xi_- \in \Lambda_\rho$ with the property that $\rho (\gamma_{n_k})|_{C_{\xi_-}}$ converges to the constant map $\xi_+$ uniformly on compacta. 
 Since $\mathcal{D}(K)$ is uniformly away from the naive thickening, the convergence property implies our desired extension to the boundary.
\end{proof}
\end{lemma}

Notice that under these hypotheses $\mathcal{D}(\tilde{X})$ is contained in the complement of any thickening associated to an ideal of $W$.

A consequence of Remark \ref{rmk:flagsl_3} and Theorem \ref{thm:domaindiscontinuity} is that $PGL_3 (\mathbb{C})$ opers that are EP-whitnessed verify that $\mathcal{D}(\tilde{X})$ lies in a cocompact domain of $\mathcal{B}$.
This domain arises from a balanced ideal construction.
However, there is only one possibility for a balanced ideal in this setup (see \cite[Example 3.24]{kapovich2017dynamics}).
This remark is the key ingredient of Theorem \ref{thm:extensionsl3}:

\begin{proof}[Proof of Theorem \ref{thm:extensionsl3}]
 Consider the family of $(D,\rho)$ opers that are EP-whitnessed and satisfy that $\mathcal{D}(\tilde{X})$ is transverse to $\Lambda_\rho$.
 Lemma \ref{lemma:localanswer} tells us that this family of opers is open.
 The proof follows from showing that the complement of this set is also open.
   
 Observe that if $(\mathcal{D},\rho)$ is an EP-whitnessed oper which is non-transverse to $\Lambda_\rho$, then there is $z_0 \in \tilde{X}$ such that $\mathcal{D}(z_0) \in \mathcal{Z}_{\xi(x)}$ for some $x \in \partial \mathbb{H}^2$.
 Since $\mathcal{D}(z_0)$ is outside the thickening defined by the balanced ideal, it must be in the interior of a codimension-one Schubert cell (this follows from the structure of the ideal).
In particular, it lies on the smooth locus of the Schubert variety.
 We will see in the next section (see Remark \ref{rmk:transversality}), that $d_{z_0} \mathcal{D} (TX) \oplus T_{\mathcal{D}(z)} \mathcal{Z}_{\xi (x)} = T_{\mathcal{D}(z)} \mathcal{B}$.
 Since transversality is an open property, and the limit set varies continuously with the representation, this is an open condition.
\end{proof}

\section{Remarks on transversality}
\label{sec:transversality}

The goal of this section is to prove that principal embeddings are maximally transverse, and Theorem \ref{thm:transversality}.
Before we do that, we take a slight detour into the geometry of $\mathcal{B}$.

\subsection{Codimension-one Schubert cells as divisors}

In the previous section, we denoted by $\mathcal{Z}_\mathfrak{b} = \mathcal{C}_{\mathfrak{b}}^c$ the set of flags non-transverse to $\mathfrak{b}$.
This set is a $B-$invariant algebraic subset of codimension one.
We can explicitly describe polynomials defining $\mathcal{Z}$ using fundamental weights for $\mathfrak{g}$.

Choose a Cartan subalgebra $\mathfrak{h} \subset \mathfrak{b}$. 
Recall that each fundamental weight $\omega_i$ is the highest weight of an irreducible representation $\rho_i : G \rightarrow V_{\omega_i}$ (possibly taking a double cover of the adjoint group).
Given $v_{\omega_i}$ the highest-weight vector, we can define a line bundle over $G/B$ via $L_{\omega_i} := G \times \mathbb{C} v_{\omega_i}/B \rightarrow G/B$, where $B$ acts diagonally.
The space of sections $H^0 (G/B,L_{\omega_i})$ is an irreducible $G-$representation under the pullback action isomorphic to $V_{\omega_i}^*$ (this is the easy part of the Bott-Borel-Weil Theorem, see \cite[Lemma 6.1.15]{chriss1997representation}).

The following is well-known for $PGL_n (\mathbb{C})$ (see for instance \cite[Chapter 1]{brion2005lectures}).
We include a proof since we could not find an explicit reference for the general case.

\begin{prop}
 Let $\sigma_i$ be the highest weight vector of the space of sections $H^0 (G/B,L_{\omega_i})$ as a $G-$representation, then the set of zeros $\mathcal{Z}_i$ of $\sigma_i$ is contained in $\mathcal{Z}_\mathfrak{b}$. 
 Moreover, $\mathcal{Z} = \bigcup_{i=1}^l \mathcal{Z}_i$.
\begin{proof}
 Identifying $\mathcal{B}\cong G/B$ and choosing $H \subset B$ a Cartan subalgebra, Bruhat's decomposition tells us that $G/B$ decomposes as a union of the $B-$orbits $w B \in G/B$, where $w \in W$ is the Weyl group.
 We get that each orbit $B w B$ is an open cell, and the closure of each ball $\mathcal{Z}_w$ is contained in the union of cells of higher codimension and is called a Schubert variety.
 The dimension of each cell is given by the length $\ell (w)$, i.e., the minimal number of simple reflections needed to produce $w$.
 This is all classical as in \cite[Chapter 28]{humphreys2012linear}.

 Particularly, codimension-one cells correspond to elements $w_0 s_i$, where $w_0 \in W$ is the longest element of $W$, and $s_i$ is a simple reflection along the $i-th$ coroot space.
 The proposition follows from proving that $\mathcal{Z}_i = \mathcal{Z}_{w_0 s_i}$.
 Notice that $\mathcal{Z}_i$ is closed and $B-$invariant by construction. 
Therefore, it will suffice to verify that $\sigma_i (w_0 s_j B)$ is zero for $j = i$, and non-zero otherwise.

 To see this, notice that a section of $L_{\omega_i}$ is a function $\sigma: G \rightarrow \mathbb{C}$ such that $\sigma_i(b g) = \chi_i^{-1} (b) \sigma_i(g)$, where $\chi_i : B \rightarrow \mathbb{C}$ is the character defined by $\omega_i$.
 The highest weight vector of $H^0 (\mathcal{B},L_{\omega_i})$ is defined by the function $\sigma_i (g) = \alpha_{\omega_i} (g v_{\omega_i}),$ where $\alpha_{\omega_i}$ is the highest weight vector of $V_{\omega_i}^*$.
 Notice the following two facts:
   \begin{itemize}
      \item $w v_{\omega_i}$ is a weight vector with weight $w. \omega_i$ (where the dot denotes the action of $W$ in $\mathfrak{h}^*$).
      \item $s_j . \omega_i$ is $\omega_i$ if $j \neq i$, and $\omega_i - \alpha_i$ otherwise. Here $\alpha_i$ is the simple root defining $\omega_i$.
   \end{itemize}
 From these two facts, we get that $w_0 s_j v_{\omega_i}$ is a weight vector associated to $w_0 \omega_i$ if $j\neq i$, or $w_0 \omega_i - w_0 \alpha_i$ for $j = i$.
 Since $\alpha_{\omega_i}$ pairs non-trivially only with the lowest weight vector (with weight $w_0 \omega_i$), we get that $\sigma_i (w_0 s_j) = 0$ iff $i = j$.  
\end{proof}
\end{prop}

\begin{ex}
When $G = SL_n (\mathbb{C})$, the representation associated to $\omega_i$ is $\Lambda^i \mathbb{C}^n$.
The highest weight vectors are $v_{\omega_i}= e_1\wedge\ldots\wedge e_i$, and $\alpha_{\omega_i} = dx^1\wedge\ldots \wedge dx^i$.
Therefore, we get:
\[
\sigma (gB) = \alpha_{\omega_i} (g v_{\omega_i})= \det \begin{pmatrix}
g_{11} & g_{12} & \cdots & g_{1i} \\
g_{21} & g_{22} & \cdots & g_{2i} \\
\vdots & \vdots & \ddots & \vdots \\
g_{i1} & g_{i2} & \cdots & g_{ii}
\end{pmatrix},
\]
\end{ex}

The element $w_0 s_i B$ lies in the smooth part of $\mathcal{Z}_i$.
A computation shows that the tangent space is given by:

$$T_{w_0 s_i B} \mathcal{Z}_i = \bigoplus_{\alpha \in \Sigma^+ \setminus \{-w_0 (\alpha_i)\}} \mathfrak{g}_{-\alpha}/ \mathfrak{b},$$

where $\alpha_i$ is the root defining the fundamental representation $V_{\omega_i}$.
Notice that $-w_0 (\alpha_i)$ is a simple root, therefore we get:

\begin{rmk}
\label{rmk:transversality}
 If a $G-$oper intersects a Schubert variety in its smooth locus, then the intersection is transverse.
 This is because the previous computation shows that $T_{w_0 s_i B} \mathcal{Z}_i$ is transverse to any principal direction.
\end{rmk}

\subsection{Coming back to opers}

As mentioned in the introduction, developing maps of opers with real monodromy are known to satisfy good transversality properties (see for instance \cite[Theorem 10.4.3]{labourie2018goldman}).
We include here some initial observations concerning transversality properties in the complex setup.

We start by noticing that the principal embedding satisfies the following:

\begin{prop}
\label{lemma:fuchsiantransverse}
   Let $\mathcal{P}: \mathbb{P}^1 \rightarrow \mathcal{B}$ be a principal embedding, then its image is maximally transverse.
   \begin{proof}
   As discussed in Section \ref{subsection:triplets}, the image of a principal embedding is identified with the boundary of a totally geodesic copy of $\mathbb{H}^3$ in $\mathbb{X}$, when thinking of $\mathcal{B}$ as the strata of $\partial \mathbb{X}$ defined by $h$ in a principal triplet.
   Since $\mathbb{H}^3$ is rank one, any pair of different points in $\partial \mathbb{H}^3$ can be connected by a geodesic.
   Thus, the image of $\mathcal{P}$ is a transverse set.

   To prove maximal transversality, fix an arbitrary flag $\mathfrak{b} \in \mathcal{B}$. 
   Let $\mathcal{Z}_i$ be the zero locus of the highest weight section $\sigma_i$ of $L_{\omega_i}$ associated with $\mathfrak{b}$.  
   The restriction $\sigma_i|_{\mathcal{P}(\mathbb{P}^1)}$ is a holomorphic section of the pullback bundle $\mathcal{P}^* L_{\omega_i}$ over $\mathbb{P}^1$.
   Therefore, to prove $\mathcal{Z}_i \cap \mathcal{P} (\mathbb{P}^1) \neq \emptyset$, it suffices to prove that the Chern number $c_1 (\mathcal{P}^* L_{\omega_i})$ is positive.
   It is easy to check from looking at $\sigma_i$ as an equivariant function $\sigma_i : G \rightarrow \mathbb{C}$ that $c_1 (\mathcal{P}^* L_{\omega_i}) = \omega_i (h)$, for $h$ half the sum of positive coroots.
   One can observe from Section \ref{subsection:triplets} that $\omega_i (h)$ are positive integers for all $i$.
   \end{proof}
\end{prop}

As a consequence, $G-$opers coming from quasi-Fuchsian projective structures are locally-maximally transverse.

We now verify that the oper condition automatically implies local transversality. In fact, we prove the stronger statement stated in the introduction:

\begin{proof}[Proof of Theorem \ref{thm:transversality}]
Fix $z_0 \in \tilde{X}$, let $\mathcal{D}(z_0) = \mathfrak{b}$, and identify $G/B \cong \mathcal{B}$ as the orbit of $\mathfrak{b}$. Using the setup from Section \ref{section:ds}, we lift $\mathcal{D}$ to a map $\overline{\mathcal{D}} : \tilde{X} \rightarrow G$ satisfying the differential equation $\overline{\mathcal{D}}^{-1} \overline{\mathcal{D}}'(z) = f + b(z)$, for $b(z) \in \mathfrak{b}$.

Let $M(z) = \rho_i (\overline{\mathcal{D}}(z)) \in GL (V_{\omega_i})$, where $\rho_i :G \rightarrow V_{\omega_i}$ is the irreducible representation (up to maybe lifting $\overline{\mathcal{D}}$ to a double cover of $G$).
Taking Taylor expansions of the differential equation at $z_0$, we find that the $k-th$ derivative satisfies: 

$$M^{(k)} (z_0) = (f^k + v_k(z_0)).M(z_0),$$

where $v_k(z_0)$ is a linear combination of terms of the form $f^j b_j,$ where $j < k$, and $b_j \in \mathfrak{b}$.
The dot denotes the derived action on the representation $V_{\omega_i}$.

Let $m_i$ be the height of the representation $V_{\omega_i}$ (the number of applications of $f_i \in \mathfrak{g}_{-\alpha_i}$ required to reach the lowest weight from the highest weight).
If $k < m_i$, any vector $f^k b_k v_i$ remains in a weight space of $V_{\omega_i}$ which is strictly higher than the lowest weight space. 
Since $\ker \alpha_{\omega_i}$ is the direct sum of weight spaces of $V_{\omega_i}$, except for the lowest weight, we get that:

$$\left( \frac{d}{dz} \right)^k\vline_{z = z_0} \sigma_i (\overline{\mathcal{D}(z)}) = \alpha_{\omega_i} (M^{(k)}(z_0) v_{\omega_i}) = 
\begin{cases}
   & 0 \: \text{ if }k < m_i,\\
   & v_i^* (f^{m_i} v_i), k = m_i,
\end{cases}$$

where $\sigma_i \in H^0 (G/B,L_{\omega_i})$ are the sections defining the non-transversal set $\mathcal{Z}$.

Since $f$ has a non-trivial component in every $f_i \in \mathfrak{g}_{-\alpha_i}$ (this is the oper condition), $v_i^* (f^{m_i} v_i) \neq 0$ for every $i$. This verifies that the holomorphic function $\sigma_i (\overline{\mathcal{D}(z)})$ is non-constant. 
Particularly, the zeros are isolated.
Applying this for all $i$ yields the desired result.
\end{proof}

Notice that unlike Remark \ref{rmk:transversality}, here we have to deal of the oper intersecting a Schubert variety along its singular locus.
This makes the transversality of an oper trickier to control.

\appendix

\section{Surfaces equidistant to \texorpdfstring{$\Ep$}{Ep}}
\label{sec:equidistant}

In this section, we show how an EP-surface $Ep$ can be naturally extended to a map from $\tilde{X}\times \mathbb{R}$ into the symmetric space $\mathbb{X}$ sending each slice to a surface equidistant to $Ep$. 

\subsection{Extending flat connections}

Let $E \rightarrow M$ be a $\mathfrak{g}-$bundle equipped with a flat connection $D$.
Consider $p : M \times \RR \rightarrow M$ the projection onto the first factor, and denote the pullback bundle $p^* E = \overline{E}$.
The goal of this section is to explain a way of producing connections $\overline{D}$ on $\overline{E}$ that extend $D$ and remain flat.

\begin{rmk}
The bundle $\overline{E}$ comes equipped with a differential operator $D^t$ defined as follows: let $s \in \mathcal{A}^0 (M \times \RR, \overline{E})$ given in a trivialization by a function $s: U \subset M \times \RR \rightarrow \mathfrak{g}$, define $D^t s = \partial_t s(z,t) dt$.
This definition is independent of the chosen trivialization, and the differential operator $\nabla^t$ satisfies:
$$D^t (f s) = df (dt \otimes \partial_t) s + f D^t s.$$
\end{rmk}

The following lemma will produce a flat extension of $D$ to $\overline{E}$ from a $x \in \mathfrak{g}$:

\begin{lemma}
Let $\overline{E}$ be as above, $x \in \mathfrak{g}$ and define a differential operator $D^x$ as:
$$D^x s = Ad (e^{tx}) D (Ad(e^{-tx}) s), \: \forall s \in \mathcal{A}^0 (M, \overline{E}).$$
Then the connection:
$$\overline{D^x} = D^x + D^t - dt \otimes ad (x),$$
is flat.
\begin{proof}
Flatness is local, therefore we can work in a trivialization of $E$ for which $D = D^0 + ad (A)$, where $D^0$ is the trivial connection, $A \in \mathcal{A}^0 (M,E)$.
This will induce a trivialization on $\overline{E}$, and the connection $\overline{D^x}$ can be written as $\overline{D^x} =\overline{D^0} + ad(\overline{A})$, where $\overline{D}_0$ is the trivial connection and:
$$\overline{A}_{(x,t)} = Ad(e^{tx}) A_x - dt \otimes x.$$
The curvature is given by $d\overline{A} + \frac{1}{2}[\overline{A},\overline{A}]$.
A quick computation shows:
\begin{align*}
& d\overline{A}_{(x,t)} = Ad(e^{tx}) dA_x + dt \wedge [x, Ad(e^{tx})A_x],\\
& \frac{1}{2} [\overline{A},\overline{A}]_{(x,t)} = \frac{1}{2} Ad(e^{tx})[A_x,A_x] - dt \wedge [x, Ad(e^{tx}) A_x].  
\end{align*}
Since $D$ is flat, we get $dA + \frac{1}{2}[A,A]=0$. 
Adding both equations yields the proof. 
\end{proof}
\end{lemma}

\subsection{The equidistant surfaces}

Applying the previous lemma with the vector $h$ of the principal triplet used to construct the mapping data $(P_G [\mathfrak{g}],P_B [\mathfrak{b}], D^{\vec{\alpha}})$, we obtain a mapping triple $(\overline{P_G [\mathfrak{g}]},\overline{P_B [\mathfrak{b}]}, \overline{D^{\vec{\alpha}}})$ defined over $X \times \mathbb{R}$, where the lemma gives the connection, and the bundles are given by pullbacking by $p: X \times \mathbb{R} \rightarrow X$.

A computation shows that:

$$\overline{D^{\vec{\alpha}}} = \overline{\nabla} + e^{-t} \tau \otimes ad(f) + e^t H \otimes ad(e) + \sum_{i=1}^l e^{m_i t} (\tau \alpha_i) \otimes ad (e_{m_i}) - dt \otimes ad (h),$$

where $\overline{\nabla} = \nabla + \nabla^t$, where $\nabla^t$ is the operator defined above, and $\nabla$ is the Levi-Civita connection of $H$ (acting on the $TX-$factor).
Notice that the Cartan involution $\Theta (\beta_i \otimes x_i) = *\beta_i \otimes \theta x_i$ as in \ref{eq:cartan} keeps making sense in this bundle, therefore $(\overline{P_G [\mathfrak{g}]}, \Theta, \overline{D^{\vec{\alpha}}})$ defines mapping data to the symmetric space $\mathbb{X}$.

A computation similar to Lemma \ref{lemma:maurercartan} shows:

\begin{lemma}
   The Maurer-Cartan form for the triple $(\overline{P_G [\mathfrak{g}]}, \Theta, \overline{D^{\vec{\alpha}}})$ is given by:
   $$\overline{\omega} = -2 \!\left( \cosh (t) (\tau \otimes f + H \otimes e) + \frac{1}{2} \sum_{i=1}^l e^{m_i t} \left( (\tau \alpha_i) \otimes e_{m_i} - (H * \alpha_i) \otimes \theta e_{m_i} \right) - dt \otimes h \!\right).$$
\end{lemma}

Therefore $\overline{\omega} (\partial_t) = 2 \otimes h$ is normal to the EP-surface.
A computation with the Levi-Civita connection sounds that $\overline{\nabla}^{LC}_{\partial_t} \omega (\partial_t) = 0$, therefore when we develop the triplet $(\overline{P_G [\mathfrak{g}]}, \Theta, \overline{D^{\vec{\alpha}}})$ we get a map $\overline{Ep} : \tilde{X} \times \mathbb{R} \rightarrow \mathbb{X}$ with the property that $\overline{Ep}(z,\cdot)$ is a geodesic.
The discussion of Lemma \ref{lemma:normalweyl} shows that these geodesic points towards apartments defined by the operator $\mathcal{D}$.

Flowing to the other end, this extension flows towards another curve $\tilde{X} \rightarrow \mathcal{B}$ that one can see is no longer holomorphic unless we are in the Fuchsian oper.

\section{Differential geometric computations}

\subsection{Some identities on Riemann surfaces}
\label{subsec:computations1}

Let $g$ be a Riemannian metric on a Riemann surface $X$.
Its $\mathbb{C}-$linear extension to $TX \otimes \mathbb{C}$ is 
always of the form:

\begin{equation}
\label{eq:complexifiedmetric}
g = \alpha dz^2 + \beta dz \overline{dz} + \overline{\alpha} \overline{dz^2},
\end{equation}

where $\alpha$ is complex valued, $\beta$ is real valued, and $\det_{\mathbb{C}} (g) = |\alpha|^2- (\frac{\beta}{2})^2 < 0$.
See for instance, Equation (\ref{eq:metric}) for the EP-surface situation.
Motivated by the Epstein scenario, we will assume from now on that $\alpha$ is a holomorphic quadratic differential.

Our goal is to present a coordinate-free expression for the curvature $K^g$ that will be used in Section \ref{sec:criterion}.
Extending the Levi-Civita connection $\nabla^g$ in a $\mathbb{C}-$linear way, one can verify using Koszul's formula that:

\begin{lemma}
\label{lemma:koszul}
The Christoffel symbols of a metric $g$ given by Equation (\ref{eq:complexifiedmetric}), where $\alpha$ is holomorphic are equal to:
\[
\begin{array}{ll}
\displaystyle
\Gamma^{z}_{zz}
= \frac{1}{4\det_{\mathbb{C}} (g)}\Big(2\overline{\alpha}\alpha_z - \beta \beta_z \Big),
&
\displaystyle
\Gamma^{z}_{z\bar z}
= 0,
\\[1.5em]
\displaystyle
\Gamma^{\bar z}_{zz}
= \frac{1}{4\det_{\mathbb{C}} (g)}\Big(-\beta \alpha_z + 2\alpha \beta_z \Big),
&
\displaystyle
\Gamma^{\bar z}_{z\bar z}
= 0,
\end{array}
\]
and $\Gamma_{\overline{zz}}^z = \overline{\Gamma_{zz}^{\overline{z}}}$, $\Gamma_{\overline{zz}}^{\overline{z}} = \overline{\Gamma_{zz}^z}$.
\end{lemma}

Notice that $\overline{g} = \beta dz \overline{dz}$ is compatible with the complex structure, and therefore defines a Hermitian metric on each $K^{l}$ via $\norm{\alpha}^2_{\bar{g}} = \frac{|\alpha|^2}{\beta^l}$, inducing a Chern connection $\nabla^{\bar{g}}$ satisfying $\nabla^{\bar{g}}_{\partial_z} \alpha = \partial_z \alpha - l (\log \beta)_z \alpha$.

Comparing the curvature of a metric $g$ as in Lemma \ref{lemma:koszul}, with the curvature of the metric $\overline{g}$, we get the following:

\begin{prop}
\label{coro:curvature}
Let $g$ be as in Lemma \ref{lemma:koszul}, and $\overline{g} = \beta dz \overline{dz}$. The curvatures of $g$ and $\overline{g}$ are related by the expression:

\[
K^g = \frac{K^{\bar{g}}}{1 - 4 \norm{\alpha}^2_{\bar{g}}} + \frac{4}{(1-4\norm{\alpha}^2_{\bar{g}} )^2} \norm{\nabla^{\bar{g}} \alpha}^2_{\bar{g}},
\]

where we think of $\nabla^{\overline{g}} \alpha$ as a section of $K^3$.

\begin{proof}
   Let $A = \nabla^g - \nabla^{\bar{g}}$ be the difference tensor.
   Lemma \ref{lemma:koszul} and some algebra yields:
   $$A (\partial_z) \partial_z = \frac{1}{2 \beta (\norm{\alpha}^2_{\bar{g}} - \frac{1}{4})} \left( (\langle \nabla^{\bar{g}}_{\partial_z} \alpha,\alpha\rangle \partial_z - \frac{1}{2} \nabla^{\bar{g}}_{\partial_z} \alpha (\partial_z,\partial_z) \partial_{\bar{z}}  \right)= \overline{A(\partial_{\bar{z}}) \partial_{\bar{z}}},$$
   and $A(\partial_z) \partial_{\bar{z}} = A(\partial_{\bar{z}}) \partial_z = 0$. Let the $\partial_k$ component of $A(\partial_i) \partial_j$ be denoted $A_{ij}^k$.

   Using that $g_{z\overline{z}} = \overline{g}_{z\overline{z}}$, a straightforward computation shows that:

   $$R^g_{z\overline{zz}z} - R^{\overline{g}}_{z\overline{zz}z} = (\overline{A_{zz}^{\overline{z}}})_z \alpha + \overline{A_{zz}^{\overline{z}}} \frac{\alpha_z}{2} +(\overline{A_{zz}^z})_z \frac{\beta}{2}.$$

   The identity $(\overline{A_{zz}^{\overline{z}}})\alpha = - \frac{\beta}{2} \overline{A_{zz}^z}$ and some rearranging leads to $R^g_{z\overline{zz}z} - R^{\overline{g}}_{z\overline{zz}z} = -\frac{\overline{A_{zz}^{\overline{z}}}}{2} \nabla_{\partial_z}^{\overline{g}} \alpha$. 
   This is equivalent to the desired equality.
\end{proof}
\end{prop}

Particularly, notice that the second term on the right-hand side is always positive.
We finish this section by recalling a couple of well-known identities that will be handy in Section \ref{sec:criterion}:

\begin{description}
   \item[Curvature of conformal metrics] Given $g,\overline{g}$ metrics on a surface such that $\overline{g} = f g$ for some positive function $f$, then:
   \begin{equation}
   \label{eq:Liouville}
   K^{\overline{g}} = \frac{1}{f} \left( K^g - \frac{1}{2}\Delta_g \log f \right),
   \end{equation}
   where $\Delta_g F = div(\nabla F)$ is the $g-$Laplacian.
   See for instance \cite[Chapter 7]{lee2018introduction} for a proof.
   \item[Bochner identity] Given $g = \beta dz \overline{dz}$ with curvature $K^g$, and $\alpha \in H^0 (X,K^l)$, then:
   \begin{equation}
   \label{eq:Bochner}
   \frac{1}{2} \Delta_g \norm{\alpha}^2 = 2 \norm {\nabla \alpha}^2 + l K^g \norm{\alpha}^2,
   \end{equation}
   where we think of $\nabla \alpha$ as an $l+1$ differential, and $\Delta_g = \frac{4}{\beta} \partial_z \partial_{\bar{z}}$.
   This can be checked by a direct computation. 
   Notice that the factor of $2$ in $\norm {\nabla \alpha}$ appears because we are using the ``Teichm\"uller convention'' for the norm of the quadratic differentials, instead of the one coming from Riemannian geometry (both differ by a factor of two because $g_{z\overline{z}} = \frac{\beta}{2}$).
\end{description}

\subsection{Cauchy estimates for holomorphic differentials}
\label{subsec:computations2}

The Anosov criteria we provide depend on the hyperbolic norm of $\nabla \alpha$, where $\alpha$ is a holomorphic $l-$differential, and $\nabla$ is the Chern connection induced by the hyperbolic norm $H$ in $K^l$.

Cauchy estimates allow us to upper bound $\norm{\nabla \alpha}$ in terms of the sup norm of $\alpha$ in thick regions of Teichm\"uller space.

Everything follows from the following local computation:

\begin{lemma}
   Let $\alpha$ be a holomorphic $l-$differential on $\mathbb{H}^2$, then:
   $$\norm{\nabla \alpha} (0) \leq \frac{1}{2} \coth \left( \frac{R}{2} \right) \cosh^{2l} \left( \frac{R}{2}\right) \sup_{D(0,R)} \norm{\alpha}.$$
\begin{proof}
   In the unit disk model, the line element of the hyperbolic metric is given by $\rho |dz|$, for $\rho = \frac{2}{1- |z|^2}$. 
   Notice that $\rho (0) = 2$, and $(\log \rho)_z (0) = 0$, thus $\norm{\nabla \alpha} (0) = \frac{1}{2^{l+1}} |\alpha_z| (0)$.
   
   Standard Cauchy estimates of holomorphic functions imply that $|\alpha_z (0)| \leq \frac{1}{R^{eu}} \sup_{D_{eu}(0,R^{eu})} |\alpha|$, where $D_{eu}$ denotes the Euclidean disk.
   Multiplying and dividing by $\frac{2}{1- (R^{eu})^2}$, plugging into the equation for $\norm{\nabla \alpha}$ and using $R^{eu} = \tanh \left( \frac{R}{2}\right)$, we get the desired upper bound.
\end{proof}
\end{lemma}

\begin{coro}
   Given $\alpha$ a holomorphic $l-$differential on a compact hyperbolic Riemann surface $X$, then:
   $$\sup_X \norm{\nabla \alpha} \leq C(X) \sup_X \norm{\alpha},$$
    where $C(X)$ depends on the injectivity radius $\Inj(X)$ as follows:
   $$ C(X) = 
   \begin{cases} 
   \frac{1}{2} \coth \left( \frac{\Inj(X)}{2} \right) \cosh^{2l} \left( \frac{\Inj (X)}{2}\right) & \text{if } \tanh \left( \frac{\Inj}{2} \right) \leq \frac{1}{\sqrt{2l+1}} \\
   \frac{1}{2} \sqrt{2l+1} \left( 1 + \frac{1}{2l} \right)^l & \text{otherwise.}
   \end{cases}
   $$

\begin{proof}
   As long as $R \leq \Inj (X)$, we can take a local isometry from a Riemannian ball to a hyperbolic ball and apply the previous lemma.
   The precise constant $C(X)$ follows automatically from studying the function $f(x) = \frac{1}{2} \frac{1}{x} \left(\frac{1}{1-x^2} \right)^l$.
\end{proof}
\end{coro}

Notice that when $\Inj (X) \to 0$, then $C(X) \approx \frac{1}{R}$.

\section{Construction of the Developing map}
\label{subsub:homogeneousspaces}
The goal of this section is to explain how developing data $(P_G, P_H, A)$ over a manifold $M$ gives rise to maps from the universal cover $\tilde{M}$ to $G/H$ equivariant with respect to a representation induced by $A$ (as defined in Section \ref{subsec:developing}).
The material is not new (see for instance \cite[Chapter 2]{Kobayashi1963Foundations}).

Assume first that $M$ is simply connected.
Since the bundle is flat and $M$ is simply connected, there exists a global trivialization $\Phi$ taking $A$-horizontal sections to constant sections.
This trivialization depends on the identification of the fiber $(P_G)_x$ at $x \in M$ with $G$. Parallel transport extends this identification to $\Phi$.

Without loss of generality, assume $P_H \subset P_G$. 
Notice that $\Phi ((P_H)_x) \cap \{x \} \times G$ is identified with a subset of $\{x\} \times G$ invariant under the right action of $H$, i.e., a left coset.
Projecting $M \times G$ to $M \times G/H$, we get that $\Phi (P_H)$ is the graph of a function $\Dev: M \rightarrow G/H$ which we call the \emph{developing map}.

Notice that the composition:

$$TP_H \rightarrow TP_G \xrightarrow{d \Phi} TM \times TG \rightarrow T G/H,$$

gives rise to an $H-$invariant form $\overline{\omega} \in \mathcal{A}^1 (P_H, \mathfrak{g}/\mathfrak{h})$ which descends to the differential of the developing map.
This shows that what we defined in Section \ref{subsec:developing} as the Maurer-Cartan form of the triple $(P_G, P_H,A)$ is identified with the derivative of the developing map.

Notice that the only choice we made in the construction was the trivialization $\Phi$.
Any other such trivialization will differ from the original one by an element $g \in G$, which will change $\Dev$ for the map $g \circ Dev : M \rightarrow G/H$.

If $M$ is not simply connected, one applies this discussion to $(p^* P_G,p^* P_H, p^*A)$, for $p : \tilde{M} \rightarrow M$ a choice of a universal covering map.
By construction, the developing map $\\Dev : \tilde{M} \rightarrow G/H$ will be $Deck (M) \cong \pi_1 (M)$, $\rho-$equivariant, where $\rho : \pi_1 (M) \rightarrow G$ is induced from the flat connection $A$.

\section{Root data}
\label{appendix:root}

In the following table, we compile some Lie-theoretic constants associated with complex simple Lie algebras relevant to our computations.
It contains the orbits of simple roots, together with $\norm {\Theta}^2$ and the minimal angle $\phi_\Theta$ between $h$ and the walls defined by $\Theta$.
The norms are taken with respect to the Killing form, which has the advantage that $\norm{\Theta_{L}}^2$ is always the inverse of the Coxeter number (see \cite[Chapter 22]{humphreys2012introduction}), once we know this, the Dynkin diagram encodes the value of $\norm{\Theta_{S}}$.

\begin{center}
\renewcommand{\arraystretch}{2.4}
\begin{tabular}{| >{\centering\arraybackslash}m{2.5cm} | >{\centering\arraybackslash}m{4.0cm} | >{\centering\arraybackslash}c | >{\centering\arraybackslash}c |}
\hline
\textbf{Type} & \textbf{Diagram} ($\Theta$ filled) & $ \norm{\Theta}^2$ & $\sin \phi_\Theta$ \\ \hline

$A_n$ & \dynkin{A}{*.*} & $\frac{1}{n+1}$ & $\sqrt{\frac{6}{n(n+1)(n+2)}}$ \\ \hline

$B_n$ (Long) & \dynkin{B}{*.*o} & $\frac{1}{2n-1}$ & $\sqrt{\frac{3}{n(n+1)(2n+1)}}$ \\ \hline

$B_n$ (Short) & \dynkin{B}{o.o*} & $\frac{1}{4n-2}$ & $\sqrt{\frac{6}{n(n+1)(2n+1)}}$ \\ \hline

$C_n$ (Short) & \dynkin{C}{*.**o} & $\frac{1}{2n+2}$ & $\sqrt{\frac{6}{n(2n-1)(2n+1)}}$ \\ \hline

$C_n$ (Long) & \dynkin{C}{o.oo*} & $\frac{1}{n+1}$ & $\sqrt{\frac{3}{n(2n-1)(2n+1)}}$ \\ \hline

$D_n$ & \dynkin{D}{*.***} & $\frac{1}{2n-2}$ & $\sqrt{\frac{3}{n(n-1)(2n-1)}}$ \\ \hline

$E_6$ & \dynkin{E}{6} & $\frac{1}{12}$ & $\sqrt{\frac{1}{312}}$ \\ \hline

$E_7$ & \dynkin{E}{7} & $\frac{1}{18}$ & $\sqrt{\frac{1}{798}}$ \\ \hline

$E_8$ & \dynkin{E}{8} & $\frac{1}{30}$ & $\sqrt{\frac{1}{2480}}$ \\ \hline

$F_4$ (Long) & \dynkin[edge length=0.5cm]{F}{**oo} & $\frac{1}{9}$ & $\sqrt{\frac{1}{156}}$ \\ \hline

$F_4$ (Short) & \dynkin[edge length=0.5cm]{F}{oo**} & $\frac{1}{18}$ & $\sqrt{\frac{1}{78}}$ \\ \hline

$G_2$ (Long) & \dynkin[edge length=0.5cm]{G}{o*} & $\frac{1}{4}$ & $\sqrt{\frac{1}{56}}$ \\ \hline

$G_2$ (Short) & \dynkin[edge length=0.5cm]{G}{*o} & $\frac{1}{12}$ & $\sqrt{\frac{3}{56}}$ \\ \hline

\end{tabular}
\end{center}

Here is how we compute the angle, and Davalo's constant:

\begin{description}
   \item[Computing $\phi_{min}$] Given $h_\alpha$ a simple coroot, we get that the angle $\phi_\alpha$ between $h$ and $\ker \alpha$ is given by $\sin \phi_\alpha = \frac{\kappa (h,h_\alpha)}{\sqrt{\kappa (h_\alpha,h_\alpha) \kappa (h,h)}}.$ It is well-known that given $\alpha \in \Sigma$, $\alpha (h) = ht (\alpha)$ (see \cite[Chapter 5]{serre2000complex}). 
   When $\alpha$ is simple, we get $1 = \alpha (h) = 2\frac{\kappa (h,h_\alpha)}{\kappa (h_\alpha,h_\alpha)}$.
   Thus:

   $$\sin \phi_\alpha = \frac{1}{2} \sqrt{\frac{\kappa (h_\alpha,h_\alpha)}{\kappa (h,h)}}.$$

   Notice that to minimize $\phi_\alpha$, we must pick a simple root $\alpha$ where $h_\alpha$ is as small as possible, i.e., we need to pick $\alpha$, the longest simple root. 

   \item[On Davalo's constant] Our Anosov criterion depends on the constant $c_\Theta$ defined in Theorem \ref{thm:davalosufficient}.
   To compute it, notice that given $\alpha \in \Sigma$ any root, $\alpha (h_\Theta)$ is always an integer.
   Furthermore, studying the root system one can check that there is always a root such that $\alpha (h_{\Theta_{L}}) = 1$, therefore:
   $$c_{\Theta_{S}} = \frac{m_{short}}{2 \norm{\Theta_{S}}} \geq \frac{1}{2 \norm{\Theta_{L}}} = c_{\Theta_{L}}.$$
   In fact, one can check that $m_{short} =1$, except in the $B_n$ case, where it is $2$.
\end{description}

\bibliographystyle{alpha}
\bibliography{references.bib}

\end{document}